\documentclass[12pt]{article}

\usepackage{amsmath,amssymb,amsthm,latexsym}
\usepackage[english]{babel}

\usepackage{graphicx}
\usepackage{enumitem}
\usepackage{comment}

\usepackage[autostyle]{csquotes}
\usepackage{circuitikz}
\usepackage{tikz,graphicx,subfigure}
\usetikzlibrary{hobby}
\usepackage{hyperref}
\hypersetup{
    colorlinks=true,
    linkcolor=black,
    filecolor=magenta,
    citecolor=black,
    urlcolor=cyan,
    pdftitle={Overleaf Example},
    pdfpagemode=FullScreen,
    }
\usetikzlibrary{decorations.markings}
\usetikzlibrary{arrows}

\newtheorem{theorem}{Theorem}

\newtheorem{lemma}[theorem]{Lemma}
\newtheorem{proposition}[theorem]{Proposition}
\newtheorem{corollary}[theorem]{Corollary}

\theoremstyle{definition}
\newtheorem{definition}[theorem]{Definition}
\newtheorem{rhproblem}[theorem]{RH problem}
\newtheorem{remark}[theorem]{Remark}

\allowdisplaybreaks[4]

\let\Re\undefined
\let\Im\undefined
\DeclareMathOperator{\Re}{Re}
\DeclareMathOperator{\Im}{Im}

\DeclareMathOperator{\Ai}{Ai}

\DeclareMathOperator{\Disc}{Disc}

\DeclareMathOperator*{\supp}{supp}
\DeclareMathOperator*{\area}{area}

\numberwithin{equation}{section}
\numberwithin{theorem}{section}

\def\Xint#1{\mathchoice
	{\XXint\displaystyle\textstyle{#1}}%
	{\XXint\textstyle\scriptstyle{#1}}%
	{\XXint\scriptstyle\scriptscriptstyle{#1}}%
	{\XXint\scriptscriptstyle\scriptscriptstyle{#1}}%
	\!\int}
\def\XXint#1#2#3{{\setbox0=\hbox{$#1{#2#3}{\int}$}
		\vcenter{\hbox{$#2#3$}}\kern-.5\wd0}}

\def\dashint{\Xint-}
\usepackage{authblk}
\title{This is some thing}
\author[1]{M. Kieburg\footnote{m.kieburg@unimelb.edu.au}}
\author[2]{A.B.J. Kuijlaars\footnote{arno.kuijlaars@kuleuven.be}}
\author[1,2]{S. Lahiry \footnote{sampad.lahiry@kuleuven.be, s.lahiry@unimelb.edu.au}}

\affil[1]{School of Mathematics and Statistics, University of Melbourne, 813
Swanston Street, Parkville, Melbourne, 3010, Victoria, Australia}
\affil[2]{Department of Mathematics,  Katholieke Universiteit Leuven,  Celestijnenlaan 200 B bus
2400, Leuven, 3001, Belgium}
\date{}                     %% if you don't need date to appear
\title{Orthogonal polynomials in the normal matrix model
	with two insertions}

\date{\today}

\begin{document}

\maketitle
\begin{abstract}
    We consider orthogonal polynomials with respect to the weight 
    $|z^2+a^2|^{cN}e^{-N|z|^2}$
    in the whole complex plane. We obtain strong asymptotics and the limiting
    normalized zero counting 
 measure (mother body) of the orthogonal polynomials of degree $n$ in the scaling limit $n,N\to \infty$ such that $\frac{n}{N}\to t$. We restrict ourselves to the case $a^2\geq 2c$, $cN$ integer, and $t<t^{*}$ where $t^{*}$ is a constant depending only on $a,c$. Due to this restriction, the mother body is supported on an interval. We also find the two dimensional equilibrium measure (droplet) associated with the eigenvalues in the corresponding normal matrix model. 
 
 Our method relies on the recent result that the planar orthogonal polynomials 
 are a part of
 a vector of type I multiple orthogonal polynomials, and this enables
 us to apply the steepest descent method to the associated Riemann-Hilbert problem. 
\end{abstract}
\newpage

\tableofcontents

\section{Introduction}

\subsection{Planar orthogonal polynomials}
We consider a certain class of orthogonal polynomials whose weight function is supported on the whole complex plane. Let $P_{n,N}(z)$ be the monic polynomial of degree $n$ such that
\begin{equation}\label{orthogonal}
    \int_{\mathbb C} P_{n,N}(z)\overline{P_{m,N}(z)} e^{-N\mathcal{V}(z)} dA(z)=h_{n,N}\delta_{mn} \quad n,m=0,1, \ldots
\end{equation}
Here, $e^{-N\mathcal{V}(z)}$ is our weight function, $dA(z) = dx dy$ denotes the Lebesgue measure in the complex plane, 
$h_{n,N}$ is the norm of the polynomial with respect to the considered weight and $\delta_{mn}$ denotes the Kronecker delta symbol. 
The important feature of this model is that the weight function depends on a parameter $N$ that grows proportional to $n$. 
It is assumed that $\mathcal{V}(z)$ has sufficient growth at infinity so that the
integrals in \eqref{orthogonal} are convergent.

Our motivation comes from the relation between such planar orthogonal polynomials
and the normal matrix model. One considers the random matrix ensemble distributed according to the following density on the space of $n\times n$ normal matrices
\begin{equation}\label{normal}
    \mu_{n,N}(dM) = \frac{1}{Z_{N,n}}e^{-N\mathrm{Tr}\mathcal{V}(M)}dM,
\end{equation}
where $Z_{n,N}$ is a normalizing constant such that \eqref{normal}
is a probability measure. A typical form of $\mathcal{V}$ is 
\begin{equation}\label{normal2}
	\mathcal{V}(M)= MM^{*}-V(M)-\overline{V}(M^{*})
\end{equation}
with a scalar function $V$ satisfying certain growth conditions
at infinity.  Then it is known that the eigenvalues of $M$ have the
joint distribution proportional to
\begin{equation} \label{normal3} 
	\prod_{1 \leq j < k \leq n} |z_j-z_k|^2  \prod_{j=1}^n e^{-N |z_j|^2 + 2N \Re V(z_j)}, \end{equation}
	which is a Coulomb gas in an external potential.
The orthogonal polynomial $P_{n,N}$ from \eqref{orthogonal} is the average 
characteristic polynomial and the eigenvalue distribution \eqref{normal3} is
a determinantal point process with a correlation kernel built out of the orthogonal
polynomials, see, e.g., \cite{EF05}.

In addition, as $n,N\to \infty$ with $n/N \to t > 0$, the eigenvalues of $M$ fill out a compact two-dimensional domain $\Omega$ that is commonly referred to as the droplet.
Wiegmann and Zabrodin  \cite{WZ00, Z11} showed that (see also \cite{EF05, GTV14}) 
if 
\begin{equation}\label{taylor}
    V(z)=\sum_{k=1}^{\infty}\frac{t_k}{k} z^{k}
\end{equation}
and provided the droplet $\Omega$ is connected with $0 \in \Omega$, 
then $\Omega = \Omega(t,t_1,t_2,\ldots)$ is uniquely 
characterized by its area and its exterior harmonic moments, namely
\begin{equation}\label{harmonicmoment}
 t = \frac{1}{\pi}\area(\Omega), \qquad t_{k}=-\frac{1}{\pi}\int_{\mathbb C \setminus\Omega} \frac{dA(z)}{z^k}, \quad k=1,2,3,\ldots,
\end{equation}
where for $k \leq 2$ one has to suitably regularize the integral \eqref{harmonicmoment}.
As $t$ increases, the boundary of the droplet evolves according to the model of Laplacian growth 
\cite{GTV14, HM13}. The growth model may be unstable in time, in particular
for polynomial potentials \cite{BK12, EF05, KT15}, and in some scenarios displays a topological phase transition as analyzed for example in \cite{BBLM15, BGM17,B24}. 

\subsection{Normal matrix model with insertions}
Of particular interest is the case of
\begin{equation} \label{insertion} 
	V(z) = \sum_{k=1}^r c_k \log \left(z-a_k  \right), \quad a_k \in \mathbb C, \quad c_k > 0, \end{equation}
which is known as the normal matrix model with $r$ insertions.
Here, the joint eigenvalue density \eqref{normal3} is proportional to
\begin{equation} \label{insertion2} 
	\prod_{1 \leq j < k \leq n} |z_j-z_k|^2 \cdot
		\prod_{j=1}^n \prod_{k=1}^r |z_j-a_k|^{2 c_k N} \cdot \prod_{j=1}^n e^{-N |z_j|^2}, \end{equation}
which is a Coulomb gas in a quadratic potential that also interacts with 
fixed repelling external charges of strength $2 c_k N$ 
located at $a_k$ for $k=1, \ldots, r$. Alternatively, in case $c_k N$ is
an integer for every $k$, one can view \eqref{insertion2}
as the joint eigenvalue density of a complex Ginibre matrix of size $n + \sum_{k=1}^r c_k N$
conditioned to have $c_k N$ eigenvalues at $a_k$ for every $k=1,\ldots, r$,
see \cite{BF22} for a survey on the Ginibre unitary ensemble and its variations.

\begin{figure}[!tbp]
	\centering
	\includegraphics[width=0.9\textwidth, trim = 0 1.5cm 0 1cm,clip]{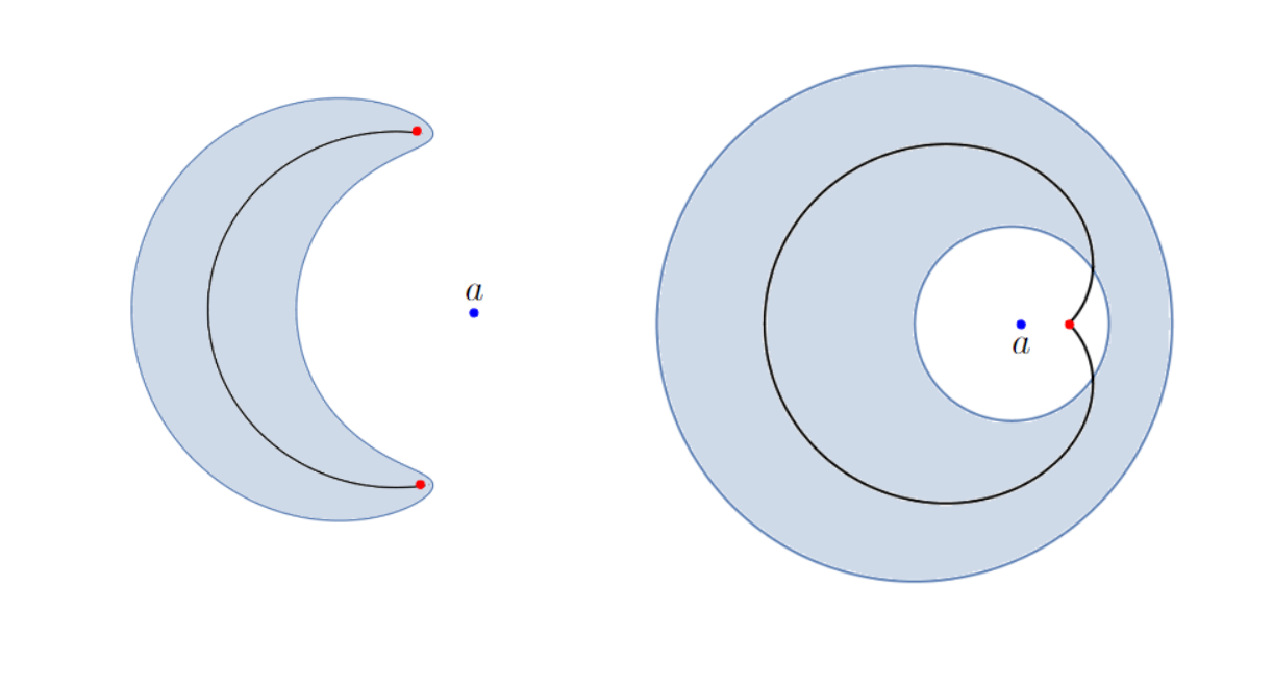}
	\caption{Droplet in shaded blue and the support of the mother body in black for the case $r=1$. The precritical case on the left and the postcritical case on the right.
	\label{oneptfig}}
\end{figure}

The case $r = 1$ in \eqref{insertion}, \eqref{insertion2} was analyzed in \cite{BBLM15}.
It was found that the planar orthogonal polynomials $P_{n,N}$ are also orthogonal
on a closed contour in the complex plane with a non-Hermitian weight function.
As a consequence, the polynomials are characterized by a $2 \times 2$ matrix valued
Riemann Hilbert problem \cite{FIK92} and the Deift Zhou method of steepest
descent \cite{DKMVZ99, DZ93} was successfully applied in \cite{BBLM15}.
One outcome is that the zeros of the polynomials $P_{n,N}$ when $n,N \to \infty$
with $n/N \to t > 0$ (with $t$ in the precritical case) 
tend to a contour inside the droplet 
with a limiting density that was explicitly calculated. The limiting zero distribution
is known as the mother body, or potential theory skeleton of the droplet, 
as it has the same logarithmic potential in the exterior region 
as the uniform measure on the droplet, while being supported on a contour,
see left panel of Figure \ref{oneptfig}. The droplet grows as $t$
increases while avoiding the external charge at $a$. 
At a critical value $t = t_c$ the droplet closes on itself and it becomes doubly
connected in the postcritical case, see right panel of Figure \ref{oneptfig}.
Following the results from \cite{BBLM15}, there are recent
works on the scaling limit of the correlation kernel in the critical 
regime \cite{KLY23}, as well
as on the large $n$ expansion of the partition function  \cite{BSY24} in the case
of one insertion.

For $r \geq 2$ the planar orthogonal polynomials $P_{n,N}$ are multiple
orthogonal of type II on a contour as shown in \cite{LY19}. This also
allows for a Riemann-Hilbert formulation \cite{VGK01}, but it is of size $(r+1) \times (r+1)$.
A steepest descent analysis was done by Lee and Yang in \cite{LY17,LY23} for
the case of small insertions, i.e., for the case where $c_k$ decays as $N$ increases
in such a way that $c_k N$ remain bounded as $n,N \to \infty$.  In that case the droplet 
is a disk (as it is
for the normal matrix model without insertions), but the zeros of the orthogonal polynomials
have an intricate non-trivial behavior within the unit disk.
The Riemann-Hilbert steepest descent analysis has not been performed for fixed $c_k$
in \eqref{insertion2}.

More recently, the planar orthogonal polynomials were found to be multiple
orthogonal of type I on a contour \cite{BKP23} as well, but only in cases 
where $c_k N$ is an integer
for every $k = 1, \ldots, r$. Type I multiple orthogonality comes with a Riemann-Hilbert problem as well; see \cite{VGK01} and  below. The weight functions in the
type I multiple orthogonality are much simpler than the ones in the type II multiple
orthogonality. As a result, we found that the type I Riemann-Hilbert problem tractable
to asymptotic analysis for the very special case of $r=2$, $a_1 = -a_2 = ia$ with $a > 0$
and $c_1 = c_2 = c > 0$ that we focus on the present work. 
Thus we consider  \eqref{normal} and \eqref{normal2}, with
\begin{equation}\label{potential}
    V(z) = c\log(z^2+a^2) \qquad a,c>0.
\end{equation}
We also assume $a^2 \geq 2c$ for reasons that will become clear shortly.
 We use $\Omega_t$ to denote the droplet.
The model, thus, depends on three positive parameters $a,c,t$.

\begin{figure}
	\centering 
\includegraphics[trim = {0 1.2cm 0 1cm},clip,width=\textwidth]{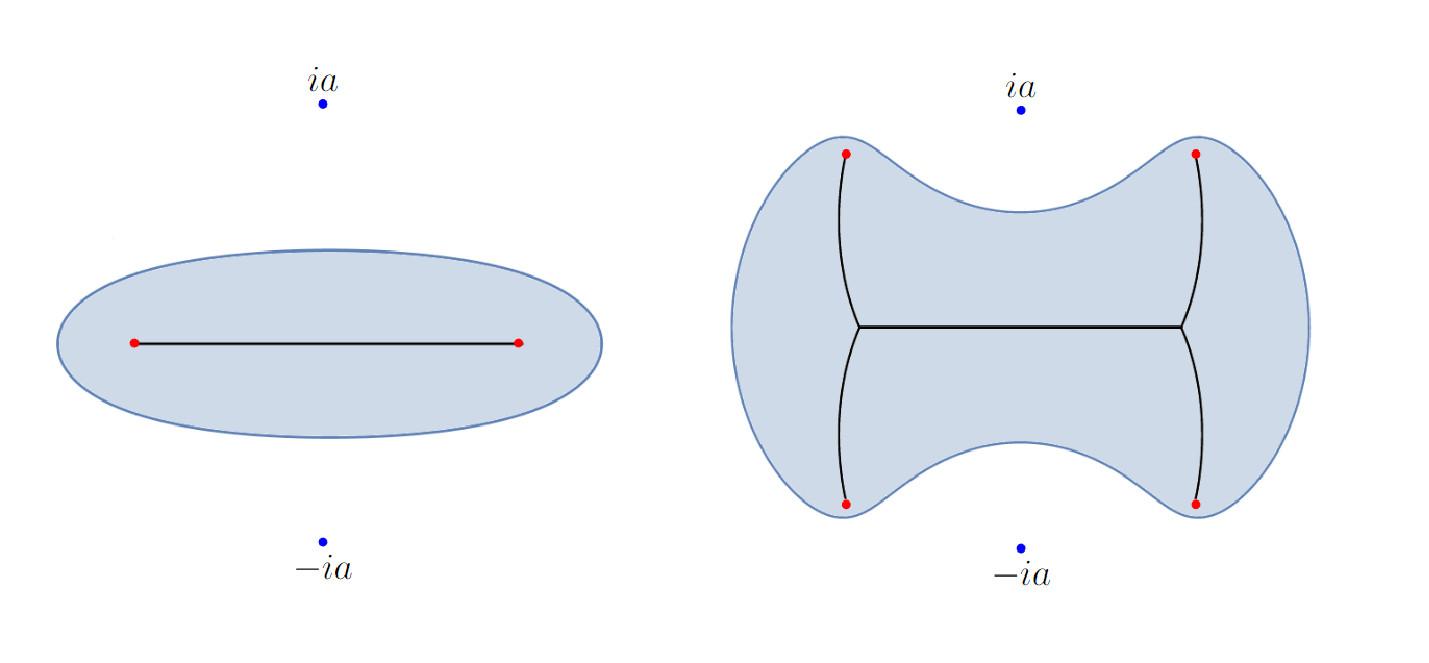}
\caption{The droplet $\Omega_{t}$ in shaded blue and the support of the mother body in black for $t<t^{*}$ on the left and $t^{*}<t<t_{c}$ on the right. \label{droplet1}}
\end{figure}

Our main new results deal with the asymptotic behavior of 
the orthogonal polynomials $P_{n,N}$ in this model. For all $t > 0$
there is a mother body supported on a contour, or system of contours, that
is the limiting distribution of the zeros of the polynomials. 
The model has two critical values,
denoted by $t^{*}$ and $t_{c}$ with 
\[  0 < t^{*} < t_{c}=a^2+2a\sqrt{c}-c. \]  
Depending on the range of $t$, with $a, c$ fixed, we observe the 
following in the large $n$ limit, 
see also Figures \ref{droplet1} and \ref{droplet2}. 

    \begin{figure}[t] 
    	\centering 
\includegraphics[width=\textwidth]{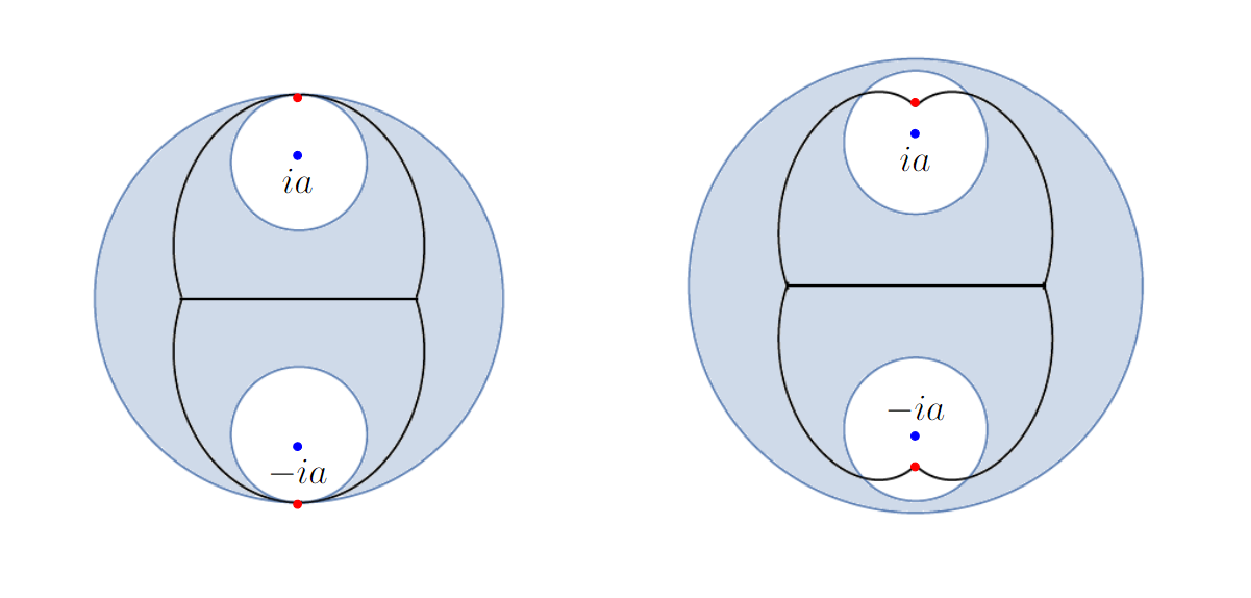}
\caption{The droplet $\Omega_{t}$ in shaded blue and the
	support of the mother body in black for $t=t_{c}$ on the left and $t>t_{c}$ on the right. \label{droplet2}}
\end{figure}
 \begin{description}
 \item[Phase 1, $0<t<t^{*}$]: 
 The support of the mother body is an interval on the real line. 
 The droplet $\Omega_{t}$ is simply connected and contains the support of the mother body, 
 see the left panel of Figure \ref{droplet1}.
 \item[Phase 2, $t^{*}<t<t_{c}$]:  
 	The support of the mother body consists of an interval on the real line, 
 	together with four analytic arcs emanating from the two endpoints of the interval.
 	The droplet $\Omega_t$ is still simply connected as in Phase 1, see the
 	right panel of Figure \ref{droplet1}.
\item[Phase 3, $t>t_{c}$]: 
	The analytic arcs in the  support of the mother body close onto themselves to form two contours connecting the two endpoints of the interval. One contour is in the
	upper half-plane and the other one is its mirror image in the lower
	half-plane. The droplet $\Omega_{t}$ is multiply connected. 
	It is equal to the closed disk with radius $\sqrt{t+2c}$ around $0$ 
	minus the two disks of radius $\sqrt{c}$ centered at $\pm ia$, see Figure \ref{droplet2}.
  \end{description}
In this paper, we restrict ourselves to Phase 1, i.e., to the regime $0 < t< t^{*}$.
We plan to discuss the other phases in a subsequent work.
 
 In the case $a^2 < 2c$ the insertions are close to the real axis (relative to
 the strength of the charges). Then, the droplet $\Omega_t$
 has two disjoint components if $t$ is small. Our methods may be adapted to handle
 this case, but we do not pursue this here. 
 
 \subsection{Type I multiple orthogonality}
 
 As already mentioned, our starting point is the type I multiple
 orthogonality that for the case $r=2$ and $a_1 = a_2 = ia$, $c_1 = c_2 = c$
 arises as follows. Let $V$ be as in \eqref{potential} with $cN$ an integer,
 so that $e^{NV(z)} = (z^2+a^2)^{cN}$ is a polynomial.
It was proved in \cite[Theorem 2.1]{BKP23} that the planar orthogonal polynomial
$P_{n,N}$ of degree $n$ with planar orthogonality
\begin{equation} \label{typeI0} 
	\int_{\mathbb C} P_{n,N}(z) \overline{z}^k e^{N \Re V(z)} e^{-N |z|^2} dA(z) = 0 
	\end{equation}
is characterized by
 
 \begin{equation} \label{typeI2}  
 	\left( \frac{d^2}{dz^2} + a^2N^2 \right)^{cN} \left[P_{n,N}(z) e^{NV(z)} \right] = \mathcal{O}(z^n) \quad \text{ as } z \to 0. \end{equation} 
 	Thus $P_{n,N} e^{NV}$ ``almost'' belongs to the kernel of the differential operator $\left(\frac{d^2}{dz^2} + a^2 N^2\right)^{cN}$.

 Following \cite{BKP23}, we obtain the existence of two polynomials $Q_{j,n,N}$,   
 with $\deg Q_{j,n,N} \leq c N-1$, for $j=1,2$, 
 such that
 \[ P_{n,N}(z) e^{NV(z)} + Q_{1,n,N}(z) e^{-iaNz} + Q_{2,n,N}(z) e^{iaNz} 
 	= \mathcal{O}\left(z^{n + 2cN}\right) \]
 as $z \to 0$.
 Equivalently, it means
 \begin{equation} \label{typeI3} 
 	P_{n,N}(z) e^{NV(z)} e^{-iaNz} + Q_{1,n,N}(z) e^{-2iaNz} = - Q_{2,n,N}(z)  +
 	 \mathcal{O}\left(z^{n + 2cN}\right) \end{equation}
 as $z \to 0$, which means that 
 the left-hand side of \eqref{typeI3} has a gap in its power series expansion around $0$:
 namely the coefficient of $z^{n+2cN-1-k}$ is zero for $k=0, \ldots, n + cN-1$.
  By Cauchy's theorem this implies that
for any contour $\gamma$ going around zero once in the positive direction,
we have 
\begin{equation} \label{typeI4}
	\frac{1}{2\pi i} \oint_{\gamma}
 	\left(P_{n,N}(z) e^{NV(z)} e^{-iaNz}
 		+ Q_{1,n,N}(z) e^{-2iaNz} \right) z^{k-n-2cN} dz  = 0,
 	\end{equation}
 	for $k=0, 1, \ldots, n+cN-1$. 
This is type I multiple orthogonality with two weight functions
$\frac{e^{NV(z)} e^{-iaNz}}{z^{n+2cN}}$
and $\frac{e^{-2iaNz}}{z^{n+2cN}}$, see e.g.\ \cite{I09}, \cite{NS91}. The polynomials $P_{n,N}$ and $Q_{1,n,N}$
are characterized by \eqref{typeI4}, and this leads to the
following RH problem, see, e.g., \cite{VGK01} and Lemma \ref{lemma12}.

\begin{rhproblem} \label{rhproblemY}
\leavevmode\newline\vspace{-0.6cm}
\begin{itemize}
  \item
		$Y : \mathbb C \setminus \Sigma_Y \to \mathbb C^{3 \times 3}$ is analytic,
		where $\Sigma_Y = \gamma$ with counterclockwise orientation.
		\item On $\Sigma_Y$, we have a jump $Y_+ = Y_- J_Y$, 
		where
		\begin{equation} \label{Yjump} 
			J_Y(z) = \begin{pmatrix} 1 & 0 &  z^{-(n+2cN)} e^{NV(z)}  e^{-ia Nz} \\
				0 & 1 & z^{-(n+2cN)} e^{-2ia Nz} \\
				0 & 0 & 1 \end{pmatrix}, \quad z \in \Sigma_Y. \end{equation}
		\item As $z \to \infty$, we have 
		\begin{equation} \label{Yasymp} 
			Y(z) = \left(I_3 + \mathcal O(z^{-1}) \right)	
			\begin{pmatrix} z^n & 0 & 0 \\ 0 & z^{cN} & 0 \\ 0 & 0 & z^{-n-cN} \end{pmatrix}.
		\end{equation}
	\end{itemize}
\end{rhproblem}

\begin{lemma} \label{lemma12} \cite{BKP23} 
	Assume that $cN \geq 1$ is an integer. Then, 
	the RH problem \ref{rhproblemY}
has a unique solution and $P_{n,N}(z) = Y_{11}(z)$.
\end{lemma}
\begin{proof}
Putting $Y_{11}(z) = P_{n,N}(z)$, 	
	 $Y_{12}(z) = Q_{1,n,N}(z)$, and
\[ Y_{13}(z) = \frac{1}{2\pi i} \oint_{\gamma}
	\left[ P_{n,N}(s) e^{NV(s)} e^{-iaNs} + Q_{1,n,N}(s) e^{-2iaNs} \right] \frac{ds}{(z-s) s^{n+2cN}}
\]
one may check that the conditions for the first row of
$Y$ are satisfied. The other rows of $Y$ can be filled with the help of 
polynomials with appropriately modified
type I multiple orthogonality properties, and slightly different degrees. 
These polynomials exist since $P_{n,N}$ and $Q_{1,n,N}$
are uniquely characterized by \eqref{typeI4} and the degree conditions
$\deg P_{n,N} = n$, $\deg Q_{1,n,N} \leq cN-1$ with $P_{n,N}$ monic,
see also \cite{BKP23, VGK01}.
\end{proof}

\section{Statement of results}

\subsection{\texorpdfstring{Measures $\mu_1$ and $\mu_2$: existence and properties}{}}
We obtain the asymptotics of the polynomials $P_{n,N}$
through a Deift-Zhou steepest descent analysis of the RH problem \ref{rhproblemY}.
Following the fundamental works \cite{BI99,DKMVZ99}, this
has become a standard tool for the asymptotics of orthogonal polynomials with 
varying orthogonality weights, as well as for multiple orthogonal polynomials. 
The equilibrium measure in external field plays a major role in \cite{DKMVZ99}.
For multiple orthogonal polynomials we need more than one equilibrium measure.

In the present work we need two measures $\mu_1$ and $\mu_2$
on the real line, whose properties will be listed in Theorem \ref{thm:mu12} below,
see Remark \ref{vectorep} for the connection with a vector equilibrium problem.

Throughout the paper, we use
\begin{equation}
U^{\mu}(z) = \int \log \frac{1}{|z-s|} d\mu(s), \qquad z \in \mathbb C.
\end{equation} 
to denote the logarithmic potential of a measure $\mu$.
We also  use the standard notation 
\begin{equation}\label{logenergy}
 I(\mu_1,\mu_2) = \iint \log \frac{1}{|z-w|} d\mu_1(z) d\mu_2(w) \end{equation}
for the mutual energy between $\mu_1$ and $\mu_2$,
and $I(\mu) = I(\mu,\mu)$ for the logarithmic
energy of $\mu$.
Our main references for
logarithmic potential theory are \cite{R12, ST97}.
\begin{theorem} \label{thm:mu12}
	Suppose $a^2 \geq 2c >0$. Then there is $t^* > 0$
such that for every $0 < t < t^*$ there exist real numbers $x_1 = x_1(t)$ and $x_2 = x_2(t)$
with $x_2 > x_1 > 0$ and 
a pair $(\mu_1, \mu_2)$ of measures on $\mathbb R$, both symmetric with respect to $0$, satisfying the following.
\begin{enumerate} [label=\rm(\arabic*)]
 
	\item \label{item1} $\mu_1$ is a measure with 
	$\supp(\mu_1) = [-x_1,x_1]$ and $\int d\mu_1 = 1$,
	that is absolutely continuous with respect
	to  Lebesgue measure with a density that is real analytic
	and positive on $(-x_1,x_1)$ and vanishes like a square
	root at the endpoints $\pm x_1$.
	
	\item \label{item2}  $\mu_2$ is a measure with 
	$\supp(\mu_2) = \mathbb R$ and $\int d\mu_2 = \frac{t+c}{t}$.
	Additionally, using  $\sigma$ to denote the measure with constant 
	density $\frac{a}{\pi t}$ on the real line, we have
	\begin{equation} \label{mu12cond3a}
		\mu_2 \leq \sigma,
		\end{equation}
	and 
	\begin{equation} \label{mu12cond3b} 
		\supp(\sigma-\mu_2) = (-\infty,-x_2] \cup [x_2, \infty).	
		\end{equation}
	The density of $\sigma- \mu_2$ is real analytic and positive 
	on $(-\infty,-x_2) \cup (x_2,\infty)$ and vanishes like a square root
	at $\pm x_2$.

\item \label{item3} There is a real constant $\ell$ such that 
	\begin{equation} \label{mu12cond4} 
		2 U^{\mu_1} - U^{\mu_2} - \Re V_1
	   \begin{cases} = \ell, & \text{ on } \supp(\mu_1) = [-x_1,x_1],
	   	\\ > \ell, & \text{ on } [-x_2,-x_1) \cup (x_1, x_2], \end{cases} 
	   		\end{equation}
	   	where 
	   	\begin{equation} \label{V1} 
	   		V_1(z) =  \frac{c}{t} \log (z^2 + a^2). \end{equation}
	   		
	  \item \label{item4}
	 \begin{equation} \label{mu12cond5} 
	 	2 U^{\mu_2} - U^{\mu_1} + \Re V_2
	 \begin{cases} = 0, & \text{ on } % \supp(\sigma-\mu_2) 
	 (-\infty, -x_2] \cup [x_2, \infty), 
	 \\
		< 0, & \text{ on } (-x_2,x_2),
	\end{cases} \end{equation}
	where
	\begin{equation} \label{V2} 
		 V_2(z) = \frac{t+2c}{t} \log z. \end{equation}
		 
	\item \label{item5} There is a smooth closed contour $\gamma$ going around $\pm ia$ 
	in the complex plane such that
	\begin{equation} \label{mu12cond6} U^{\mu_1} + U^{\mu_2} - \Re V_1 + \Re V_2 + \frac{a}{t} |\Im z| > \ell	
		\quad \text{ on } \gamma, \end{equation}
	where $\ell$ is the same constant as in \eqref{mu12cond4} and $V_1$, $V_2$ are as in \eqref{V1} and \eqref{V2}.
	The contour $\gamma$ is symmetric with respect to the real and imaginary axes
	and intersects the positive real line at a point in the interval $(x_1,x_2)$ only.
	
	\item \label{item6} $t^{*}$ is determined by the property that 
	$ \lim\limits_{t \to t^*-} x_1(t) = \lim\limits_{t \to t^*-} x_2(t)$.
\end{enumerate}
\end{theorem}

\begin{remark} \label{vectorep} 
Note that besides the parameters $a, c$, and $t$, there are three ingredients in
the properties of Theorem \ref{thm:mu12}, namely the measure $\sigma$ 
and the functions $V_1$ and $V_2$. The measure $\sigma$ 
acts as an upper constraint on the measure $\mu_2$, as discussed in item \ref{item2}. 
The functions $V_1$ and $V_2$ from \eqref{V1} and \eqref{V2} are multi-valued, 
but $\Re V_j$ and $V_j'$ for $j=1,2$ are well-defined and single valued for $j=1,2$. 
Also $e^{nV_1}$ and $e^{nV_2}$ are well-defined provided that
both $n$ and $n c/t$ are integers.
\end{remark}

The properties in \ref{item3} and \ref{item4}  are the Euler-Lagrange variational conditions for
a vector equilibrium problem associated with the energy functional
\begin{equation} \label{vectorenergy}
	I(\mu_1) + I(\mu_2) - I(\mu_1,\mu_2) - \int \Re V_1 \, d\mu_1 + \int \Re V_2 \, d\mu_2  
\end{equation}
subject to measures $\mu_1,\mu_2$ satisfying the normalizations
given in \ref{item1}, \ref{item2} and the upper constraint from \eqref{mu12cond3a}, with the additional assumption that $\supp(\mu_1) \subset [-x_2,x_2]$.
In addition there are strict inequalities in \ref{item3}  and \ref{item4}, which will be
important for the Riemann-Hilbert analysis that follows.

We also note that by property \ref{item1} the minimizer $\mu_1$ is supported on one interval,
and it may be called one-cut regular, because of the properties of the density of $\mu_1$
stated in \ref{item1}. 

The property in \ref{item5}  is an additional property of the logarithmic potentials of the two
measures that will also be used in the Riemann-Hilbert analysis to estimate 
the jump matrix on the contour $\gamma$.

\subsection{Droplet and mother body}

The next result shows that $\mu_1$ is indeed the mother body for the droplet
$\Omega_t$. 
\begin{theorem}\label{thm:droplet}
	Suppose $a^2 \geq 2c$ and $0<t<t^{*}$. Let $\mu_{1}$ be as in Theorem~\ref{thm:mu12}. 
	\begin{itemize}
	\item[\rm (a)] The identity
		\begin{equation} \label{S1zidentity}
			t\int\frac{d\mu_{1}(s)}{z-s}+\frac{2cz}{z^2+a^2}=\overline{z} 
		\end{equation}
	holds for $z$ on a simple closed curve $\partial \Omega_t$, symmetric with respect
	to the real and imaginary axes, that is the boundary of a 
	simply connected domain $\Omega_t$.   The interval $[-x_1, x_1]$
	is contained in $\Omega_t$ and the points $\pm ia$ are in 
	$\mathbb C \setminus \Omega_t$.
	
	\item[\rm (b)] We have 
	\begin{equation}\label{motherbody}
		\frac{1}{\pi} \int_{\Omega_t} \frac{dA(s)}{z-s} =  t\int\frac{d\mu_1(s)}{z-s} \quad \text{for } z \in \mathbb C \setminus \Omega_t.
	\end{equation}
	
	\item[\rm (c)] The domain has $\area(\Omega_t) = \pi t$
	and harmonic moments \eqref{harmonicmoment},
		\begin{equation} \label{wz2}
			t_{k} = - \frac{1}{\pi} \int_{\mathbb C \setminus \Omega_t} \frac{dA(z)}{z^k} = \begin{cases}
				0, &  \text{ if $k$ is odd}, \\
				(-1)^{\frac{k}{2}+1} \frac{2c}{a^{k}}, &  \text{ if $k$ is even},
			\end{cases}
		\end{equation}
		for $k=1,2,\ldots$
	
	\item[\rm (d)] $\Omega_t$ is the droplet for the normal matrix model \eqref{normal}, \eqref{normal2} with $V$ given by \eqref{potential}.		
	\end{itemize}
\end{theorem}
The droplet $\Omega_{t}$ in the normal matrix model is uniquely 
characterized by a variational problem from logarithmic potential theory
as well, see, e.g.,\ \cite{EF05}. From Theorem \ref{thm:droplet} (d)
it then follows that for some constant $\ell_t$,
\begin{equation}\label{aw1}
	2U^{\frac{1}{\pi} dA\mid_{\Omega_t}}(z) + |z|^2 - 2c \log |z^2+a^2| 
	\begin{cases}=\ell_{t} \qquad z\in \Omega_t, \\
		\geq\ell_{t} \qquad z\in \mathbb C.
	\end{cases}  
\end{equation}
After integration, we obtain from \eqref{motherbody} that
\begin{equation} \label{aw2} 
	U^{\frac{1}{\pi} dA\mid_{\Omega_t}}(z) = t U^{\mu_1}(z), \quad z \in 
\mathbb C \setminus \Omega_t, \end{equation}
that is, the logarithmic potential of $\mu_1$ agrees with 
that of the normalized Lebesgue measure on the droplet.

\begin{remark}[Remark on conformal map]
The identity \eqref{S1zidentity} means that the left-hand side is the
Schwarz function of the boundary of the droplet. Later in the paper,
we are going to  call it $S_1(z)$, see  \eqref{S1z} below. 

Denote $f$ to be the Riemann map from the exterior of the unit disk to the exterior of the droplet (which exists since $\Omega_t$ is simply connected).
The Schwarz function provides a way to analytically 
continue the conformal  map on the interior of the unit disc. 
Let $f$ be normalized such that $f(w) = \rho w + O(w^{-1})$
as $w \to \infty$, where $\rho >0$ is 
the logarithmic capacity of $\Omega_t$. 
Then defining 
\begin{equation} \label{fsymmetry} f(w) =  S_1( f(w^{-1})), \qquad |w| < 1, \end{equation}
 one can readily show, using the property \eqref{S1zidentity}
  of the Schwarz function and the symmetry in the
  real axis, that this
definition extends $f$ to a rational function on $\mathbb C$.
The rational function has poles
at the two points $w= \pm i \alpha$ on the imaginary axis that are such 
that $\pm i \alpha^{-1}$ are mapped to $\pm ia$ by $f$. Note that $0 < \alpha < 1$. 
Then $f$ takes the form
\begin{equation}\label{conformalmap}
    f(w) = \rho w + \frac{2 \kappa w}{w^2+\alpha^2},\end{equation}
for certain real constants $\rho$, $\kappa$, and $\alpha$.
See Lemma \ref{lemma56} below for equations that relate these
constants to the parameters $a, c$, and $t$ of the model. 

The Phase 1 that we are considering here (i.e., $0 < t < t^*$) corresponds
to $f$ having real critical points, that is, there are $w_1, w_2$
with $0 < w_2 < w_1$ such that $f'(\pm w_j) = 0$ for $j=1,2$.
Then the points $x_1, x_2$ from Theorem \ref{thm:mu12} are the critical values
$x_j = f(w_j)$ for $j=1,2$. 
For $t = t^*$ the critical points of $f$ coincide, 
and for larger $t$ the
critical points of $f$ are no longer real and this brings us to Phase 2.
For even larger $t$ there is no conformal map \eqref{conformalmap}
and then we are in Phase 3.
\end{remark}

\subsection{Strong asymptotics of the polynomials}

We now state our main result concerning the strong asymptotics of the orthogonal polynomials $P_{n,N}(z)$. We use the probability measure $\mu_1$
satisfying the properties of Theorem \ref{thm:mu12} and its
associated $g$-function
 \begin{equation} \label{g1z} 
 	g_1(z) = \int \log(z-s) d\mu_1(s). \end{equation}
We also use the conformal map $F$ from the exterior of the droplet to
the exterior of the unit circle, with $F(\infty) = \infty$ and $F'(\infty) > 0$. 
It is the inverse of  \eqref{conformalmap}, and it extends to an analytic
function $F : \mathbb C \setminus [-x_1,x_1] \to \mathbb C$.
Note that $F(z) = \rho^{-1}z + \mathcal{O}(z^{-1})$ as $z \to \infty$.

We then have  the following.
\begin{theorem}\label{thm:PnNasym}
    Let $0<t<t^{*}$ and let $n,N\to\infty$ such that $n-Nt$ is bounded.
    Suppose  $cN$ is an integer for every $N$. Then we have
  \begin{equation} \label{PnNasym} 
  	P_{n,N}(z) =   \left(\rho F'(z)\right)^{1/2}  e^{n g_1(z)} \left( 1 + \mathcal O(1/n) \right) 
\quad  \text{as } n \to \infty, \end{equation}
  uniformly for $z$ in compact subsets of $\overline{\mathbb C} \setminus [-x_1,x_1]$.
\end{theorem}
We believe that the condition that $cN$ is an integer is not 
essential. As already said, we need it for the solvability of the RH problem \ref{rhproblemY}, see Lemma~\ref{lemma12}, that is basic for our analysis.

The prefactor $(\rho F'(z))^{1/2}$ in \eqref{PnNasym} arises in our
approach as the $(1,1)$ entry of the global parametrix $M$ that we use
in the steepest descent analysis, see RH problem \ref{rhproblemM}, and
see \eqref{M11z} for the connection with the conformal map ($F$ is denoted
as $F_1$ in \eqref{M11z}). It turns out that $\rho F'$ has no zeros in
$\mathbb C \setminus [-x_1,x_1]$, and $(\rho F'(z))^{1/2}$ denotes the unique branch
of the square root that tends to $1$ as $z \to \infty$.

\begin{remark} In principle our methods allow us to determine the
	$\mathcal{O}\left(n^{-1}\right)$ term in \eqref{PnNasym}
	and as well as higher order terms. This can be done by improving
	the matching condition that we will obtain in the steepest descent analysis,
	similar to what is done in  \cite{DKMVZ99,KMVV04} 
	for orthogonal polynomials on the real line. We do not
	pursue this here. 
\end{remark}

\begin{remark}
The asymptotic formula \eqref{PnNasym} is known in the exterior
and on the boundary of the droplet. This was shown by Hedenmalm
and Wennman \cite{HW21} in much greater generality. What is new is 
that for our specific case the asymptotic formula extends to the complement 
of the mother body. This should be a quite general phenomenon, and
it will be most interesting to prove it in more generality.  

The asymptotic formula is stated in \cite[Theorem 1.3]{HW21} for orthonormal
polynomials. For the monic polynomials it means that (in our notation)
\begin{equation} \label{PnNasymHW} P_{n,N}(z) = (\rho F'(z))^{1/2} 
	F(z)^n e^{\frac{n}{2t} (\mathcal Q_t(z)-\mathcal Q_t(\infty))} \left( 1 + \mathcal{O}(n^{-1}) \right) \end{equation}
as $n \to \infty$ with an error term that has a full asymptotic expansion
in inverse powers of $n$.
The function $\mathcal Q_t$ in \eqref{PnNasymHW} is bounded 
and holomorphic in the complement of the droplet $\Omega_t$ (including infinity) and 
its real part equals
the potential $|z|^2 - 2c \log |z^2+a^2|$ along the 
boundary of $\partial \Omega_t$.

We claim that we can take, with the constant $\ell_t$ from \eqref{aw1},
\begin{equation} \label{defQt}
	\mathcal Q_t(z) = 2 t(g_1(z) - \log F(z)) + \ell_t. 
	\end{equation}
To see this, we note that $F^{-1} e^{g_1}$ is analytic and non-zero in
$\Omega_t$ with the limit $\rho$ at infinity. Then it has an analytic
logarithm and this is what is meant by  $g_1 - \log F$ in \eqref{defQt}.
On $\partial \Omega_t$ we have
\begin{equation} \label{ReQt} \Re \mathcal Q_t(z) = 2 t \Re g_1(z) + \ell_t = -2t U^{\mu_1}(z) + \ell_t, 
	\qquad z \in \partial \Omega_t. \end{equation}
The identity \eqref{aw2} extends by continuity to the boundary,
and combining this with \eqref{aw1} (which is also valid on $\partial \Omega_t$),
we find from \eqref{ReQt} that indeed $\Re \mathcal Q_t$ agrees
with $|z|^2 - 2c \log |z^2+a^2|$. 

Inserting \eqref{defQt} into \eqref{PnNasymHW} we find our asymptotic
formula \eqref{PnNasym}.
\end{remark}
 
As a consequence of Theorem \ref{thm:PnNasym}, we obtain the 
limiting zero counting measure of the polynomials $P_{n,N}(z)$.

\begin{theorem} \label{thm:zeros}
	Under the same assumptions as in Theorem \ref{thm:PnNasym},
	we have that all zeros of $P_{n,N}$ tend to the interval
	$[-x_1,x_1]$ as $n \to \infty$. In addition, $\mu_{1}$ is the weak limit of the normalized zero counting measures of the polynomials $P_{n,N}$.
\end{theorem}
Again, we believe that the condition that $cN$ is an integer is not necessary.

\subsection{About the proof and the organisation of the paper}
As already indicated, the starting point of our analysis
is the characterization of $P_{n,N}$ as a type I multiple
orthogonal polynomial and the resulting
RH problem \ref{rhproblemY} for $P_{n,N}$.

The Deift-Zhou method of steepest descent \cite{DKMVZ99,DZ93} requires the existence of the two measures $\mu_1$ and $\mu_2$
whose desirable properties are detailed in Theorem \ref{thm:mu12}.
We obtain these measures as minimizers of a vector
equilibrium problem from logarithmic potential theory.

 We consider the energy functional \eqref{vectorenergy} 
 acting on pairs $(\mu_1,\mu_2)$ of measures.
There are two external fields $-\Re V_1$ and $\Re V_2$ in \eqref{vectorenergy}
coming from \eqref{V1} and \eqref{V2}. 
Such vector equilibrium problems have appeared before in the analysis of 
random matrix models, see e.g.~\cite{BDK11, BK12, CK22, DK09}.
A general existence and uniqueness theorem is in \cite{HK12},
and it applies to our situation.

In the present situation the two measures are on the real line, with total
masses $\int d\mu_1 = 1$ and $\int d\mu_2 = \frac{t+c}{t}$,
as in \ref{item1} and \ref{item2} of Theorem \ref{thm:mu12}.
Thus, $\mu_1$ is
a probability measure, but $\mu_2$ is not. The second measure satisfies 
an upper constraint $\mu_2 \leq \sigma$ as in \eqref{mu12cond3a}.
Equilibrium problems with upper constraint appeared in the asymptotic
analysis of orthogonal polynomials with a discrete orthogonality \cite{BKMM07, DS97, R96},
and see \cite{ALT11, BDK11, DKRZ12} for other connections with random matrix theory,
and \cite{DM98,K00,LL83} for earlier works in integrable systems involving a minimization
problem with upper constraint. 

Given $\mu_1$, the external field $\Re V_2$ acting on $\mu_2$ is weakly admissible in the sense of \cite{S05}, see also \cite{HK12, OSW19}. As a result, the
minimizing measure $\mu_2$  will have unbounded support. 
However, the  external field $-\Re V_1$
acting on $\mu_1$ tends to $-\infty$ at infinity and therefore does not act as a confining potential on the whole real line. As a result, the minimization problem
is not well-defined when considering arbitrary probability measures $\mu_1$ on $\mathbb R$.

We resolve this issue by restricting to measures $\mu_1$ supported on  some
suitable interval $[-X,X]$, in such a way that the minimizing $\mu_1$
is supported in  $(-X,X)$. This can only be done for $t<t^{*}$ where $t^{*}$ is a constant which depends only on $a$ and $c$. 
One of our technical novelties lies in judiciously choosing the cutoff $X$. 

The analysis of the vector equilibrium problem
in Section \ref{VEPsection} will be 
a main part of the present paper. Along the way, we also establish the monotonicity property of
the measures $\mu_1$ and $\mu_2$ as a function of $t$, namely
$t\mu_1$ and $t\mu_2$ are increasing with $t$, see Lemma \ref{lemma34}(c).
As a result, we will also have that $x_1(t)$ and $x_2(t)$ are increasing for $0 < t < t^*$.

The supports of the measures $\mu_1$, $\mu_2$ give
rise to the construction of a three sheeted Riemann surface
with a meromorphic function $S$ on it as in Definition \ref{def:Schwarz} below. On the first sheet $S$
coincides with the Schwarz function \eqref{S1zidentity}.
The equilibrium measures $\mu_1$ and $\mu_2$
provide explicit formulas for its analytic continuation to the other two sheets, see \eqref{S2z}, \eqref{S3z} below.
We give the proof of Theorem \ref{thm:mu12} in  Section \ref{mainthmproofsection}.

In Section \ref{dropletproof}, we construct the droplet and 
prove Theorem \ref{thm:droplet}. Finally, in 
Section \ref{Steepestdescent}, we complete the 
Deift-Zhou steepest descent analysis, leading to 
the proofs of Theorems \ref{thm:PnNasym} and \ref{thm:zeros}.

\section{Vector equilibrium problem}\label{VEPsection}

\subsection{Setup and discussion}

We modify the energy functional \eqref{vectorenergy} by multipliying it by $t^2$
and consider $\nu_1 = t\mu_1$, $\nu_2 = t \mu_2$ instead of $\mu_1$ and $\mu_2$.
Then in view of \eqref{V1} and \eqref{V2} we consider
the energy functional (recall Definition \eqref{logenergy})
\begin{multline} \label{Etnu}
	E_t(\nu_1,\nu_2) = I(\nu_1) + I(\nu_2) - I(\nu_1,\nu_2) \\
	- c \int \log |x^2+a^2| d\nu_1(x)
	+(t+2c) \int \log |x| d\nu_2(x) \end{multline}
acting on pairs of measures $(\nu_1,\nu_2)$ with
\begin{equation} \label{nu12norms}
	\int d \nu_1 = t, \quad \int d\nu_2 = t+c, \end{equation}
 and
\begin{equation} \label{sigma0} 
	\nu_2 \leq \sigma_0, \quad d\sigma_0(x) = \frac{a}{\pi} dx. \end{equation}
We cannot minimize $E_t$ over all such pairs, since
given $\nu_2$, the minimization over $\nu_1$ is ill-defined
since the effective external field $-U^{\nu_2} - c \log |x^2+a^2|$
acting on $\nu_1$ has not enough increase at infinity. Because of \eqref{nu12norms}
it behaves like $(t-c) \log |x|$ at infinity, while it should behave at
least like $2t \log |x|$ to be (weakly) admissible for $\nu_1$ of total mass $t$. 

By contrast, given $\nu_1$ with compact support
the minimization over $\nu_2$ is well-defined as it
results in a weakly admissible equilibrium problem \cite{S05}, since the
effective external field $-U^{\nu_1} + (t+2c) \log |x|$ acting on $\nu_2$ 
behaves like $2(t+c) \log|x|$ at infinity, and $\nu_2$ is required to have total
mass $t+c$, see \eqref{nu12norms}. There is a unique minimizer,
although not necessarily with bounded support.

To have a well-defined vector equilibrium problem we introduce a cut-off $X >0$ 
and impose that 
\begin{equation} \label{nu1cutoff} 
		\supp(\nu_1) \subset [-X,X] \end{equation}
Then for every $X > 0$ there is a unique minimizer $(\nu_1,\nu_2)$ of \eqref{Etnu} 
under the conditions \eqref{nu12norms}, \eqref{sigma0} and \eqref{nu1cutoff}.
The minimizer is characterized by  the following Euler-Lagrange variational conditions
\begin{align} \label{nu1Xcond0}
	2 U^{\nu_1} - U^{\nu_2} - c \log |x^2+a^2| &
	\begin{cases} = \ell_1, &  \text{ on } \supp(\nu_1) \subset [-X,X], \\ 
		\geq \ell_1, &  \text{ on } [-X,X], 
	\end{cases} \\ \label{nu2Xcond0}
	2U^{\nu_2} - U^{\nu_1} + (t+2c) \log|x| &
	\begin{cases} = \ell_2, & \text{ on } \supp(\sigma_0-\nu_2) \cap \supp(\nu_2), \\
		\leq \ell_2, & \text{ on } \mathbb R \setminus 
		\supp(\sigma_0 - \nu_2), \\
		\geq \ell_2, & \text{ on } \mathbb R \setminus \supp(\nu_2).	 
	\end{cases}		 
\end{align}
The inequality $\leq \ell_2$ in \eqref{nu2Xcond} is a consequence of the upper
constraint \eqref{sigma0} acing on $\nu_2$.
Later, we will see that $\supp(\nu_2) = \mathbb R$ and
$\ell_2 = 0$, as will follow from Lemma~\ref{lemma34} below.

\begin{remark}[on continuity of logarithmic potentials]
	In the general theory of equilibrium measures with external fields \cite{ST97}
	the variational conditions are valid quasi-everywhere (q.e.), that is,
	with the possible exception of a polar set, see also \cite{DS97} for
	the situation with upper constraint. In the present situation, there
	are no q.e.\ exceptional sets, as can be seen as follows.
	
	First of all, the upper constraint \eqref{sigma0} implies that  $U^{\nu_2}$
	is a continuous function on $\mathbb C$. Then $\nu_1$ is the minimizer
	of an equilibrium problem on the interval $[-X,X]$ with a continuous
	external field $-U^{\nu_2}(x) - c \log |x^2+a^2|$. Then it is known
	that \eqref{nu1Xcond0} is valid without any q.e.\ exceptional set. 
	In addition $U^{\nu_1}$ is continuous as well. This in turn implies
	that \eqref{nu2Xcond0} holds everywhere on the indicated sets. 
	
	Since the minimizers of \eqref{Etnu} have continuous logarithmic
	potentials, we will typically assume this property in what follows.
\end{remark}

In what follows $a$ and $c$ are fixed and we vary $t$ and $X$.
If we want to emphasize the dependence on $t$ and $X$ we denote 
the minimizer by $\left(\nu_{1,t}^{X}, \nu_{2,t}^X\right)$.

\subsection{\texorpdfstring{Case $t=0$}{}}

At  $t=0$ the measure $\nu_1$ disappears (the total mass vanishes while the measure is non-negative). We write $\rho$ instead
of $\nu_2$  and the 
equilibrium problem \eqref{Etnu}--\eqref{nu1cutoff}, reduces to the minimization of
\begin{equation} \label{Etrho}
	E_0(\rho) =	 I(\rho) + 2c \int \log |x| d\rho(x) \end{equation}
over measures $\rho$ on $\mathbb R$ satisfying 
\begin{equation} \label{rhonorms}
	\int d \rho = c, \qquad d\rho(x) \leq \frac{a}{\pi} dx. \end{equation}

\begin{proposition} \label{prop31}
	The unique minimizer $\rho$ of \eqref{Etrho} subject to \eqref{rhonorms} is given by
	\begin{equation} \label{rhodef} 
		\frac{d\rho}{dx} = \begin{cases} 
			\frac{a}{\pi}, & \text{ on } [-\frac{c}{a}, \frac{c}{a}], \\
			\frac{a}{\pi} - \frac{\sqrt{a^2x^2-c^2}}{\pi |x|}, &
			\text{ on }  \mathbb R \setminus (-\frac{c}{a}, \frac{c}{a}). 
	\end{cases}  \end{equation}
	Its logarithmic potential $U^{\rho}$ is, then,
	\begin{equation} \label{Urho}
		U^{\rho}(x)  = \begin{cases} 
			\sqrt{c^2-a^2x^2} - c \log \left(c+ \sqrt{c^2-a^2x^2}\right) + c \log a, 
			& \text{on } [-\frac{c}{a}, \frac{c}{a}], \\
			- c \log|x|, &
			\text{on }  \mathbb R \setminus (-\frac{c}{a}, \frac{c}{a}). 
			\end{cases}  \end{equation}
	\end{proposition}
\begin{proof}
	Let $\rho$ be given by \eqref{rhodef}. It is easy to see that
	$0 < \frac{d\rho}{dx}  \leq \frac{a}{\pi}$ and, thus, $\rho$ is a measure
	satisfying $\rho \leq \sigma_0$. 
	We are going to compute its Cauchy transform.
	
	For $R > \frac{c}{a}$ and $z \in \mathbb C \setminus \mathbb R$, we have from \eqref{rhodef},
	\begin{align} \label{contour1} 
		\int_{-R}^R \frac{d \rho(s)}{z-s}  = I_1(R) + I_2(R) 
	\end{align}
	with 
	\begin{align} \label{contour2}
	I_1(R) & = \frac{a}{\pi} \int_{-R}^R \frac{ds}{z-s}, \\
	\label{contour3}
	I_2(R) & = 	- \frac{1}{\pi} \int_{\frac{c}{a}}^R \frac{\sqrt{a^2s^2-c^2}}{(z-s)s} ds
	+ \frac{1}{\pi} \int_{-R}^{-\frac{c}{a}} \frac{\sqrt{a^2s^2-c^2}}{(z-s)s} ds.
	\end{align}
	The integral in \eqref{contour2} is easy to evaluate and
	the result is
	\begin{equation} \label{contour4} 
	\lim_{R \to \infty} I_1(R) = \mp ia,
			\quad \text{ for } \pm \Im z > 0.
		\end{equation}
	
	In \eqref{contour3} we introduce the change of variables $s = \frac{1}{u}$ to
	obtain
	\begin{equation} \label{contour5}
	I_2(R) = 	- \frac{1}{\pi}
			\left( \int_{-\frac{a}{c}}^{-\frac{1}{R}} + 
				\int_{\frac{1}{R}}^{\frac{a}{c}} \right)
			\frac{\sqrt{a^2-c^2 u^2}}{(uz-1)u} du. 
		\end{equation}
	We use the partial fraction decomposition
	$\frac{1}{(uz-1)u} = \frac{z}{uz-1} - \frac{1}{u}$
	and we observe that the contribution from the $\frac{1}{u}$ term	
	will vanish as it leads to an odd integrand. The remaining
	integral has no singularity at $u=0$,
	and we take the limit $R \to \infty$ 
	\begin{equation} \label{contour6}
	\lim_{R \to \infty}	I_2(R) = - \frac{z}{\pi}
		\int_{-\frac{a}{c}}^{\frac{a}{c}} 
		\frac{\sqrt{a^2-c^2 u^2}}{uz-1} du. 
	\end{equation}
	After rescaling, the integral \eqref{contour6} is
	transformed to the
	well-known Cauchy transform of the semicircle law,
	see e.g.\ \cite{AGZ10,T12},
	\[ F_{sc}(z) = \frac{2}{\pi} \int_{-1}^1
		\frac{\sqrt{1-u^2}}{z-u} du = 2z - 2 (z^2-1)^{1/2}, \]
	namely
	\begin{equation} \label{contour7}
		\lim_{R \to \infty} I_2(R)
			= \frac{a}{2} F_{sc}\left(\frac{c}{az} \right)
			= \frac{c}{z} - \frac{(c^2-a^2 z^2)^{1/2}}{z}
		\end{equation}
	with the square root that is positive for $z \in (-\frac{c}{a},\frac{c}{a})$ and  holomorphic for
	$z \in \mathbb C \setminus ((-\infty, - \frac{c}{a}] \cup
	[\frac{c}{a},\infty))$.
	
	Combining \eqref{contour1}, \eqref{contour2}, and \eqref{contour7},
	we obtain the Cauchy transform 
	\begin{equation} \label{contour8} 
		\int \frac{d\rho(s)}{z-s} = \mp ia + \frac{c}{z} - \frac{(c^2-a^2z^2)^{1/2}}{z}
		\quad \text{ for } \pm \Im z > 0. \end{equation} 	

	The square root in \eqref{contour8} is such that 
	\begin{equation} \label{contour9}
		(c^2-a^2z^2)^{1/2} = \begin{cases} -iaz  + O(z^{-1}) & \text{ as } 
		z \to \infty \text{ with } \Im z > 0, \\
		iaz + O(z^{-1}) & \text{ as } 
		z \to \infty \text{ with } \Im z < 0. \end{cases} \end{equation} 
 	Then letting $z \to \infty$  along
	the positive imaginary axis, we conclude from 
	\eqref{contour8} and \eqref{contour9}  that $\int d\rho = c$.
	After integration we also obtain from \eqref{contour8}
	\begin{multline} \label{contour10}
		\int \log(z-s) d\rho(s) = -iaz
		- (c^2-a^2 z^2)^{1/2}  \\
		+ c \log \left( c + \sqrt{c^2-a^2 z^2} \right)
		- c \log a + c \pi i/2, \quad \Im z > 0.
	\end{multline}
	The constant of integration comes from the fact that the
	left-hand side is $ c\log z + o(1)$ as $z \to \infty$, and the
	constant $- c \log a + c \pi i/2$ ensures that this also holds for the right-hand side.
	Letting $z \to x \in \mathbb R$ and taking the real part, we 
	obtain the expressions for $U^{\rho}(x)$ as in formula \eqref{Urho}.
	
	From \eqref{Urho} we compute for $x \in (-\frac{c}{a}, \frac{c}{a})$,
	\[ \frac{d}{dx} \left(U^{\rho}(x) + c \log |x|\right)
	= \frac{\sqrt{c^2-a^2 x^2}}{x}   \]
	Thus $U^{\rho}(x) + c \log |x|$ is strictly increasing
	on the interval $(0,\frac{c}{a}]$ and strictly decreasing on $(-\frac{c}{a},0)$. 
	Since it takes the value $0$ at $x = \pm \frac{c}{a}$, we obtain
	\begin{equation} \label{rhocond}	
		U^{\rho}(x) + c \log |x| \begin{cases} = 0, & \text{ for } x \in \supp(\sigma_0 - \rho), \\
			\leq 0, & \text{ for } x \in \mathbb R, \end{cases}  \end{equation}
	where we also used the equality \eqref{Urho} on 
	$\mathbb R \setminus (-\frac{c}{a}, \frac{c}{a}) = \supp(\sigma_0 - \rho)$,
	The properties \eqref{rhocond} are the Euler Lagrange variational conditions 
	that characterize the minimizer of the constrained equilibrium problem.
	Since $\rho$ satisfies the conditions \eqref{rhonorms} we conclude
	that it is indeed the minimizer.	
	This completes the proof of Proposition \ref{prop31}.
\end{proof}
It is obvious from \eqref{rhodef} that $\rho$ is increasing as a function
of $c$.

\subsection{De la Vall\'ee Poussin Theorem}

In what follows, we are going to use several times 
a theorem of de la Vall\'ee Poussin that is stated in the book
of Saff and Totik \cite[Theorem IV.4.5]{ST97} for measures of compact support.
The assumption of compact support is not necessary.
We verified that the proof goes through  without modification
also if the measures do not have compact support. 
\begin{theorem} \label{thm:dlVP}
    Let $\mu$ and $\nu$ be two measures and let $\Omega$ be a domain
	in which both potentials $U^{\mu}$ and $U^{\nu}$ are finite and satisfy with some
	constant $c$ the inequality
	\begin{equation} \label{dlVP} U^{\mu}(z) \leq U^{\nu}(z) + c, \quad z \in \Omega. \end{equation}
	If $A$ is the subset of $\Omega$ in which equality holds in \eqref{dlVP} then 	$\nu\mid_A \leq \mu\mid_A$,
	i.e., for every Borel subset $B$ of $A$ the inequality $\nu(B) \leq \mu(B)$ holds.
\end{theorem}

As a first application of Theorem \ref{thm:dlVP} we prove the following.

\begin{lemma} \label{lemma33}
	Let $t > 0$.
	Suppose $\nu_1$ is a measure of compact support, total mass $\int d\nu_1 = t$
	and a continuous logarithmic potential.
	Let $\nu_2$ be the minimizer of \eqref{Etnu} among measures on $\mathbb R$ satisfying 
	$\int d\nu_2 = t + c$ and $\nu_2 \leq \sigma_0$. 
	Then, $\nu_2 \geq \rho$ (where $\rho$ is as in Proposition \ref{prop31})
	and $\nu_2$ satisfies the variational conditions
	\begin{equation} \label{nu2cond} 
		2 U^{\nu_2} - U^{\nu_1} + (t+2c) \log|x| 
		\begin{cases} = 0, & \text{ on } \supp(\sigma_0-\nu_2), \\
			\leq 0, & \text{ on } \mathbb R.
		\end{cases} \end{equation}
  In particular $\nu_2$ is supported on the full real line.
\end{lemma}
\begin{proof}
	The minimizer $\nu_2$ satisfies, for some $\ell_2 \in \mathbb R$, the variational conditions
	\begin{equation} \label{nu2cond0} 2 U^{\nu_2} - U^{\nu_1} + (t+2c) \log|x| 
	\begin{cases} % = \ell, & \text{ on } \supp(\nu_2) \cap \supp(\sigma_0-\nu_2), \\
		  \leq \ell_2, & \text{ on } \supp(\nu_2), \\
		  \geq \ell_2, & \text{ on } \supp(\sigma_0-\nu_2).
	\end{cases} \end{equation}
	Combining this with \eqref{rhocond}  we obtain
	\begin{equation} \label{nu2rhocond} 2 U^{\nu_2} - 2 U^{\rho} - U^{\nu_1} + t \log|x| 
	\begin{cases} \leq \ell_2, & \text{on } \supp(\nu_2)  
		\cap \supp(\sigma_0-\rho),  \\
		\geq \ell_2, & \text{on }  \supp(\sigma_0 - \nu_2). 
		\end{cases} \end{equation}
	
	The left-hand side of \eqref{nu2rhocond} is subharmonic in every domain $\Omega \subset \mathbb C$
	with $\left. \nu_2 \right|_{\Omega} \leq \left. \rho \right|_{\Omega}$.
	This is certainly the case if $\left.\nu_2 \right|_{\Omega} = 0$ or $\left. \rho \right|_{\Omega} = 
	\left. \sigma_0 \right|_{\Omega}$. Hence, the left-hand side 
	is subharmonic outside $\supp(\nu_2) \cap \supp(\sigma_0-\rho)$. It tends to
	$0$ at infinity, since 
	$2\int d\nu_2 - 2 \int d\rho - \int d\nu_1 = 2(t+c) - 2c - t = t$. 
	Then we first conclude that $\supp(\nu_2) \cap \supp(\sigma_0-\rho)$ is non-empty, since otherwise
	the left-hand side of \eqref{nu2rhocond} would
	be a subharmonic function on $\mathbb C$ which is bounded above
	and Liouville's theorem for subharmonic 
	function \cite{R95} would imply that it is identically $0$ which it is not.
	
	Then, by the maximum principle for subharmonic functions,
	the left-hand side of \eqref{nu2rhocond} 
	attains its maximum on $\supp(\nu_2) \cap \supp(\sigma_0 - \rho)$, and \eqref{nu2rhocond} yields
	\[ 2 U^{\nu_2} - 2 U^{\rho} - U^{\nu_1}
 	 + t \log |x| \leq \ell_2 \quad \text{ on } \mathbb C, \]
	with equality on $\supp(\sigma_0 - \nu_2)$.
	
	Then we apply Theorem \ref{thm:dlVP} with $\mu = 2\nu_2$ and $\nu = 2\rho + 
	\nu_1 + t \delta_0$,
	to find that $2 \nu_2 \geq 2 \rho + \nu_1 + t \delta_0$ on
	$\supp(\sigma_0-\nu_2)$, and in particular $\nu_2 \geq \rho$ on $\supp(\sigma_0-\nu_2)$.
	On $\mathbb R \setminus \supp(\sigma_0 -\nu_2)$, the measure $\nu_2$
	agrees with $\sigma_0$. Since $\rho \leq \sigma_0$, we conclude that 
	$\nu_2 \geq \rho$ on $\mathbb R \setminus \supp(\sigma_0-\nu_2)$ as well,
	and therefore $\nu_2 \geq \rho$ on the whole real line. 
	In particular $\nu_2$ has full support. Then, letting $x \to \infty$ 
	in \eqref{nu2cond0}, we obtain $\ell_2=0$,  and \eqref{nu2cond} follows.
\end{proof}

\subsection{\texorpdfstring{Monotonicity in $t$}{}}
Recall that $(\nu_{1,t}^X, \nu_{2,t}^X)$ is the minimizer of
\eqref{Etnu} under the conditions \eqref{nu12norms}, \eqref{sigma0}
and \eqref{nu1cutoff}. We use 
\begin{equation} \label{Delta12X}
	\Delta_{1,t}^X = \supp(\nu_{1,t}^X), \qquad \Delta_{2,t}^X = \supp(\sigma_0-\nu_{2,t}^X).
\end{equation}

The variational conditions \eqref{nu1Xcond0} and \eqref{nu2Xcond0}
are satisfied. From Lemma~\ref{lemma33} it follows that 
$\nu_{2,t}^X \geq \rho$, and \eqref{nu2Xcond0} will simplify 
to \eqref{nu2cond}. Thus, we
have for some constant $\ell_{t}^X$
\begin{align} \label{nu1Xcond}
	2 U^{\nu_{1,t}^X} - U^{\nu_{2,t}^X} - c \log |x^2+a^2| &
	\begin{cases} = \ell_t^X, &  \text{ on } \Delta_{1,t}^X \subset [-X,X], \\ 
		\geq \ell_t^X, &  \text{ on } [-X,X], 
	\end{cases} \\ \label{nu2Xcond}
	2U^{\nu_{2,t}^X} - U^{\nu_{1,t}^X} + (t+2c) \log|x| &
	\begin{cases} = 0, & \text{ on } \Delta_{2,t}^X, \\
		\leq 0, & \text{ on } \mathbb R.	 
	\end{cases}		 
\end{align}
Expanding on the idea used in the above proof of Lemma \ref{lemma33}, we will
be able to establish that $\nu_{1,t}^X$ and $\nu_{2,t}^X$ are increasing as a function
of $t >  0$, see part~(c) of Lemma \ref{lemma34} below.

To do so, we need to introduce more definitions and notations. 
Given $\nu_2$ with $\int d\nu_2 = t + c$, $0 \leq \nu_2 \leq \sigma_0$ we minimize $E_t$,
see \eqref{Etnu},
with respect to $\nu_1$ only, subject to $\int \nu_1 = t$ and $\supp(\nu_1) \subset [-X,X]$. The minimizer is denoted by $\nu_{1,t}^{X,\nu_2}$.
Likewise, given $\nu_1$ with $\int d\nu_1 = t$
and $\supp(\nu_1) \subset [-X,X]$, we minimize $E_t$ with respect to $\nu_2$, subject to
$\int d\nu_2 = t + c$ and $0 \leq \nu_2 \leq \sigma_0$.
The minimizer is denoted by $\nu_{2,t}^{X, \nu_1}$.

The variational conditions for these minimization
problems are with some constants $\ell = \ell_{t}^{X,\nu_2}$
\begin{align} \label{nu1Xnu2cond}
	2 U^{\nu_{1,t}^{X,\nu_2}} - U^{\nu_2} - c \log |x^2+a^2| &
	\begin{cases} = \ell, & \text{ on } \supp(\nu_{1,t}^{X,\nu_2}) \subset [-X,X], \\ 
		\geq \ell, &  \text{ on } [-X,X], 
	\end{cases} \\ \label{nu2Xnu1cond}
	2U^{\nu_{2,t}^{X,\nu_1}} - U^{\nu_1} + (t+2c) \log|x| &
	\begin{cases} = 0, & \text{ on } \supp(\sigma_0 - \nu_{2,t}^{X,\nu_1}), \\
		\leq 0, & \text{ on } \mathbb R,	 
	\end{cases}		 
\end{align}
see also Lemma \ref{lemma33}. Also by Lemma \ref{lemma33} 
we have $\nu_{2,t}^{X,\nu_1} \geq \rho$.

\begin{lemma} \label{lemma34}
	Let $t > t'>  0$ and $X > 0$.
	\begin{enumerate}
		\item[\rm (a)] Let $\nu_2'\leq \nu_2$ be measures with
		continuous logarithmic potentials, total masses $\int d \nu_2' = t' + c$, 
		$\int d\nu_2 = t+c$, and upper constraint $\nu_2 \leq \sigma_0$. 
		Then $\nu_{1,t'}^{X, \nu_2'} \leq \nu_{1,t}^{X, \nu_2}$.
		
		\item[\rm (b)] Let $\nu_1' \leq \nu_1$ be measures with 
		continuous logarithmic potentials, total masses $\int d\nu_1 = t'$ and $\int d \nu_2 = t$ and support
		$\supp(\nu_1) \subset [-X,X]$. Then
		$\nu_{2,t'}^{X, \nu_1'} \leq \nu_{2,t}^{X, \nu_1}$.
		
		\item[\rm (c)] We have $ \nu_{1,t'}^{X} \leq \nu_{1,t}^{X}$
		and $\nu_{2,t'}^{X} \leq \nu_{2,t}^{X}$. 
	\end{enumerate}
\end{lemma}
\begin{proof}	
	(a) The variational conditions are \eqref{nu1Xnu2cond} and
	\begin{equation} \label{nu1pXnu2cond} 
		2U^{\nu_{1,t'}^{X,\nu_2'}} - U^{\nu_2'} - c \log(x^2+a^2) 
		\begin{cases}
			= \ell' & \text{ on } \supp(\nu_{1,t'}^{X,\nu_2'})
			\subset [-X,X], \\
			\geq \ell' & \text{ on } [-X,X], \end{cases} \end{equation}
	for some $\ell'$. Then by \eqref{nu1Xnu2cond} and \eqref{nu1pXnu2cond}, where we write $\nu_1 = \nu_{1,t}^{X,\nu_2}$ and $\nu_1'= \nu_{1,t'}^{X,\nu_2'}$,
	\begin{equation} \label{Unu1pminUnu1} 2 U^{\nu_1} - 2 U^{\nu_1'} - U^{\nu_2-\nu_2'}  \begin{cases}
			\geq \ell - \ell' & \text{ on } \supp(\nu_1'), \\
			\leq \ell - \ell' & \text{ on } \supp(\nu_1). \end{cases} \end{equation}
	The left-hand side of \eqref{Unu1pminUnu1} is subharmonic in 
	$\mathbb C \setminus \supp(\nu_1)$ (we use  $\nu_2'\leq\nu_2$).
	At infinity, it tends to $-\infty$. Hence, the maximum is assumed at $\supp(\nu_1)$, where it is bounded
	by $\ell-\ell'$ by the second inequality in \eqref{Unu1pminUnu1}. This implies
	\[ 2 U^{\nu_1} - 2 U^{\nu_1'} - U^{\nu_2-\nu_2'}  \leq \ell - \ell' \text{ on } \mathbb C, \]
	and equality holds on $\supp(\nu_1')$ because of the first inequality in \eqref{Unu1pminUnu1}. 
	Therefore,   $2\nu_1 \geq 2 \nu_1' + \nu_2-\nu_2'$ by \eqref{thm:dlVP} on $\supp(\nu_1')$,
	and in particular $\nu_1' \leq \nu_1$ on $\supp(\nu_1')$. Outside $\supp(\nu_1')$ we have $\nu_1 \geq 0 = \nu_1'$, and it follows that $\nu_1' \leq \nu_1$ everywhere.
	
	\medskip
	
	(b) 
	Both $\nu_{2,t}^{X,\nu_1}$ and $\nu_{2,t'}^{X,\nu_1'}$ 
	have full supports by Lemma \ref{lemma33}.
	We also have variational conditions \eqref{nu2Xnu1cond} and
	similarly
	\begin{equation} \label{nu2pXnu1cond} 
		2U^{\nu_{2,t'}^{X,\nu_1'}} - U^{\nu_1'} + (t'+2c) \log |x| \begin{cases} = 0, &  \text{ on } \supp(\sigma_0-\nu_{2,t'}^{X,\nu_1'}), \\
			\leq 0, & \text{ on } \mathbb R. \end{cases} \end{equation}
	Writing $\nu_2 = \nu_2^{t,X,\nu_1}$ and 
	$\nu_2' = \nu_2^{t',X,\nu_1'}$ and
	combining \eqref{nu2Xnu1cond} and \eqref{nu2pXnu1cond}, we get
	\begin{equation} \label{Unu2pminUnu2} 
		2 U^{\nu_2} - 2 U^{\nu_2'} - U^{\nu_1-\nu_1'} - (t-t') U^{\delta_0}
		\begin{cases} \leq 0, & \text{ on } \supp(\sigma_0-\nu_2'), \\
			\geq 0, & \text{ on } \supp(\sigma_0-\nu_2). \end{cases} \end{equation}
	The left-hand side of \eqref{Unu2pminUnu2} is subharmonic in $\mathbb C \setminus \supp(\sigma_0-\nu_2')$ 
	since $t'< t$, $\nu_1' \leq \nu_1$ (everywhere) and $\nu_2 \leq \nu_2'$ in $\mathbb C \setminus \supp(\sigma_0-\nu_2')$. 
	Hence, the maximum of the left-hand side of \eqref{Unu2pminUnu2} is attained at the boundary, 
	which now consists of $\supp(\sigma_0-\nu_2')$
	and the point at infinity. There is a limit $0$ at infinity, because the total
	masses of the involved measures add up to $0$. On $\supp(\sigma_0-\nu_2')$ the values are $\leq 0$
	due to the first inequality in \eqref{Unu2pminUnu2}.
	This means
	\[  2 U^{\nu_2} - 2 U^{\nu_2'} - U^{\nu_1-\nu_1'} - (t-t') U^{\delta_0} \leq 0 \]
	and equality holds on $\supp(\sigma_0 - \nu_2)$ because of the second inequality
	in \eqref{Unu2pminUnu2}.
	Theorem \ref{thm:dlVP} tells us that $2\nu_2 \geq 2\nu_2' + \nu_1-\nu_1' + (t-t') \delta_0$
	on $\supp(\sigma_0-\nu_2)$, and, consequentially, $\nu_2 \geq \nu_2'$, 
	on $\supp(\sigma_0-\nu_2)$.
	Outside $\supp(\sigma_0-\nu_2)$ we have $\nu_2 = \sigma_0 \geq \nu_2'$. 
	Thus, $\nu_2' \leq \nu_2$ everywhere as claimed in part (b). 
	
	\medskip
	
	(c)
	Let $(\widetilde{\nu}_1, \widetilde{\nu}_2)$ be the minimizer of $E_{t}$ 
	with the extra 
	conditions $\widetilde{\nu}_1 \geq \nu_{1,t'}^X$ and $\widetilde{\nu}_2 \geq \nu_{2,t'}^X$,
	in addition to the usual conditions \eqref{nu12norms} and \eqref{sigma0} associated with $t$
	and $X$.
	
	Because of strict convexity of the energy functional, cf.~\cite{HK12}, 
	and convexity of its domain, there is a unique
	minimizer. Then, $\widetilde{\nu}_1$ is the unique minimizer
	of $ E_t( \cdot, \widetilde{\nu}_2)$ under the conditions
	$\int \widetilde{\nu}_1 = t$, $\supp(\widetilde{\nu}_1)
	\subset [-X,X]$, and $\widetilde{\nu}_1 \geq 
	\nu_{1,t'}^X$. 
	However, due to part (b) the condition $\widetilde{\nu}_1 \geq \nu_{1,t'}^X$
	is redundant. If we do not include this condition a priori, then the
	minimizer is still going to satisfy it, since $\widetilde{\nu}_2 \geq \nu_{2,t'}^X$. This leads to the conclusion
	\begin{equation} \label{nu1tilde} \widetilde{\nu}_1 =  \nu_{1,t}^{X,\widetilde{\nu}_2}. \end{equation}
	
	Similarly, due to part (a) $\widetilde{\nu}_2$ is the unique minimizer
	of $E_t(\widetilde{\nu}_1, \cdot)$ under the conditions 
	$\int \widetilde{\nu}_2 = t + c$, $\widetilde{\nu}_2 \leq \sigma_0$, i.e.,
	\begin{equation} \label{nu2tilde} \widetilde{\nu}_2 =  
		\nu_{2,t}^{X,\widetilde{\nu}_1}. \end{equation}
	This implies that $(\widetilde{\nu}_1, \widetilde{\nu}_2) = (\nu_{1,t}^{X}, \nu_{2,t}^X)$
	and part (c) of the lemma follows.
\end{proof}

\subsection{Result on equilibrium measures}\label{eqmeasure}

\subsubsection{Statement}

We continue with the study of the measures $\nu_{1,t}^X$ and $\nu_{2,t}^X$.
We are interested in situations where $X$ is not in the support of $\nu_{1,t}^X$. 
Recall that $\Delta_{1,t}^X$ and $\Delta_{2,t}^X$ are given by
\eqref{Delta12X}.

In this subsection we are going to prove the following result.

\begin{proposition} \label{prop35}
	Suppose $t > 0$ and $X > 0$ are such that $\Delta_{1,t}^X \subset (-X,X)$ and $\Delta_{2,t}^X \subset \mathbb R \setminus (-X,X)$. Then the following hold.
	\begin{enumerate}
		\item[\rm (a)]
		$\Delta_{2,t}^X$ is the union of two intervals 
		$(-\infty, -x_2] \cup [x_2, \infty)$ for some $x_2 \geq X$. 
		\item[\rm (b)]
		$\Delta_{1,t}^X$ is either an interval $[-x_1,x_1]$ with $0 <x_1 < X$,
		or a union of two intervals $[-x_1,-x_0] \cup [x_0, x_1]$ with
		$0 < x_0 < x_1 < X$.
		\item[\rm (c)] 	The measure $\nu_{1,t}^X$ has a density 
		that is real analytic and positive on the interior of $\Delta_{1,t}^X$
		except possibly at $0$,
		and vanishes like 	a square root at the endpoints of $\Delta_{1,t}^X$.
		\item[\rm (d)] 
		The measure $\nu_{2,t}^X$ satisfies $\nu_{2,t}^X \geq \rho$ and
		$\sigma_0 - \nu_{2,t}^X$ has a density that is real
		analytic and positive on the interior of $\Delta_{2,t}^X$ and vanishes
		like a square root at the endpoints of $\Delta_{2,t}^X$.
		\item[\rm (e)] 
		The variational inequality in \eqref{nu1Xcond} is strict outside of $\Delta_{1,t}^X$ 
		and  extends as a strict inequality up to $\pm x_2$, that is,
		\begin{align} \label{Unu1strict}
			2 U^{\nu_{1,t}^X} - U^{\nu_{2,t}^X} - c \log |x^2+a^2| 
			> \ell_t^X,  \quad \text{ on } [-x_2,x_2] \setminus \Delta_{1,t}^X, 
			\end{align}
		\item[\rm (f)] 
		The variational inequality in \eqref{nu2Xcond} is strict outside of $\Delta_{2,t}^X$, 	that is,
		\begin{align}
			2U^{\nu_{2,t}^X} - U^{\nu_{1,t}^X} + (t+2c) \log|x| 
			< 0,  \quad \text{ on } (-x_2,x_2).	
			\label{Unu2strict} 		 
		\end{align}
	\end{enumerate}
\end{proposition}

The condition $a^2 \geq 2c$ will be used to guarantee that $\Delta_{1,t}^X$
is an interval, since otherwise $\Delta_{1,t}^X$ may be the union of two intervals. This is done in Lemma \ref{lemma42}.

\subsubsection{Proof of Proposition \ref{prop35} (a)}

\begin{proof}
	We write $\nu_j = \nu_{j,t}^X$ for $j=1,2$.
	
	Take $R > 2X$ and let $\sigma_{0,R}$ and $\nu_{2,R}$ 
	be the restrictions of $\sigma_0$ and $\nu_2$ to $[-R,R]$.
	Then by \eqref{nu2Xnu1cond}
	\begin{multline} 2 U^{\sigma_{0,R}-\nu_{2,R}} 
		+ U^{\nu_1} - (t+2c) \log |x| - 2 U^{\sigma_{0,R}} - 2 U^{\nu_2 - \nu_{2,R}} \\ 
		\begin{cases} = 0, & \text{ on } \supp(\sigma_0-\nu_2) \\
			\geq 0, & \text{ on } \mathbb R.
			\end{cases} \end{multline}
Observe that  $\sigma_{0,R} - \nu_{2,R} \geq 0$ and $\supp(\sigma_{0,R} - \nu_{2,R}) =
[-R,R] \cap \supp(\sigma_0-\nu_2)$. Then 
$\sigma_{0,R} -\nu_{2,R}$ is the equilibrium measure corresponding to the external field
\begin{equation} \label{VR} V_R := U^{\nu_1} - (t+2c) \log |x| - 2 U^{\sigma_{0,R}} - 2 U^{\nu_2 - \nu_{2,R}} \end{equation}
for measures on $[-R,R]$ with finite total mass equal to $\int d(\sigma_{0,R} - \nu_{2,R})$.

The first term on the right-hand side of \eqref{VR} is convex
on $[X,\infty)$ since $\supp(\nu_1) \subset [-X,X]$. The second
term is clearly convex on $(0,\infty)$. The third term is convex on $[-R,R]$,
which can be checked by direct calculation
\begin{align*} - 2 U^{\sigma_0,R}(x) & = \frac{2a}{\pi} \int_{-R}^R \log |x-s| ds \\
	& = \frac{2a}{\pi} \left[ (R-x) \log (R-x) + (R+x) \log(R+x) -2R \right] 
	\end{align*}
and
\begin{equation} \label{convex3} 
	\frac{d^2}{dx^2} \left(- 2 U^{\sigma_0,R}(x) \right)	
	= \frac{4aR}{\pi(R^2-x^2)} \geq \frac{4a}{\pi R},
	\quad \text{ for } x \in [-R,R]. \end{equation}
The fourth term on the right-hand side of \eqref{VR} is concave on $[-R,R]$.
We calculate its second derivative for $x \in [-R,R]$, using also the
symmetry (about the origin) of $\nu_2-\nu_{2,R}$,
\begin{align*} \frac{d^2}{dx^2} \left(- 2 U^{\nu_2 - \nu_{2,R}}(x) \right)
	& = \frac{d^2}{dx^2} \left( 2 \int \log |x-s| d(\nu_2-\nu_{2,R})(s) \right) \\
	& =  -2 \int \frac{1}{(x-s)^2} d(\nu_2-\nu_{2,R})(s) \\
	& = -4 \int_{R}^{\infty} \left( \frac{1}{(x-s)^2} + 
	\frac{1}{(x+s)^2} \right) d(\nu_2-\nu_{2,R})(s).
\end{align*}
The integrand is decreasing for $s \in [R,\infty)$ and we obtain
the lower estimates
\begin{align*} \frac{d^2}{dx^2} \left(- 2 U^{\nu_2 - \nu_{2,R}}(x) \right)
	\geq   - \frac{8(R^2+x^2)}{(R^2-x^2)^2} (t+c) \quad
	\text{ for } x \in [-R,R].
	\end{align*}
and
\begin{align} \label{convex4} \frac{d^2}{dx^2} \left(- 2 U^{\nu_2 - \nu_{2,R}}(x) \right)
	\geq   - \frac{54}{R^2} (t+c) \quad
	\text{ for } x \in [-\tfrac{R}{\sqrt{2}},\tfrac{R}{\sqrt{2}}],
\end{align}
From \eqref{convex3} and \eqref{convex4} it follows that for $R$ large
enough $-2 U^{\sigma_{0,R}} - 2 U^{\nu_2-\nu_{2,R}}$ is convex
on $[-\tfrac{R}{\sqrt{2}},\tfrac{R}{\sqrt{2}}]$.

Then the external field \eqref{VR} is convex on $[X, \tfrac{R}{\sqrt{2}}]$. 
By a standard result from 
equilibrium problems with external fields, this implies 
that $\supp(\sigma_{0,R} - \nu_{2,R} ) \cap [X, \tfrac{R}{\sqrt{2}}]$
is an interval (maybe empty) for every large enough $R$. 
Letting $R \to \infty$, we conclude that
$\supp(\sigma_0-\nu_2) \cap [X, \infty)$ is an interval.
Since $\sigma_0-\nu_2$ has infinite mass in $[X,\infty)$, 
it has to be an interval of the form $[x_2,\infty)$ for some $x_2 \geq X$. 
By symmetry 
$\supp(\sigma_0-\nu_2) \cap (-\infty,-X] = (-\infty,-x_2]$, and by assumption of Proposition \ref{prop35} we have $\supp(\sigma_0 - \nu_2) = \Delta_{2,t}^X \subset	
\mathbb R \setminus (-X,X)$. Then part (a) of Proposition \ref{prop35} follows.
	\end{proof}

\subsubsection{\texorpdfstring{Functions $S_j$ and spectral curve}{}}\label{Schwarzsection}

The proofs of parts (b)-(f) of Proposition \ref{prop35} require more preliminary work.
	The measure $\nu_{1,t}^X$ minimizes (recall Definition \ref{logenergy}) 
	\[ I(\nu_1) - I(\nu_1, \nu_{2,t}^X) - c \int \log (x^2+a^2) d\nu_1(x) \]	
	among all measures $\nu_1$ on $[-X,X]$ with total mass $t$.
	This is an equilibrium problem with an effective external field
	$ - U^{\nu_{2,t}^X} - c \log |x^2+a^2|$ acting on $\nu_1$
	that is real analytic on $(-X,X)$ because of the assumption that
	$\Delta_{2,t}^X \subset \mathbb R \setminus (-X,X)$.
	Since also $\Delta_{1,t}^X \subset (-X,X)$, it then 
	follows from \cite{DKM98} that $\Delta_{1,t}^X$ consists
	of a finite
	union of intervals. Also $\nu_{1,t}^X$ has a bounded density with respect
	to Lebesgue measures, that is real analytic on the interior of
	each interval in its support, and vanishes at least
	as a square root at endpoints. We also denote $x_{1}$ to be the rightmost end point of $\Delta_{1,t}^X$.

	For simplicity of notation we write $\nu_1$ and $\nu_2$
	instead of $\nu_{1,t}^X$ and $\nu_{2,t}^X$, and we consider
	$t$ and $X$ fixed. We also write $\Delta_1$ and $\Delta_2$
	instead of $\Delta_{1,t}^X$ and $\Delta_{2,t}^X$.
	If we are going to change $t$ and $X$, then
	we will use the more elaborate notation again.

	Based on $\nu_1$ and $\nu_2$ we define three functions which play a significant role in the rest of the analysis.
\begin{definition} \label{def:Schwarz}
	We define 
	\begin{align} \label{S1z} 
		S_1(z) & =   \int \frac{d\nu_1(s)}{z-s} + \frac{2cz}{z^2+a^2},
		&&\text{for } z \in \mathbb C \setminus \Delta_1, \\ \label{S2z}
		S_2(z) & = -  \int \frac{d\nu_1(s)}{z-s} +  \int \frac{d\nu_2(s)}{z-s} \pm ia, 
		&& \text{for } \pm \Im z > 0, \\ \label{S3z}
		S_3(z) & = -  \int \frac{d\nu_2(s)}{z-s} + \frac{t+2c}{z}  \mp ia,
		&& \text{for } \pm \Im z > 0.
	\end{align}
\end{definition}
We will later see that $S_{1}(z)$ plays the role of the  Schwarz function for the droplet $\Omega_t$ defined in Theorem \ref{thm:droplet} and $S_2(z)$ and $S_{3}(z)$ are it's analytic continuation on Riemann sheets. Moreover we will be able to recover the measures $\nu_1$ and $\nu_2$ from the jumps of $S_i(z)$ on branch cuts.

The function $S_1$ is meromorphic with simple
poles at $\pm ia$ of residue $c$, while $S_2$ and $S_3$
are defined and analytic on $\mathbb C \setminus \mathbb R$. All functions are odd, i.e., $S_j(-z) = - S_j(z)$ for $j=1,2,3$, since the measures are 
symmetric with respect to $0$.

The functions $S_2$ and $S_3$ are the analytic continuations
of $S_1$ to the other sheets of a three-sheeted Riemann surface,
as we will show next. We write
\begin{equation} \label{R123}
	\mathcal R_1: \ \mathbb C \setminus \Delta_1, \quad
	\mathcal R_2: \ \mathbb C \setminus (\Delta_1 \cup \Delta_2), \quad
	\mathcal R_3: \  \mathbb C \setminus \Delta_2, 
\end{equation}
with $\mathcal R_j$ glued to $\mathcal R_{j+1}$ along $\Delta_j$ for $j=1,2$,
in the usual crosswise manner, see Figure~\ref{fig:three-sheets}
for the case where $\Delta_1$ consists of one interval. 
We also add three points at infinity
to obtain the compact Riemann surface $\mathcal R$.
The Riemann surface has branch points at $\pm x_1$ and $\pm x_2$, 
but there is no branching at infinity. It has genus $g$ if $g+1$ is
the number of intervals in $\Delta_1$.

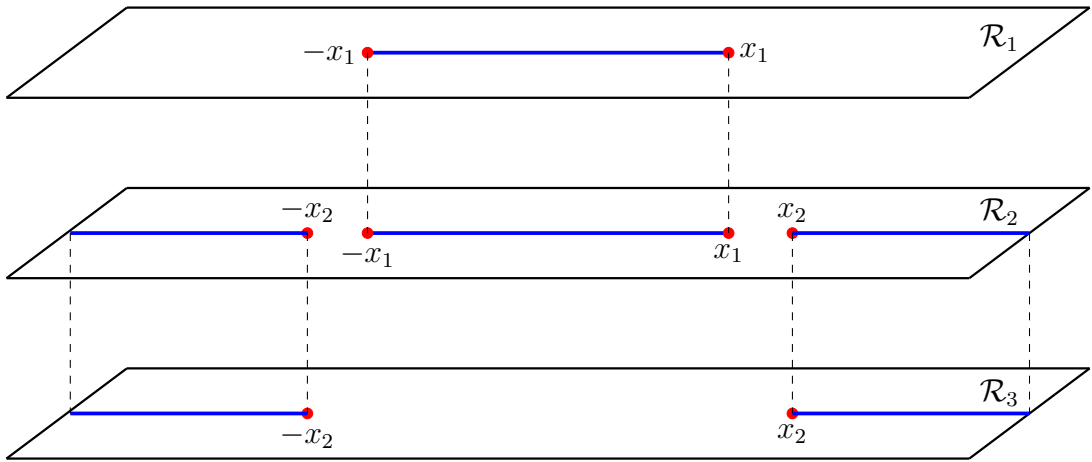
\begin{figure}
	\begin{tikzpicture}[scale=0.8]
		\draw[line width=0.3mm] (0,0) -- (16,0);
		\draw[line width=0.3mm] (0,0)--(2,1.5);
		\draw[line width=0.3mm] (2,1.5)--(18,1.5);
		\draw[line width=0.3mm](16,0)--(18,1.5);
		\draw[line width=0.3mm] (0,3) -- (16,3);
		\draw[line width=0.3mm] (0,3)--(2,4.5);
		\draw[line width=0.3mm] (2,4.5)--(18,4.5);
		\draw[line width=0.3mm](16,3)--(18,4.5);
		\draw[line width=0.3mm] (0,-3) -- (16,-3);
		\draw[line width=0.3mm] (0,-3)--(2,-1.5);
		\draw[line width=0.3mm] (2,-1.5)--(18,-1.5);
		\draw[line width=0.3mm](16,-3)--(18,-1.5);
		\draw[line width=0.5mm, color=blue](6,0.75)--(12,0.75);
		\draw[line width=0.5mm, color=blue](6,3.75)--(12,3.75);
		\filldraw[red] (12,3.75) circle (2.5 pt)  node[right,black]{$x_1$};
		\filldraw[red] (12,.75) circle (2.5 pt)  node[below,black]{$x_1$};
		\filldraw[red] (6,3.75) circle (2.5 pt)  node[left,black]{$-x_1$};
		\filldraw[red] (6,0.75) circle (2.5 pt)  node[below,black]{$-x_1$};
		\filldraw[red] (5,0.75) circle (2.5 pt)  node[above,black]{$-x_2$};
		\filldraw[red] (5,-2.25) circle (2.5 pt)  node[below,black]{$-x_2$};
		\filldraw[red] (13.06,-2.25) circle (2.5 pt)  node[below,black]{$x_2$};
		\filldraw[red] (13.06,0.75) circle (2.5 pt)  node[above,black]{$x_2$};
		
		\draw[line width=0.5mm, color=blue](1.06,0.75)--(5,0.75);
		\draw[line width=0.5mm, color=blue](13.06,-2.25)--(17.0,-2.25);
		\draw[line width=0.5mm, color=blue](1.06,-2.25)--(5,-2.25);
		\draw[line width=0.5mm, color=blue](13.06,0.75)--(17.0,0.75);
		\draw[dashed](12,3.75)--(12,0.75); 
		\draw[dashed](6,3.75)--(6,0.75);
		\draw[dashed](1.06,.75)--(1.06,-2.25);
		\draw[dashed](5,.75)--(5,-2.25);
		\draw[dashed](13.06,.75)--(13.06,-2.25);
		\draw[dashed](17,.75)--(17,-2.25);
		
		\draw (16.5,3.6)  node[above]{$\mathcal{R}_1$};
		\draw (16.5,0.7)  node[above]{$\mathcal{R}_2$};
		\draw (16.5,-2.3)  node[above]{$\mathcal{R}_3$};
		
	\end{tikzpicture}
	\caption{The three sheets $\mathcal R_1$, $\mathcal R_2$ and $\mathcal R_3$ of the Riemann surface in the case where $\Delta_1 = [-x_1,x_1]$ consists of one interval. \label{fig:three-sheets}}
\end{figure}

\begin{lemma} \label{lemma32}
	The functions \eqref{S1z}--\eqref{S3z} are the three branches of a meromorphic function
	$S$
	on the Riemann surface $\mathcal R$, where $S_j$ is considered as a function
	on the $j$th sheet.	
	The meromorphic function has degree three with three simple poles, namely 
	at $\pm ia$ on the first sheet, and at $0$ on the third sheet.
\end{lemma}
\begin{proof}
	Denoting by $S_{j,\pm}(x)$ the limiting value of $S_j(z)$ as $z \to x \in \mathbb R$
	with $ \pm \Im z > 0$, we have to show
	that $S_{2,+} = S_{2,-}$ on $\mathbb R\setminus \left(\Delta_{1}\cup \Delta_{2}\right)$, $S_{3,+} = S_{3,-}$
	on $(-x_2,0) \cup (0,x_2)$, $S_{1,\pm} = S_{2,\mp}$ on $\Delta_1$ and $S_{2,\pm} = S_{3,\mp}$
	on $\Delta_2$.
	
	To do so, we use the Sokhotskii-Plemelj formula according to which for
	measures $\mu$ with a continuous density, such as $\nu_1$ and $\nu_2$,
	one has 
	\begin{equation} \label{pvint1} 
		\lim_{\varepsilon \to 0+} \int \frac{d\mu(s)}{x \pm i \varepsilon -s}
		= \dashint  \frac{d\mu(s)}{x-s}  \mp \pi i \frac{d\mu}{dx}, 
		\qquad x \in \mathbb R,
	\end{equation}
	where $\dashint$ denotes the principal value. 
	We also use that
	\begin{equation} \label{pvint2} 
		\dashint \frac{d\mu(s)}{x-s} = \frac{d}{dx} \left( \int \log |x-s| d\mu (s) \right)
		= - \frac{d}{dx} U^{\mu}(x). \end{equation}
	
	From \eqref{S1z}, \eqref{S2z}, \eqref{S3z}, \eqref{pvint1} and \eqref{pvint2} one obtains for $x \in \mathbb R$
	\begin{align} 	\label{S1pm} 
		S_{1,\pm}(x) & = - \frac{d}{dx} U^{\nu_1}(x) \mp \pi i \frac{d\nu_1}{dx} 
		+ \frac{2cx}{x^2+a^2}, \\ 	\label{S2pm} 
		S_{2,\pm}(x) & = \frac{d}{dx} U^{\nu_1}(x) \pm  \pi i \frac{d\nu_1}{dx}
		-  \frac{d}{dx} U^{\nu_2}(x) \mp \pi i \frac{d\nu_2}{dx} 
		\pm ia, \\
		\label{S3pm} 
		S_{3,\pm}(x) & =  \frac{d}{dx} U^{\nu_2}(x) \pm  \pi i\frac{d\nu_2}{dx}
		+ \frac{t+2c}{x}  \mp ia.
	\end{align}
	For $x \in (-x_2, x_2)$ we have $\frac{d\nu_2}{dx} = \frac{d\sigma_0}{dx} = \frac{a}{\pi}$  by part (a) of Proposition \ref{prop35}.   
	Then, \eqref{S3pm} shows that $S_{3,+}(x) = S_{3,-}(x)$ for $x \in (-x_2,0) \cup (0,x_2)$,
	and $S_3$ has analytic continuation across $(-x_2,0) \cup (0,x_2)$ with
	a simple pole at $0$. From \eqref{S2pm} one obtains 
	\begin{align} \label{S2pmb} 
		S_{2,\pm}(x)  =  \frac{d}{dx} U^{\nu_1}(x) \pm  \pi i \frac{d\nu_1}{dx}
		-  \frac{d}{dx} U^{\nu_2}(x), \quad \text{for } x \in (-x_2,x_2). 
	\end{align}
	For $x \in (-x_2,x_2) \setminus \Delta_1$ one has additionally
	$\frac{d\nu_1}{dx} = 0$, and thus $S_{2,+}(x) = S_{2,-}(x)$ by \eqref{S2pmb}.
	Hence, $S_2$ has analytic continuation across the 
	intervals in $\mathbb R \setminus (\Delta_1 \cup \Delta_2)$. 
	
	Also from \eqref{S1pm} and \eqref{S2pmb}, we get
	\begin{align*} S_{1,\pm}(x) - S_{2,\mp}(x)
		= -2\frac{d}{dx} U^{\nu_1}(x) +  \frac{d}{dx} U^{\nu_2}(x) +
		\frac{2cx}{x^2+a^2}, \qquad
		x \in (-x_2,x_2),  \end{align*}
	which is zero on $\Delta_1$ 
	since, 
	$2 U^{\nu_1} -  U^{\nu_2} - c \log(x^2+a^2)$ is constant on 
	$\supp(\nu_1) = \Delta_1$ due to \eqref{nu1Xnu2cond}.
	This means $S_{1,\pm} = S_{2,\mp}$ on $\Delta_1$. 
	
	Finally, from \eqref{S2pm} and \eqref{S3pm}, we find
	\begin{align*} S_{2,\pm}(x) - S_{3,\mp}(x)
		=  \frac{d}{dx} U^{\nu_1}(x) - 2 \frac{d}{dx} U^{\nu_2}(x) - \frac{t+2c}{x}, 
		\qquad \text{for } x \in \mathbb R \setminus \Delta_1. \end{align*}
	Due to,  \eqref{nu2Xnu1cond}
	we have that $2 U^{\nu_2} -  U^{\nu_1} + (t+2c) \log|x|$ is constant
	and thus $S_{2,\pm} = S_{3,\mp}$ on $\Delta_2$.
	
	We conclude that $S$ is indeed a meromorphic function on the Riemann surface
	with simple poles at $\pm ia$ on $\mathcal R_1$ and $0$ at $\mathcal R_3$.
	There is no pole at infinity, and, due to the total mass properties
	of $\nu_1$ and $\nu_2$ and the definitions \eqref{S1z}--\eqref{S3z}, one has
	\begin{equation} \label{S123asymp}
		\begin{aligned}
			S_1(z) & = \frac{t+2c}{z} + \mathcal O\left(z^{-3}\right),  \\
			S_2(z) & = \pm ia + \frac{c}{z} + \mathcal O\left(z^{-2}\right), \quad \pm \Im z > 0, \\
			S_3(z) & = \mp ia + \frac{c}{z} + \mathcal O\left(z^{-2}\right),  \quad \pm \Im z > 0,
		\end{aligned}
	\end{equation}
	as $z \to \infty$. Thus, $S$ has exactly three simple poles and therefore has degree three.
\end{proof}
The analytic continuation of $S_3$ across $(-x_2,x_2)$ is also denoted by $S_3$
and similarly for $S_2$.

We recover the measures from the $S_j$ functions via the Stieltjes-Perron inversion formula:
\begin{align} \label{nu1x}
	d\nu_1(x) & = \frac{1}{2\pi i} \left( S_{1,-}(x) - S_{1,+}(x) \right) dx,  &&  x \in \Delta_1, \\
	\label{nu2x}
	d\nu_2(x) & = \frac{a}{\pi} dx - \frac{1}{2\pi i} \left( S_{3,-}(x) - S_{3,+}(x) \right) dx, 
	&& x \in \mathbb R. \end{align}
see also \eqref{S1pm} and \eqref{S3pm}.
From \eqref{S123asymp} and \eqref{nu2x} it follows in particular that
\begin{equation} \label{mu2asymp} 
	\frac{d\mu_2(x)}{dx} = \mathcal O\left(x^{-2}\right), \quad \text{ as } x \to \pm \infty. 
\end{equation}

Because we are dealing with branches of an algebraic functions, any symmetric function in the $S_j$'s is a rational function. 
\begin{lemma}
	We have 
	\begin{equation} \label{Sjsum}
	\begin{aligned} 
		S_1(z) + S_2(z) + S_3(z) & = \frac{2cz}{z^2+a^2} + \frac{t+2c}{z} \\
		& = \frac{(t+4c)z^2 + (t+2c)a^2}{z(z^2+a^2)}, 
	\end{aligned} \end{equation}
	and, for some real constants $C_1$ and $C_2$,
	\begin{align} \label{SjSksum}
		S_1(z) S_2(z) + S_1(z) S_3(z) + S_2(z) S_3(z) & = \frac{a^2z^2 + C_1}{z^2+a^2}, \\
		S_1(z) S_2(z) S_3(z) & = \frac{(t+2c)a^2z^2 + C_2}{z(z^2+a^2)}. \label{Sjprod}
	\end{align}
	The constants satisfy
	\begin{equation} \label{C1C2ineq} 
		C_1 > a^4 +4c^2  \quad \text{ and } \quad C_2 \geq 0.  \end{equation}
\end{lemma}
\begin{proof}
	The sum formula \eqref{Sjsum} is immediate from \eqref{S1z}--\eqref{S3z}.
	
	The product $S_1(z) S_2(z) S_3(z)$ is a rational function with simple poles at
	$z=\pm ia$ and at $z=0$. Because of \eqref{S123asymp} it satisfies
	$a^2(t+2c) z^{-1} + \mathcal O(z^{-3})$ as $z \to \infty$, which implies \eqref{Sjprod}.
	
	The left-hand side of \eqref{SjSksum}
	has simple poles at $z=\pm ia$. However, the singularity at $z=0$ is removable
	since the simple pole of $S_3$ at $z=0$ is cancelled by the zero of $S_1(z) + S_2(z)$.
	Indeed, we have  $S_{1,-}(0) = -S_{1,+}(0)$ since $S_1$ is odd, which implies 
	$S_{1,\pm}(0) + S_{2,\pm}(0) = S_{1,+}(0) + S_{1,-}(0) = 0$.
	From \eqref{S123asymp} we further see that the left-hand side of \eqref{SjSksum}
	behaves as $a^2 + \mathcal{O}(z^{-2})$ as $z \to \infty$ and this gives us \eqref{SjSksum}.
	
	The property $C_2 > 0$ comes from considering the residue at $z=0$.
	The left-hand side of  \eqref{Sjprod} has residue $S_{1,\pm}(0) S_{2,\pm}(0) (t+2c)$,
	see \eqref{S3z}, and the right-hand side has residue $\frac{C_2}{a^2}$.
	Since $S_{2,\pm}(0) = \overline{S_{1,\pm}(0)}$ , we get
	\[ C_2 = a^2 \left| S_{1,\pm}(0) \right|^2  (t+2c) \geq 0.\]
	
	Now, let us look at the residue of \eqref{SjSksum} at $z= ia$. 
	For the left-hand side we get by \eqref{S1z}--\eqref{S3z}
	\begin{align*} \lim_{z \to ia} (z-ia) S_1(z) (S_2(ia) + S_3(ia)) 
		= c \left(- \int \frac{d\nu_1(s)}{ia-s} + \frac{t+2c}{ia} \right) \end{align*}
	which by symmetry (as $\nu_1$ is even), can be evaluated further as
	\[ 2iac \int \frac{d\nu_1(s)}{s^2+a^2}  - \frac{c(t+2c)}{a}i.
	\]
	Since the right-hand side of \eqref{SjSksum} has residue  
	$\frac{a^4 - C_1}{2a} i$, we find
	\begin{align*} C_1 & = a^4 + 4c^2 + 2ct - 2a^2 c \int \frac{d\nu_1(s)}{s^2+a^2}  \\
		& = a^4 + 4c^2 + 2c \int \frac{s^2}{s^2+a^2} d\nu_1(s)  
		 > a^4 + 4c^2,
	\end{align*}
	where for the second identity we used that $\nu_1$ has total mass $t$. 
	Hence \eqref{C1C2ineq} follows.
\end{proof}

\begin{corollary} \label{corol39} 
	$S_1(z)$, $S_2(z)$ and $S_3(z)$ are the three solutions of the
	polynomial equation $P(S,z) = 0$, seen as a polynomial in $S$
	with coefficients depending on $z$, where 
	\begin{multline} \label{PSz} 
		P(S,z) =  
		z^3 S^3 - (t+4c) z^2 S^2  + 
		a^2 (zS^3 + z^3 S) \\
		- a^2(t+2c) (S^2 + z^2) 	+ C_1 zS  - C_2. \end{multline}
\end{corollary}
\begin{proof}
	Note that $(S-S_1(z))(S-S_2(z))(S-S_3(z)) = 0$ is the polynomial
	in $S$ with the three solutions $S_j(z)$, $j=1,2,3$. Expanding
	the product and using the expressions \eqref{Sjsum}--\eqref{Sjprod}
	for the elementary symmetric functions, we find that
	the coefficients are meromorphic in $z$ with possible simple poles
	at $0$ and $\pm ia$ only. The corollary then follows if we 
	multiply  by $z (z^2+a^2)$ and rearrange the terms.
\end{proof}
Observe that $P(S,z) = P(z,S)$, and $P(-S,-z) = P(S,z)$.
The latter identity reflects the fact that the solutions $S_j(z)$
are odd functions. We call $P(S,z) = 0$ the \textit{spectral curve}.
\begin{remark}
 The constant $C_1$ in \eqref{SjSksum} and $C_2$ in \eqref{Sjprod}  (which is of interest in itself)  can be determined by studying the discriminant $\Disc _S P$.  Namely the constant $C_1$ can be determined as a unique root of a certain degree $5$ polynomial with coefficient in $a,c$ and $t$. See \cite{BK12} for a similar approach. We do not give further details as it will not be relevant for our analysis.
\end{remark}

\subsubsection{Nodes and discriminant}

We have the following picture for $S_1, S_2, S_3$.
The three functions are purely imaginary on the imaginary axis.

If we assume $0\in \Delta_{1}$, because of \eqref{nu1x}
one has 
\[  \Im S_2(iy) =  - \Im S_1(iy) \to  \pi  \frac{d\nu_1}{dx}(0) \geq 0
\quad \text{as } y \to 0+. \] Then the following
picture emerges for $y > 0$. 
\begin{itemize}
	\item $y \mapsto \Im S_1(iy)$ starts at a non-positive value 
	at $y=0+$, and due to the pole at $ia$ with residue $c > 0$
	tends to $+\infty$ as $y \to a-$.
	On $(a,\infty)$ it goes from $-\infty$ to $0$.
	\item $y \mapsto \Im S_2(iy)$ starts at a non-negative value
	at $y=0+$ and tends to $a > 0$ as $y \to +\infty$
	by \eqref{S123asymp}.
	\item $y \mapsto \Im S_3(iy)$ starts at $-\infty$ as $y \to 0+$
	and tends to $-a$ as $y \to +\infty$.
\end{itemize} 
Thus, $S_1$ has a zero on the imaginary axis in $[0,ia)$, say at $ib_0$,
and then by symmetry also a zero at $-ib_0$. 
There is also a zero of $S_1$ at infinity, due to \eqref{S123asymp}.
 These are the only zeros of $S$ on the Riemann surface
since $S$ has degree three.

\begin{figure}[t] \centering 
	\includegraphics[scale=0.5]{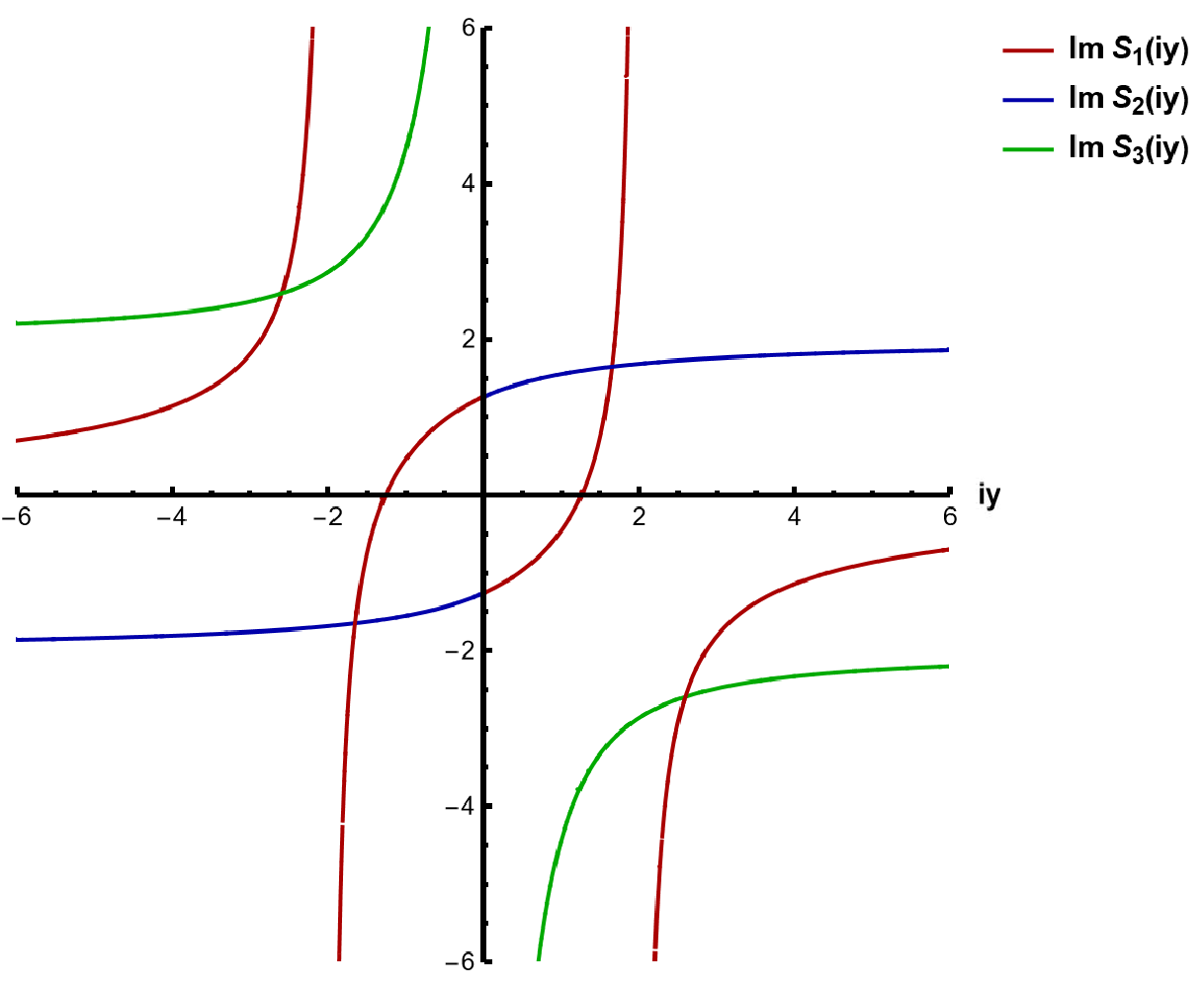}
	\caption{Behavior of $S_1$, $S_2$, $S_3$ on the imaginary axis
		in case $\frac{d\nu_1}{dx}(0) > 0$. 
		These functions are purely imaginary axis on the imaginary axis and the figure shows the graphs of their imaginary parts. The graphs of $S_1$ and $S_2$ intersect at $\pm ib_1$
		and the graphs of $S_1$ and $S_3$ intersect at $\pm ib_2$. \label{fig:Sonimagaxis}}
\end{figure}

Similarly, we conclude from the above that every purely imaginary 
value is taken at least three times on the imaginary axis by one of the functions $S_j$,
and hence exactly three times (counted with multiplicity), since the degree is three.
 In summary we obtain graphs as shown in Figure \ref{fig:Sonimagaxis}, in particular that the functions $\Im S_j$ are everywhere increasing on the imaginary axis, except at the poles. 
From the plots in Figure~\ref{fig:Sonimagaxis} we also observe
that there are points on the positive imaginary axis where two branches intersect. Namely there is $b_1 < a$ with $S_1(ib_1) = S_2(ib_1)$
and there is  $b_2 > a$ with $S_1(ib_2) = S_3(ib_2)$. 

In contrast, if we assume $0\notin \Delta_1$, then the functions $S_{1},S_{2}$ are analytic on the imaginary axis. As they are odd functions, we must have $S_{1}(0)=S_{2}(0)=0$. By analysing the graph, also in this scenario there is a point $b_{2}>a$ such that $S_{1}(ib_{2})=S_{3}(ib_{2}).$

Combining the above observations more precisely, we have the following.

\begin{corollary}\label{zeroandnode}
	There exist $b_0$, $b_1$, $b_2$ with $0 \leq b_0 \leq b_1 < a < b_2 < \infty$ such that
	$S_1(ib_0) = 0$, $S_{2,+}(0) = ib_0$, and
	\begin{equation} \label{S1ibj} 
		S_1(\pm ib_1) = S_2(\pm ib_1) = \pm ib_1, \quad S_1(\pm ib_2) = S_3(\pm ib_2) = \mp ib_2. 
	\end{equation}
	The inequalities $0 \leq b_0 \leq b_1$ are strict in case $\frac{d \nu_1}{dx}(0) >0$
	(as shown in Figure~\ref{fig:Sonimagaxis}).
\end{corollary}
\begin{proof} 
	We already noted that $S_1$ has a zero on the positive imaginary axis (or at zero),
	which we denote by $ib_0$. Since $S_{1,+}(0) \leq S_{2,+}(0)$ 
		while $\Im S_{1}(iy) \to +\infty$ as $y \to a-$,
		the graphs of $S_1$ and $S_2$ 
	intersect at a point $ib_1$ with $b_0 \leq b_1 < a$,
	and the inequality is strict when $\frac{d \nu_1}{dx}(0) >0$.
	
	We use the fact that \eqref{PSz} is symmetric in the two variables $S$ and $z$. which implies
	that the totality of the graphs is invariant under $(z,S) \mapsto (S,z)$.  
	From the graphs in Figure \ref{fig:Sonimagaxis} we then get for $b_0 < y_1 < a$
	and $y_2 > 0$,
	that $S_1(iy_1) = i y_2$ if and only if $S_2(iy_2) = iy_1$.
	Letting $y_1 \to b_0+$ we have $y_2 \to 0+$ and thus $S_{2,+}(0) = i b_0$.
	
	Also, from $S_1(ib_1) = S_2(ib_1) = i \tilde{b}$ we get $S_1(i \tilde{b}) = 
	S_2(i\tilde{b}) = ib_1$. If $b_1 < \tilde{b}$, then   $\Im S_2(ib_1) < \Im S_2(i \tilde{b})$, since $\Im S_2$ is increasing along the positive imaginary axis.
	This results in $\tilde{b} < b_1$, which is a contradiction. We get a similar
	contradiction if $\tilde{b} < b_1$. Hence $S_1(ib_1) = S_2(ib_1) = ib_1$
	and by symmetry also $S_1(-ib_1) = S_2(-ib_1) = - i b_1$, 
	since the functions are odd.
	
	The identity $S_1(ib_2) = S_3(ib_2) = - i b_2$ follows in a similar way.
\end{proof}

The points  $(\mp i b_2, \pm i b_2)$
are nodes of the spectral curve (a.k.a.\ ordinary double points), 
since the two solutions $S_1$ and $S_3$
agree at $\pm i b_2$ with value $\mp i b_2$, and their tangents at
$\pm ib_2$ are different. 
Also, $(\pm i b_1, \pm i b_1)$ are nodes if $\frac{d \nu_1}{dx}(0) >0$, i.e., if $b_1 > 0$. 
Each node is a
double zero of the discriminant $\Disc_S P$ of \eqref{PSz} with respect to $S$.

The crucial point now is the following.
\begin{lemma}
	$\Disc_S P$ is a polynomial in $z$ of degree $12$.
\end{lemma}
\begin{proof}
	Based on \eqref{PSz} the discriminant can be explicitly calculated
	by any symbolic computing system. We did it with Maple and Mathematica,
	and it indeed gives a polynomial of degree $12$
	that we could write down explicitly, but it is not very illuminating to do so here.
\end{proof}

Having only $12$ zeros of the discriminant limits of course the number
of possibilities.
We can distinguish three cases.
\begin{description}
	\item[Case $\frac{d\nu_1}{dx}(0) > 0$:] In this case, $b_1 > 0$
	and $\pm ib_1$ are double zeros of $\Disc_S P$ as well.
	Thence, we have eight zeros on the imaginary axis (counting multiplicities),
	which leaves us with only four remaining zeros.
	Since we have at least four branch points, we conclude that
	there are exactly four branch points and they are all
	simple zeros of the discriminant.
	
	\item[Case $0 \not\in \Delta_1$:] In this case $\Delta_1$
	has at least two intervals (by symmetry) and thus there are
	at least six branch points on the real axis. We also
	have the double zeros at $\pm ib_2$ and at least a 
	double zero at $ib_1 = 0$. Then we conclude that $0$
	is a double zero and we have exactly six branch points that are 
	all simple zeros 	of the discriminant.	
	
	\item[Case $0 \in \Delta_1$ and $\frac{d\nu_1}{dx}(0)= 0$:]
	In this case the density of $\nu_1$ vanishes quadratically at $0$.
	By \eqref{nu1x}, one has then $S_2(iy) - S_1(iy) = O(y^2)$
	as $y \to 0$, and the graphs of $iS_1$ and $iS_2$ not only intersect at $0$
	but are even tangent to each other. Then $0$ is a zero of
	the discriminant of order at least four. Then again by simple
	counting argument, all branch points are simple zeros
	of the discriminant and $0$ is a zero of degree exactly four. 
\end{description}
As a corollary we obtain the following result, which characterises the nodes.
\begin{corollary}\label{nodes}
    On the imaginary axis, $S_{1}(z)=S_{2}(z)$ only when $z=\pm ib_{1}$, and $S_{1}(z)=S_{3}(z)$ only when $z=\pm ib_{2}$.
\end{corollary}

\subsubsection{Proof of Proposition \ref{prop35}, parts (b)-(e)}

We continue to write $\nu_j = \nu_{j,t}^X$ and $\Delta_j = \Delta_{j,t}^X$.

\begin{proof}

(b) From the above discussion, we see that there are at most six branch
points, that is, at most six endpoints of $\Delta_1$ and $\Delta_2$.
Among them are the endpoints $\pm x_2$ of $\Delta_2$. That leaves at most four 
endpoints of $\Delta_1$, and then by symmetry $\Delta_1$ is either
a single interval $[-x_1,x_1]$ or a union of two disjoint
intervals $[-x_1,-x_0] \cup [x_0, x_1]$ with $0 < x_0 < x_1$. 
Since $\Delta_1 \subset (-X,X)$
we also have $x_1 < X$.

\medskip

(c) The measure $\nu_1$ is given by \eqref{nu1x}, which can
also be written as
\[ \frac{d\nu_1}{dx} = \frac{1}{2 \pi i} \left(S_{2,+}(x) - S_{1,+}(x)\right), \quad x \in \Delta_1. \]
It shows that the density is indeed real analytic on the interior of $\Delta_1$.
(This also follows from the results of \cite{DKM98}, since $\nu_1$
is an equilibrium measure in a real analytic external field.)
The density of $\nu_1$ will vanish like a square root at each endpoint of $\Delta_1$
since if there were a higher vanishing at an endpoint, it would 
lead to higher order zero of the discriminant. (Recall that all branch points
are simple zeros.) 

There cannot be any zeros in the interior of $\Delta_1$, except possibly at $0$,
since any such zero $x^*$ would lead to $S_{1,+}(x^*) = S_{2,+}(x^*)$
and hence to a zero of the discriminant.  

\medskip
(d) Lemma \ref{lemma33} tells us that $\nu_2 \geq \rho$.
Because of \eqref{nu2x} we have that
\[ \frac{d (\sigma_0 - \nu_2)}{dx}
	= \frac{1}{2\pi i} \left(S_{2,+}(x) - S_{3,+}(x) \right), \quad x \in \Delta_2. \]
Therefore, part (d) follows with similar reasoning as we did for part (c).

\medskip
(e) Suppose that equality holds in \eqref{Unu1strict} for some $x \in [-x_2,x_2] \setminus
\Delta_{1}$. We may assume that either $x_1 < x \leq x_2$, or $0 < x < x_0$.
Since equality holds in \eqref{Unu1strict} for $x_0$ and $x_1$,
we may apply Rolle's theorem and  find in either case that there is an $x^* \in (-x_2,x_2) \setminus (\Delta_{1} \cup \{0\})$
with
\[ \frac{d}{dx} \left. \left( 2 U^{\nu_1}(x) - U^{\nu_2}(x) - c \log(x^2+a^2) \right) 	\right|_{x=x^*} = 0. \]
In view of \eqref{S1pm} and \eqref{S3pm} this means that
\begin{equation} \label{Unu1strict1} 
	- 2 S_{1}(x^*) - S_{3}(x^*)  + \frac{2cx^*}{(x^*)^2+a^2} + \frac{t+2c}{x^*} = 0 
	\end{equation} 
To obtain \eqref{Unu1strict1} we also used $x^* \in (-x_2,x_2) \setminus \Delta_1$
which implies that $S_1(x^*)$ and $S_3(x^*)$ are real. 
Because of \eqref{Sjsum} we then get $S_{1}(x^*) = S_{2}(x^*)$ which implies that
$x^*$ is zero of the discriminant and this is a contradiction, since
the only real zeros of the discriminant are at the branch points, or possibly at $0$. 
\end{proof}

\subsubsection{Proof of Proposition \ref{prop35} (f) (first part)}\label{prelimineq}

The proof idea of part (e) can be used to establish  the inequality \eqref{Unu2strict} 
on $(-x_2,x_2) \setminus [-x_1,x_1]$.

\begin{proof}[Proof of part (f) for  $-x_2 < x \leq -x_1$ or
	$x_1 \leq x < x_2 $.] 

Suppose that the equality holds in \eqref{Unu2strict} for some $x \in (-x_2,x_2) \setminus
[-x_1,x_1]$. We may assume $x \in [x_1,x_2)$ because of the symmetry. 
Since equality holds at $x_2$ as well, Rolle's theorem gives us 
$x^* \in (x_1,x_2)$ with
\[ \frac{d}{dx} \left. \left( 2 U^{\nu_2}(x) - U^{\nu_1}(x) + (t+c) \log|x| \right) 	\right|_{x=x^*} = 0. \] 
Using \eqref{S1pm}, \eqref{S3pm}, and \eqref{Sjsum}, we then find, similar
to the proof of part (e) above, that this leads to
$S_{2}(x^*) = S_3(x^*)$ which is impossible since $x^*$
is not a branch point or a node.
\end{proof}

The proof of part (f) for $x \in [-x_1,x_1]$ requires some more effort as the same technique does not work.
We are going to vary $t$ and $X$ and so we write
$\nu_{1,t}^X$ and $\nu_{2,t}^X$. We also write $\Delta_{j,t}^X$ for $j=1,2$.

\subsubsection{The proper cutoff}
Some extra care is needed for the choice of the cutoff $X$. 

\begin{definition}  
We call  $X=X(t)$ a \textit{proper cutoff for $t$},  if
the assumptions of Proposition \ref{prop35} are satisfied, that is, if
$\Delta_{1,t}^X \subset (-X,X)$ and $\Delta_{2,t}^X \subset \mathbb R \setminus [-X,X]$.
\end{definition}
If $X$ is a proper cutoff for $t$, then
$\Delta_{1,t}^X$ and $\Delta_{2,t}^X$ have the
form as given in Proposition \ref{prop35} (a), (b).

In what follows, we show that there is a $t^{*}>0$ which satisfies part \ref{item6}  of Theorem \ref{thm:mu12}, such that for $0<t<t^*$, we can always choose a proper cutoff for $t$.
The following is now immediate from Proposition \ref{prop35} (a)-(e):

\begin{corollary} \label{cor312}
	Suppose $X$ is a proper cutoff for $t$, and let $x_2 > x_1 > 0$
	be as in Proposition \ref{prop35} (a), (b).  
	Then $X'$ is a proper cutoff for $t$  for every $X' \in (x_1, x_2)$ 
	and $\nu_{j,t}^{X'} = \nu_{j,t}^X$ for $j=1,2$. 
\end{corollary}
\begin{proof}
	Let $X' \in (x_1, x_2]$. Then $\supp(\nu_{1,t}^X) \subset [-X',X']$
	and  the variational conditions \eqref{nu1Xcond}-\eqref{nu2Xcond} 
	remain satisfied 
	by the pair ($\nu_{1,t}^X, \nu_{2,t}^X)$  if we replace $X$ by $X'$. For the variational inequality in \eqref{nu1Xcond}
	this is due to part (e) of Proposition \ref{prop35}.
\end{proof}

If $X=X(t)$ is a proper cutoff for $t$, then we write $x_1 = x_1(t)$ and $x_2 = x_2(t)$
and we note that $x_1(t) < X \leq x_2(t)$.  If $X$ is a proper cutoff for $t$,
we may vary $t$ and then the following holds.

\begin{lemma} \label{lemma313}
Suppose $X$ is a proper cutoff for $t_{0}$ with $x_1(t_0) < X < x_2(t_0)$. 
Then there exist $\varepsilon_1, \varepsilon_2 > 0$ such that $X$ is a proper cutoff 
for every $t \in (t_0-\varepsilon_2, t_0+\varepsilon_1)$.
  
\end{lemma} 
\begin{proof}
 By the monotonicity property of Lemma \ref{lemma34} (a), 
	we have that $\Delta_{1,t}^X = \supp(\nu_{1,t}^X)$ is increasing as a function of $t$,
	while $\Delta_{2,t}^X = \supp(\sigma_0 - \nu_{2,t}^X)$ is decreasing.
	Thus, we have for every $t > 0$,
	\begin{align} \label{Delta1monotone}
		t < t_0 \implies \Delta_{1,t}^X & \subset \Delta_{1,t_0}^X \subset (-X,X), \\
		\label{Delta2monotone}
		t > t_0 \implies \Delta_{2,t}^X & \subset \Delta_{2,t_0}^X \subset \mathbb R \setminus [-X,X],
	\end{align} 
	since $X$ is a proper cutoff for $t_{0}$ and $X \neq x_2(t_0)$.

	Let $\varepsilon > 0$ and suppose that $\pm X \in \Delta_{1,t_0+\varepsilon}^X$.
	Then, by \eqref{nu1Xcond} the following  equality holds at $\pm X$
	\[ 2 U^{\nu_{1,t_0+ \varepsilon}^X}(x) - U^{\nu_{2,t_0+\varepsilon}}(x)
		- c \log (x^2 + a^2)  =  \ell_{t_0+\varepsilon}^X 
		\qquad \text{ for } x = \pm X. \]
	If this holds for every $\varepsilon > 0$, then by letting $\varepsilon \to 0+$
	we easily obtain the equality for $\varepsilon = 0$ as well, 
	and this would contradict \eqref{Unu1strict}, since $x_1(t_0) < X < x_2(t_0)$.
	Thus, there is $\varepsilon_1 > 0$ such that $\Delta_{1,t_0 + \varepsilon_1}^X \subset (-X,X)$. 
	This also implies $\Delta_{1,t}^X \subset (-X,X)$ for every $t < t_0 + \varepsilon_1$
	because of the monotonicity in $t$. Combining this with \eqref{Delta2monotone}
	we conclude that $X$ is a proper cutoff for every $t \in (t_0, t_0 + \varepsilon_1)$.

	If $\pm X \in \Delta_{2,t_0-\varepsilon}^X$ then
	\[ 2 U^{\nu_{2,t_0- \varepsilon}^X}(x) - U^{\nu_{1,t_0-\varepsilon}}(x)
	+ (t_0-\varepsilon + 2c) \log |x| =  0 
	\qquad \text{ for } x = \pm X. \]
	It this holds for every $ \varepsilon >0$ then by letting $\varepsilon \to 0+$
	we obtain the equality at $x = \pm X$ for $\varepsilon =0$ as well.
	Then the strict inequality \eqref{Unu2strict} in 
	part (f) of Proposition \ref{prop35} would not hold at $\pm X$.
	However, we just proved above that strict inequality holds
	on $(-x_2(t_0),-x_1(t_0)] \cup [x_1(t_0),x_2(t_0))$, and we obtain a contradiction 
	since $x_1(t_0) < X <x_2(t_0)$.
	Thus there is an $\varepsilon_2 > 0$ such that  $\pm X \not\in \Delta_{2,t_0-\varepsilon_2}^X$. Then also $\pm X \not\in \Delta_{2,t}^X$ for every $t > t_0 - \varepsilon_2$
	by the monotonicity property.
	Hence $\Delta_{2,t}^X \subset (-\infty, X) \cup (X, \infty)$ for every $t > t_0 - \varepsilon_2$.
	Recalling \eqref{Delta1monotone} we conclude that 
	$X$ is a proper cutoff  for every $t \in (t_0-\varepsilon_2,t_0)$.
\end{proof}
\subsubsection{Proof of Proposition \ref{prop35} (f) (remaining part)}

Now we prove part (f) of Proposition \ref{prop35} for $x$ in the full interval $(-x_2,x_2)$.

\begin{proof}[Proof of part (f).] In Subsection \ref{prelimineq} we have proved the strict inequality \eqref{Unu2strict} in the interval $[-x_2(t),-x_1(t)]\cup [x_1(t),x_2(t)]$.
	It remains to prove the strict inequality \eqref{Unu2strict} on
	the interval $[-x_1(t), x_1(t)]$.
	
	Because of Corollary \ref{cor312}  we may assume without
	loss of generality that $x_1(t) < X < x_2(t)$.
	Then by Lemma \ref{lemma313}, there is a $t' < t$
	such that $X$ is a proper cutoff for  $t'$ and $x_2(t') > X$. 
	Thence we have the two sets of  variational conditions
	\begin{align} \label{Unu2Xcond1} 2 U^{\nu_{2,t}^X} - U^{\nu_{1,t}^X}
		+ (t + 2c) \log|x|
		\begin{cases} = 0, & \text{ on } \Delta_{2,t}^X, \\
			\leq 0, & \text{ on } \mathbb R, \end{cases} \\
		2 U^{\nu_{2,t'}^X} - U^{\nu_{1,t'}^X} \label{Unu2Xcond2}
			+ (t' + 2c) \log|x|
				\begin{cases} = 0, & \text{ on } \Delta_{2,t'}^X, \\
				\leq 0, & \text{ on } \mathbb R. \end{cases} \end{align}
	We subtract \eqref{Unu2Xcond2} from \eqref{Unu2Xcond1} to find
	\begin{equation} \label{Unu2tminUnu2tp} 2 U^{\nu_{2,t}^X}  - 2 U^{\nu_{2,t'}^X} 
	- U^{\nu_{1,t}^X} +  U^{\nu_{1,t'}^X}
	+  (t-t') \log|x| \leq 0,  \text{ on } \Delta_{2,t'}^X.
				\end{equation}

	Recall that $\nu_{j,t'}^X \leq \nu_{j,t}^X$ for $j=1,2$
	by the monotonicity result in Lemma \ref{lemma34} (c).
	From this it follows that $-U^{\nu_{1,t}^X} +  U^{\nu_{1,t'}^X}$
	is subharmonic in $\mathbb C$. Also, $(t-t') \log |x|$ is subharmonic, since $t'< t$.
	The two remaining terms on the left-hand side of \eqref{Unu2tminUnu2tp} 
	give us a harmonic function
	on $\mathbb C \setminus \Delta_{2,t'}^X$, since $\nu_{2,t'}^X$
	has the constant density $\frac{a}{\pi}$ on $\mathbb R \setminus \Delta_{2,t'}^X$
	and therefore $\nu_{2,t}^X = \nu_{2,t'}^X$ outside of $\Delta_{2,t'}^X$.
	Hence, the  left-hand side of \eqref{Unu2tminUnu2tp} is subharmonic 
	on $\mathbb C \setminus \Delta_{2,t'}^X$. It tends to $0$ at infinity, 
	and  the maximum principle tells us that, its maximum value is
	attained on $\Delta_{2,t'}^X$, and on $\Delta_{2,t'}^X$ only,
	as it is not constant. Hence the inequality \eqref{Unu2tminUnu2tp}
	holds on $\mathbb C$ with strict inequality $<0$ on $\mathbb C \setminus \Delta_{2,t'}^X$.
	
	In particular, since $x_2(t') > x_1(t)$ strict inequality in \eqref{Unu2tminUnu2tp}
	holds on $[-x_1(t), x_1(t)]$. After adding to this the inequality \eqref{Unu2Xcond2}
	we conclude that strict inequality in \eqref{Unu2Xcond1} holds
	for $x \in [-x_1(t), x_1(t)]$, which completes the proof. 
\end{proof}

\section{Proofs of Theorems \ref{thm:mu12}, parts \ref{item1}--\ref{item5}}\label{mainthmproofsection}

\subsection{\texorpdfstring{Definition of  $t^*$ and proof that $t^* > 0$}{}}

\begin{definition}
	We now put
	\begin{equation}\label{deft*} 
		t^* = \sup \left\{ t_0 > 0 \mid \forall t \in (0, t_0] : 
		\text{ there is a proper cutoff $X$ for $t$} \right\}. \end{equation}
\end{definition}

\begin{lemma}\label{lemma42}
If $a^2\geq 2c$, then $t^*$ defined in \eqref{deft*} is positive.
In addition, $\nu_{1,t}^X$ has a positive density at $0$
for every $t < t^*$ and proper cutoff $X$ for $t$.
\end{lemma}
In particular this means for $a^2\geq 2c$ we have $\Delta_{1,t}^X$ to be an interval $[-x_1,x_1]$. See part (b) of Proposition \ref{prop35}.
\begin{proof}
    We recall the measure $\rho$ from Proposition \ref{prop31} and its potential $U^{\rho}$
that is explicitly given by \eqref{Urho}. By straightforward calculation,
\begin{multline} \label{d2Urho} \frac{d^2}{dx^2} \left(- U^{\rho}(x) - c \log(x^2+a^2) \right) \\
	=  \frac{a^4 - 4c^2}{2ca^2} +  \frac{2c x^2 (x^2+3a^2)}{a^2(x^2+a^2)^2}
	+  \frac{a^4 x^2 \left(3c^2-a^2 x^2 + 3c \sqrt{c^2-a^2 x^2}\right)}{2c \left(c+ \sqrt{c^2-a^2 x^2}\right)^3 \sqrt{c^2-a^2 x^2}}, \\
	\text{ for } x \in [-\tfrac{c}{a}, \tfrac{c}{a}].
\end{multline} 

Due to $a^2 \geq 2c$, the right-hand side of  \eqref{d2Urho} is positive on  $[-\frac{c}{a}, \frac{c}{a}]$,
and the function $x \mapsto -U^{\rho}(x) - c \log(x^2+a^2)$ is thus strictly convex
and even on that interval and it takes its minimum value at $x = 0$ only. 
Setting $\ell_0 = - U^{\rho}(0) - 2c \log a$, then
\begin{equation} \label{Urhoineq}
		-U^{\rho}(x) - c \log(x^2+a^2) 
		\begin{cases} = \ell_0 & \text{ for } x  = 0, \\
			> \ell_0 & \text{ for } x \in [-\frac{c}{a},0) \cup (0, \frac{c}{a}].
			\end{cases} \end{equation} 

Take $X \in (0, \frac{c}{a})$. 
For $t > 0$ we have by \eqref{nu1Xcond}
\begin{equation} \label{nu1Xcond1} 2 U^{\nu_{1,t}^X}(x) - U^{\nu_{2,t}^X}(x) - c \log (x^2+a^2)
	\begin{cases} = \ell_{t}^X & \text{ for } x \in \Delta_{1,t}^X, \\
			\geq \ell_t^X & \text{ for } x \in [-X,X].
	\end{cases} \end{equation} 
As $t \to 0+$ the measures $\nu_{1,t}^X \to 0$, $\nu_{2,t}^X \to \rho$
and the left-hand side of \eqref{nu1Xcond1} tends to the left-hand side of \eqref{Urhoineq}.
Then $\ell_t^X \to \ell_0$ as $t \to 0+$, and we also find that
$\pm X \not \in \Delta_{1,t}^X$ if $t$ is small enough, i.e,
$\Delta_{1,t}^X \subset (-X,X)$.   
The inclusion $\Delta_{2,t}^X \subset \mathbb R \setminus (-X,X)$
is also satisfied, since $\nu_{2,t}^X \geq \rho$ by  Lemma \ref{lemma33}, and $X < \frac{c}{a}$.
Hence,  $X$ is proper  for $t >0$ small enough (depending on $X$).
We also conclude that $x_1(t) < X < x_2(t)$ for $t$ small enough,
and the two functions $x_1(t)$ and $x_2(t)$ are 
increasing by Lemma \ref{lemma34}(c).

Since $X>0$ can be taken as small as we like, we also get $x_1(t) \to 0$ as $t \to 0+$.
As the sets $\Delta_{1,t}^X$ are increasing, it follows that $0 \in \Delta_{1,t}^X$
whenever $X$ is a proper cutoff for $t$. In particular we always have that
$\Delta_{1,t}^X$ is equal to the single interval $[-x_1(t), x_1(t)]$, 
see Proposition \ref{prop35} (b).

By Proposition \ref{prop35} (c) the measure $\nu_{1,t}^X$ has a real
analytic density on $(-x_1(t), x_1(t))$ that does not vanish, except maybe at $0$.
We have to show that it does not vanish at $0$ either.

Take $\varepsilon > 0$ such that  $X$ is a proper cutoff for $t-\varepsilon$ as well.
Then, \eqref{nu1Xcond1} holds for both $t$ and $t-\varepsilon$.
Subtracting and using the maximum principle (as we already did a few times)
we get   
\[ 2 U^{\nu_{1,t}^X} - 2 U^{\nu_{1,t-\varepsilon}^X}
	- U^{\nu_{2,t}^X} + U^{\nu_{2,t-\varepsilon}^X} 
	 \begin{cases} = \ell_{t}^X - \ell_{t-\varepsilon}^X  & \text{ on } \Delta_{1,t-\varepsilon}^X, \\
	 	\leq \ell_{t}^X - \ell_{t-\varepsilon}^X & \text{ on } \Delta_{1,t}^X.  \end{cases} \]
Let 
\[ d\omega = \frac{1}{\pi \sqrt{x_1(t)^2 - x^2}} dx \]
	be the equilibrium measure (without any external field) of the interval $\Delta_{1,t}^X = [-x_1(t), x_1(t)]$. See for e.g., \cite{ST97}.
	It is a probability measure on $\Delta_{1,t}^X$ whose 
logarithmic potential is constant (say $\ell$) on $\Delta_{1,t}^X$.
Then we obtain
\[ 2 U^{\nu_{1,t}^X} - 2 U^{\nu_{1,t-\varepsilon}^X}
- U^{\nu_{2,t}^X} + U^{\nu_{2,t-\varepsilon}^X} - \varepsilon U^{\omega} 
	\leq \ell_{t}^X - \ell_{t-\varepsilon}^X - \varepsilon \ell, \quad 
	\text{ on } \Delta_{1,t}^X,  \]
with equality on $\Delta_{1,t-\varepsilon}^X$.

The left-hand side is subharmonic outside $\Delta_{1,t}^X$ and it tends to $0$
at infinity. By the maximum principle, it has its maximum on $\Delta_{1,t}^X$
and the inequality holds on $\mathbb C$. We apply Theorem \ref{thm:dlVP} to obtain
\[ 2 \nu_{1,t-\varepsilon}^X + (\nu_{2,t} - \nu_{2,t-\varepsilon}^X) + \varepsilon \omega	
	\leq 2 \nu_{1,t}^X  \quad \text{ on } \Delta_{1,t-\varepsilon}^X. \]
Hence $\nu_{1,t}^X \geq \frac{\varepsilon}{2} \omega$ on $\Delta_{1,t-\varepsilon}^X$,
and thus $\nu_{1,t}^X$ has a positive density at $0$.
\end{proof}

\subsection{Proof of Theorem \ref{thm:mu12}, parts \ref{item1}--\ref{item4}}
\begin{proof}  
We  proved in Lemma \ref{lemma42} that $t^* > 0$. 

Let  $t \in (0,t^*)$.
Then we have
$x_1(t)$ and $x_2(t)$ as in Proposition \ref{prop35} (a), (b), and measures $\nu_{1,t}^X$
and $\nu_{2,t}^X$ for some $X$ satisfying the remaining properties of Proposition \ref{prop35}.
We write
\begin{equation} \label{mujtdef} 
	\mu_{1,t} = \frac{1}{t} \nu_{1,t}^X, \quad \mu_{2,t} = \frac{1}{t} \nu_{2,t}^X.
	\end{equation}
Then parts \ref{item1}--\ref{item4} of  Theorem \ref{thm:mu12} 
are satisfied because of Proposition \ref{prop35} and Lemma~\ref{lemma42}.
\end{proof}
\subsection{Proof of Theorem \ref{thm:mu12} part \ref{item5}}
\begin{proof}
Define for $z \in \mathbb C$,
\begin{equation} \label{upr0} 
	u(z) = U^{\mu_1}(z) + U^{\mu_2}(z) - \Re V_1(z) + \Re V_2(z) + \frac{a}{t} | \Im z |, \end{equation}
i.e., $u$ is the left-hand side of \eqref{mu12cond6}.
For  real $z$, $u(z)$ is equal to the sum of the left-hand sides of \eqref{mu12cond4} and \eqref{mu12cond5}. 
Combining these equalities and inequalities we conclude that 
\begin{equation} \label{upr1} 
	\sup_{z \in [0,x_1]} u(z) < \ell < u(x_2). 
\end{equation}
By continuity, there is an $x_3 \in (x_1,x_2)$ with $u(x_3) = \ell$.
Actually, $x_3$ is unique with this property, since
we showed in the proofs of \ref{item3}  and \ref{item4}  above that
the left-hand sides of \eqref{mu12cond4} and
\eqref{mu12cond5}  are both strictly increasing on $(x_1,x_2)$.
Thus we have
\begin{equation} \label{upr2}
	u(x) \begin{cases} < u(x_3) = \ell, & \text{for } x \in [0,x_3), \\
		> u(x_3) = \ell, & \text{for } x \in (x_3,x_2].
	\end{cases} 
		\end{equation}
	
Define
\begin{equation} \label{upr3} 
	\mathcal{U} = \{ z \in \mathbb C \mid u(z) > \ell \}, \end{equation}
and let $\mathcal{U}_0$ be the connected component of $\mathcal{U}$
containing $x_2$. Due to \eqref{upr2} and the symmetry 
in the imaginary axis, we then have
\begin{equation} \label{upr4}
	\mathcal{U}_0 \cap [-x_1,x_1] = \emptyset.
\end{equation} 

The rest of the proof relies on the fact that $u$ is
subharmonic in 
\begin{equation} \label{upr5}
	\mathbb C \setminus ([-x_1,x_1] \cup \{ \pm ia\}).
\end{equation}
Indeed, it is immediate from \eqref{upr0} that
$u$ is harmonic away from the real axis, except at $\pm ia$,
where $-\Re V_1$ becomes $+\infty$.  
In the distributional sense one has
(with $\Delta$ denoting the Laplacian) $\Delta |\Im z| =  2dx$
and $\Delta U^{\mu} = -2 \pi \mu$  for a measure $\mu$.
Hence 
\[ \Delta \left(U^{\mu_2} + \frac{a}{t} |\Im z|\right)	
= -2 \pi  \mu_2 + \frac{2a}{t} dx,\]
which is positive due to the upper constraint $\mu_2 \leq \frac{a}{\pi t} dx$ satisfied by $\mu_2$.
Thus, $U^{\mu_2} + \frac{a}{t} | \Im z|$ is subharmonic
on $\mathbb C$. The other terms in the definition \eqref{upr0} of $u$ are harmonic in \eqref{upr5}, 
and it follows indeed that $u$ is subharmonic in the domain \eqref{upr5}.

Now we consider the boundary curve $\partial \mathcal{U}_0$
where we start at $x_3$ and follow it into the upper half-plane.
The domain $\mathcal{U}_0$ then lies to the right of $\partial \mathcal{U}_0$.
We want to know where it leaves the first quadrant.
Because of \eqref{upr4} it cannot do so in $[0,x_3]$.

Suppose that it leaves the first quadrant on the real axis
at a point $> x_3$. By symmetry it must continue in the
lower half-plane to form a closed contour starting and ending at $x_3$. 
Then  $\mathcal{U}_0$ would be a bounded domain contained in 
the region \eqref{upr5} where $u$ is subharmonic. 
Since $u = \ell$ on $\partial \mathcal{U}_0$, the maximum
principle for subharmonic functions tells us that
$u \leq \ell$ in $\mathcal{U}_0$ which is a contradiction
since by definition $u > \ell$ in $U_0$. 

Next, suppose that the curve $\partial \mathcal{U}_0$ starting
at $x_3$ tends to infinity in the first quadrant. 
Yet the first four terms in the right-hand side
of \eqref{upr0} have logarithmic behavior at infinity.
More precisely
\begin{equation} \label{upr6} 
	u(z) = -\left(1+ \frac{c}{t}\right) \log |z| + \frac{a}{t} |\Im z|
	+ \mathcal{O} (z^{-1}) \end{equation}
as $z \to \infty$.
Hence, unless we are very close to the real axis, the behavior of $u$
at infinity is dominated by the term $\frac{a}{t} |\Im z|$.
Since $u = \ell$ on $\partial \mathcal{U}_0$, it easily follows
from \eqref{upr6} that  $\arg z \to 0$ as $z \to \infty$
along $\partial \mathcal{U}_0$ in the upper half-plane. By symmetry the same is true in the lower half-plane.
 Since $\mathcal{U}_0$ is to the right of $\partial \mathcal{U}_0$,  we now also have
that $\arg z \to 0$ as $z \to \infty$ within $\mathcal{U}_0$.
Let $v(z) = \Re z^2$ and $\varepsilon > 0$. 
Then $v$ is harmonic and by the above
and \eqref{upr6} we have $u - \varepsilon v \to -\infty$
as $z \to \infty$ inside $\mathcal{U}_0$. 
We apply the maximum principle to the subharmonic function
$u-\varepsilon v$ to conclude that
\[ u - \varepsilon v \leq \max_{z \in \partial \mathcal{U}_0}
	(u(z) - \varepsilon v(z)) =
		\ell - \varepsilon \min_{z \in \partial \mathcal{U}_0} v(z),
		\quad z \in \mathcal{U}_0. \] 
Letting $\varepsilon \to 0+$ we obtain $u \leq \ell$ in $\mathcal{U}_0$, which is again a contradiction. 
Thus, the part of $\partial \mathcal{U}_0$ that starts at $x_3$ 
does not tend to infinity either.

The only remaining possibility is that $\partial \mathcal{U}_0$ comes to
the positive imaginary axis, say at $iy_0$ with $y_0 > 0$, and its mirror image
in the real axis comes to the negative imaginary axis
at $-iy_0$. 

On the imaginary axis it can be checked from, \eqref{upr0} that
  \begin{align} \nonumber
     t\frac{d}{dy} u(iy)& = t\Im \left(\int\frac{d\mu_1(w)}{iy-w}+\int\frac{d\mu_{2}(w)}{iy-w}+V_{1}'(iy)-V_{2}'(iy)\pm ia\right)
	\\
	&=\Im S_1(iy) - \Im S_3(iy).\end{align}
where we recall that $S_1$ and $S_3$ are given by \eqref{S1z} and \eqref{S3z},
 and $t \mu_j = \nu_j$ for $j=1,2$.
	 
 Additionally on the imaginary axis, both $S_{1}(iy)$ and $S_{3}(iy)$ are purely imaginary, and,  due to Corollary \ref{nodes}, $S_{1}(iy)-S_{3}(iy)$ only
 vanishes at $y= \pm b_2$. Since $b_2 > a$, there is no zero as $iy$ 
 ranges in $(0, ia)$,  hence it has a  constant sign on this interval. 
 This sign is positive as $\lim_{y\to a-}\Im(S_{2}(iy))=+\infty$,
 see also the plots in Figure ~\ref{fig:Sonimagaxis}. 
  This shows that $u(iy)$ is strictly increasing for $y \in [0,a]$. 

Thus,   clearly the interval $[iy_0, ia]$
belongs to $\mathcal{U}_0$ if $y_0 < a$, and we find a value $y_{\gamma} > a$
with $iy_{\gamma} \in \mathcal{U}_0$.  The same conclusion is true
if $y_0 > a$. 
We also take a point $x_{\gamma}$ in the real interval $(x_3,x_2)$. Then both $x_{\gamma}$ and $i y_{\gamma}$
are in $\mathcal{U}_0$. Since $\mathcal{U}_0$ is connected, we can
then connect them with a smooth contour inside $\mathcal{U}_0$
in the first quadrant. Together with its reflections
in the real and imaginary axes it constitutes 
a closed contour $\gamma$
that satisfies the requirements in part \ref{item5} of Theorem \ref{thm:mu12}.   
\end{proof}	
 
Theorem \ref{thm:mu12} is now proved, except for part \ref{item6}. 
The proof of part \ref{item6} will come after we first proved Theorem \ref{thm:droplet}.

\section{Proofs of Theorem \ref{thm:droplet} and \ref{thm:mu12} \ref{item6}} \label{dropletproof}
Recall that $S_1$, $S_2$, $S_3$ are  defined by \eqref{S1z}--\eqref{S3z}
and they are the three branches of a meromorphic function $S$ on the Riemann surface $\mathcal R$ (see \eqref{R123})
by Lemma \ref{lemma32}.
We now study $zS$ as a meromorphic function on the Riemann surface.  The level curves of $\Im zS(z)$ will be used to construct the droplet $\Omega_t$.
\subsection{\texorpdfstring{Zeros and poles of $zS$}{}}

Since $S$ is purely imaginary on the imaginary axis, we find that $zS$ is real-valued on the imaginary axis (on all sheets).

\begin{figure}[t]
	\centering
	\includegraphics[trim={0 1.5cm 0 0cm}, clip, scale=0.6]{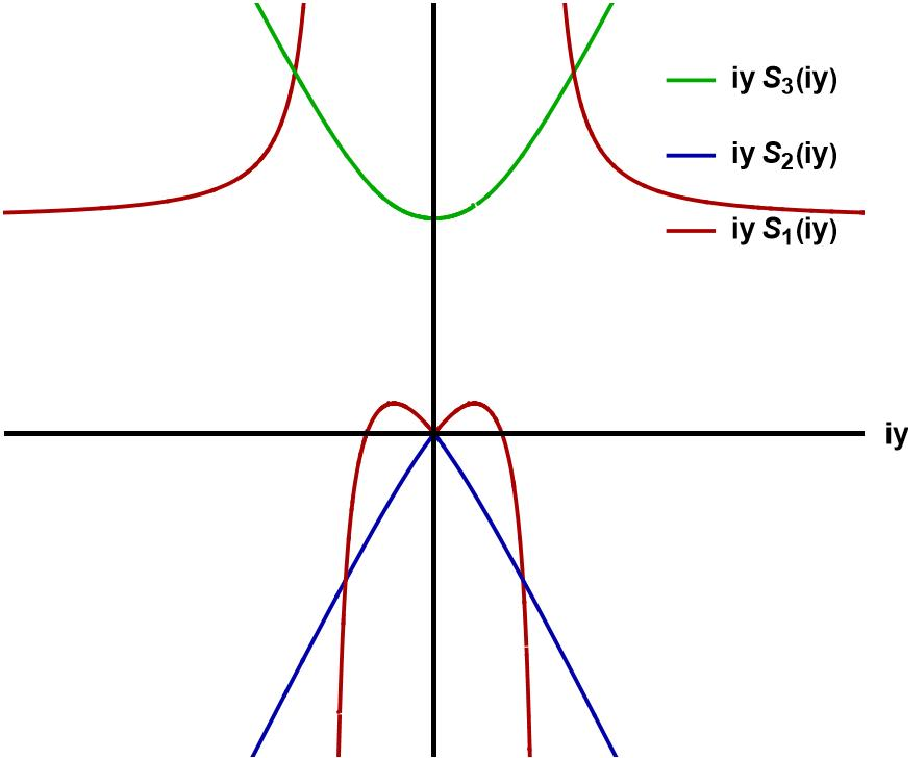}
	\caption{Plot of $zS(z)$ on the imaginary axis for $a=2,c=1,t=2$.
		$zS_1$ has four zeros on the imaginary axis, including two zeros at $0$
		on either side of the branch cut $\Delta_1$. The plot also shows that
		each real value is taken either zero or  four times on the imaginary axis (counting multiplicities including the points $\infty$).
	\label{xsximaginary2}}
\end{figure}

\begin{lemma} \label{lemma51}
    The function $zS$ is a meromorphic function of degree $4$ on $\mathcal{R}$ with the following properties.
    \begin{itemize}
     \item[\rm (a)] $zS$ has simple poles at $\pm ia$ on $\mathcal{R}_1$ 
     with residue $\pm ia c$, and simple poles at the two points at infinity
     that are common to $\mathcal{R}_2$ and $\mathcal R_3$. 
     
     \item[\rm (b)] Let $b_0 \in (0,a)$ be as in Corollary \ref{zeroandnode}.
	Then, $zS$ has simple zeros at $\pm i b_0$ on $\mathcal R_1$. 
     The other two zeros are at the two points $0$ on either side of the
     cut $\Delta_1$ that connects $\mathcal{R}_1$ and $\mathcal{R}_{2}$.
        
     \item[\rm (c)] $zS_3$ has a removable singularity at $0$ 
     with the value $t+2c$. Also, $zS_1(z) \to t+2c$ as $z \to \infty$. 
    \end{itemize}
\end{lemma}
\begin{proof}
From  Definition \ref{S1z}  we find  that $zS$ has poles 
at $\pm ia$ with residue $\pm iac$ on $\mathcal{R}_1$ as $S_{1}$ has poles 
at $\pm ia$ with residue $c$. 
From \eqref{S3z} we see that  $S_{3}$ has a simple pole with residue $t+2c$ at $0$. Then
$zS_3$ has a removable singularity with value $t+2c$. 
The only other poles of $zS$ are at infinity. From the behavior of $S_{2}$ and $S_{3}$ 
at infinity, see \eqref{S123asymp}, we find that $zS$ has simple poles at the two
points at infinity that connect the second and third sheets,
and we also find that $z S_1$ tends to $t+2c$
at infinity. This proves parts (a) and (c) of the lemma. We also conclude that $zS$
has degree four as there are exactly four poles.

For part (b) we recall from  Corollary \ref{zeroandnode}
that $S_{1}$ has zeros on the imaginary axis at $\pm i b_0$ with $0 < b_0<a$.
Then  $zS$ has zeros at $\pm i b_0$  on the first sheet.
The formulas \eqref{S1z} and \eqref{S2z} tell us that $S_{1,\pm}$ 
and  $S_{2,\pm}$ have finite non-zero values at $0$ 
meaning  $zS$ has  simple zeros at the two points at $0$ on the branch cut that
connects the first two sheets. There are no other zeros since the degree of $zS$ is four.
\end{proof}

See Figure  \ref{xsximaginary2} for a plot of $zS$ on the imaginary axis, which
shows its zeros and poles. 

\begin{corollary} \label{imaginary}
		$zS$ assumes all real values in $(-\infty,0)\cup(t+2c, \infty)$ exactly four times
		on the imaginary axis, and at no other points of the Riemann surface.
\end{corollary}
\begin{proof}
From items (a) and (c) in Lemma \ref{lemma51} it follows that 
$iy S_1(iy) \to +\infty$ as $y \to a+$
and $iy S_1(iy) \to t+2c$ as $y \to +\infty$. Hence, every real
value in $(t+2c, +\infty)$ is assumed at least once by $iy S_1(iy)$
in the interval $(a,\infty)$, and by symmetry also at least once in $(-\infty,-a)$.

Furthermore, $zS_3$ has no zeros on the imaginary axis
by Lemma \ref{lemma51}(b), with the value $t+2c$ at $0$ by Lemma \ref{lemma51}(c). 
Therefore, it is positive on the full imaginary axis.
By Lemma \ref{lemma51}(a) we know that $iy S_3(iy) \to +\infty$ as $y \to \pm \infty$.
Which leads us to the conclusion that $iy S_3(iy)$ assumes every value in $(t+2c, \infty)$ 
at least once in $(0,\infty)$ as well as at least once in $(-\infty,0)$.  

Thus, $zS$ assumes every value in $(t+2c,\infty)$ at least
four times on the imaginary axis. Since $zS$ has degree four, it takes
every value exactly four times, and there are no other points on the Riemann
surface where the value is in $(t+2c,\infty)$.
This proves part (a).

The values $(-\infty,0)$ are also taken exactly four times on the imaginary axis follow from similar considerations.
\end{proof}

Note that $zS$ is also real for real values of $z$ whenever $S$ is real.
That is $zS_1$ is real for $z \in \mathbb R  \setminus \Delta_1$,
$zS_2$ is real for $z \in \mathbb R \setminus (\Delta_1 \cup \Delta_2)$,
and $zS_3$ is real for $z \in \mathbb R \setminus \Delta_2$.
It follows from Corollary \ref{imaginary} that $0 < zS < t+2c$ for
those values on the Riemann surface.

\subsection{\texorpdfstring{Critical points of $zS$}{}}

Critical points (or ramification points) of $zS$ are points on the Riemann surface where 
$zS$ is not locally one-to-one. Since $zS$ has degree four, see Lemma \ref{lemma51},
and the Riemann surface has genus zero, there are six critical points by
the Riemann-Hurwitz formula (counting multiplicities).

\begin{lemma}\label{lemma53}
	The six critical points of $zS$ are located as follows.
\begin{itemize}
\item [\rm (a)] The points $0$ on $\mathcal{R}_3$ and $\infty$ on $\mathcal{R}_1$ are critical points of $zS$.
\item[\rm (b)] There is $p_1 > x_1$ such that $zS$ has 
critical  points at $\pm p_1$ on $\mathcal R_1$ 
with a critical value $s_1 = p_1 S_1(p_1) > 0$. Also $x \mapsto x S_1(x)$ has
a local minimum at $x=p_1$ when restricted to the interval $[x_1,\infty)$. 
\item [\rm (c)] There is $p_2$ with $0<p_2<a$ such that $zS$ has 
critical  points at $\pm ip_2$ on $\mathcal R_1$ 
with critical value $s_2 = ip_2 S_1(ip_2) > 0$. 
\item[\rm (d)] All real values in $(0,s_2)$ are taken by $zS$ four times 
on the imaginary axis of the first sheet, and at no other values on the Riemann surface.
\item[\rm (e)] We have $0 < s_2 < s_1 < t+2c$. 
\end{itemize}
\end{lemma}
\begin{proof}
(a) Note that $zS$ is an even function. Therefore, $zS_3$ is
even, then $0$ on $\mathcal R_3$ must be a critical point.
Similarly, $zS_1$ is  even, and $\infty$ on $\mathcal R_1$ is a critical point. 

(b) According to item \ref{item1}  of Theorem \ref{thm:mu12}, the density
of $\mu_1$ vanishes
like a square root at $x_1$. This implies that
\[- \int \frac{d\mu_1(s)}{(x-s)^2}  \to - \infty  \quad \text{ as } x \to x_1+. \]
By the definition \eqref{S1z} of $S_1$  we then get $S_1'(x) \to -\infty$ as $x \to x_1+$
and therefore also $(x S_1(x))' \to -\infty$ as $x \to x_1+$.

Since $0 < x_1 S_1(x_1) < t + 2c$ because of Corollary \ref{imaginary}
and $x S_1(x) \to t+2c$ as $x \to +\infty$ by Lemma \ref{lemma51}(c), we conclude
that $x S_1(x)$ must assume a minimum value at an interior point of $[x_1,\infty]$,
which gives the critical point $p_1$ of part (b) and we have $0 < p_1 S_1(p_1) < t+2c$. 
By symmetry $-p_1$ is a critical point as well.

(c) By Lemma \ref{lemma51}(b) we have that $iy S_1(iy)$ is zero at $y = b_0> 0$ and also
tends to $0$ as $y \to 0+$. It is also real valued for real values of $y$.
Therefore, there is a zero of the derivative, which gives the
critical point $ip_2$ in $(0,ib_0)$. By symmetry $-ip_2$ is a critical point as well. 

(d) Since $iy S_1(iy) > 0$ for $0 < y < b_0$, see also the graph in Figure \ref{xsximaginary2}.
we have that $s_2 = ip_2 S_2(p_2) > 0$. Every value on $(0,s_2)$ is taken two
times by $iy S_1(iy)$ in $(0,b_0)$ and by symmetry also two times in $(-r_1,0)$.
Since $zS$ has degree four there are no points on the Riemann surface where it
takes real values in $(0,s_2)$.

(e) We already observed that $s_1 = p_1 S_1(p_1)$ satisfies $0 < s_1 < t+2c$.
Since by part (d) values in $(0,s_2)$ are taken on the imaginary axis only, we also have
$s_2 < s_1$. 
\end{proof}

\subsection{Proof of Theorem \ref{thm:droplet}}

We need one more lemma which guarantees the existence of a simply connected domain which will turn out to be the droplet $\Omega_t$ of Theorem \ref{thm:droplet}.

\begin{lemma} \label{lemma54}
There is a simply connected domain $\Omega_t$, symmetric with respect
to both the real and imaginary axes, and containing the interval $[-x_1,x_1]$
such that
$z S_1(z) \in [s_2,s_1]$ for every $z \in \partial \Omega_t$.
The critical points $\pm p_1$ and $\pm i p_2$ are on the boundary of $\Omega_t$.
\end{lemma}
\begin{proof}
	By Lemma \ref{lemma53} (b), $x \mapsto x S_1(x)$ has a local minimum at $p_1$, when restricted to $(x_1, \infty)$. Then 
there is a curve $\Gamma_t$ emanating from $p_1$ into the first quadrant
such that $z S_1(z)$ is real and decreasing along $\Gamma_t$. In particular
\[ z S_1(z) \leq s_1 < t+ 2c  \quad \text{ for } z \in \Gamma_t. \]
We consider $\Gamma_t$ in the first quadrant only. Since there are no critical
points away from the axes, the curve has to exit somewhere on the axes, or
at infinity. However, the curve $\Gamma_t$ cannot go to infinity, 
since $z S_1(z) \to t + 2c$ as $z \to \infty$, by Lemma \ref{lemma51}(c).
This curve cannot  go back to the  positive real axis either, as
$x S_1(x) \geq s_1$ on $[x_1, \infty)$, while $x S_{1,+}(x)$ is
not real for $x \in (0,x_1)$. Thus, $\Gamma_t$ has to end on the positive
imaginary axis, and the only candidate is at $i p_2$, since this
is the only critical point on the positive imaginary axis.
We conclude that $z S_1(z)$ monotonically decreases from $s_1$ to $s_2$
if we move along $\Gamma_t$ from $p_1$ to $i p_2$.

By symmetry in both real and imaginary axes, there are corresponding
curves in the other quadrants, that make up a closed contour that is 
the boundary of a simply connected domain $\Omega_t$ that satisfies the 
statements of  the lemma.
\end{proof}

\begin{figure}[t]
	\centering
	\includegraphics[ trim={0cm 0.5cm 0cm 0.8cm},clip,scale=0.7]{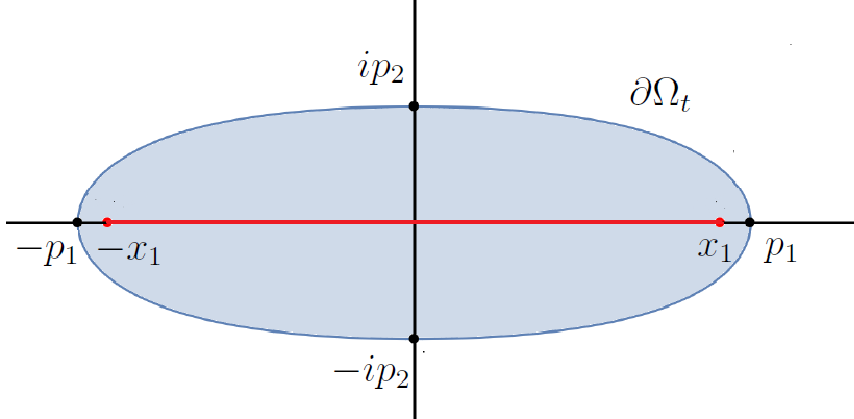}
	\caption{The domain $\Omega_t$ in shaded blue, and the interval $[-x_1,x_1]$ in red, plotted for $a=2$, $c=1$, and $t=0.1$.
	\label{dropletpic}}
\end{figure}

We are now ready for proving  Theorem \ref{thm:droplet}.

\begin{proof}[Proof of Theorem \ref{thm:droplet}]
Let $\Omega_t$ be as in Lemma \ref{lemma54}, so that $z S_1(z)$
is real for $z \in \partial \Omega_t$.

(a) From \eqref{PSz} and $P(S_1(z),z) = 0$ we obtain 
\begin{equation} \label{PS1z}
	z^2 + S_1(z)^2 = \frac{\left(zS_1(z) \right)^3- (4c+t) \left(zS_1(z) \right)^2
    	+ C_1 z S_1(z) - C_2}{a^2(2c+t- z S_1(z))}. 
\end{equation}
The right-hand side of \eqref{PS1z} is real for $z \in \partial \Omega_t$
which means that $z^2 + S_1(z)^2 \in \mathbb R$ for $z \in \partial \Omega_t$.

Then both $zS_1(z)$ and $z^2+S_1(z)^2$ are real for $z \in \partial \Omega_t$.
If $z$ is not real or purely imaginary, elementary algebra tells us that either 
$S_1(z) = \overline{z}$ or $S_1(z) = - \overline{z}$. 
Since $z S_1(z) \in [s_2,s_1]$ for $z \in \partial \Omega_t$ by Lemma \ref{lemma54}
and $s_2 > 0$, we conclude that $S_1(z) = \overline{z}$ for
every $z \in \partial \Omega_t$ that is not on the real or imaginary axis.
By continuity it then also holds for all $z \in \partial \Omega_t$. 
Since $S_1$ is given by \eqref{S1z}, the identity \eqref{S1zidentity} follows
for $z \in \partial \Omega_t$.

From Lemma \ref{lemma54} we have that $[-x_1,x_1]$
is contained in $\Omega_t$, and since $p_2 < a$, the points $\pm ia$
are in $\mathbb C \setminus \Omega_t$.
   
 (b)   By Green's theorem we have for $z$ outside the droplet
\begin{equation} \label{Green1} 
	\frac{1}{\pi}\int_{\Omega_t}\frac{dA(w)}{z-w} = 
	\frac{1}{2\pi i}\oint_{\partial \Omega_t} \frac{\overline{s}}{z-s}ds
	=\frac{1}{2\pi i}\oint_{\partial \Omega_t} \frac{S_{1}(s)}{z-s} ds, 
\end{equation}
since $S_1(s) = \overline{s}$ on the boundary of $\Omega_t$.
	
Since $S_{1}$ is meromorphic in the exterior of $\Omega_t$, the right-hand side of \eqref{Green1} can be evaluated with the residue theorem. 
The poles $\pm ia$ of $S_1$ are in $\mathbb C \setminus \Omega_t$.
Since the residue is $c$, see  formula \eqref{S1z},
the poles at $\pm ia$ give the contribution $-\frac{c}{z \mp ia}$
to the integral \eqref{Green1}. The pole at $z$ gives the contribution
$S_1(z)$ and there is no contribution from infinity due to the behavior
\eqref{S123asymp}.
We obtain
\begin{equation}\label{dropid}
   \frac{1}{\pi} \int_{\Omega}\frac{dA(w)}{z-w}=S_{1}(z)-\frac{2cz}{z^2+a^2}
	 \quad z\in \mathbb C \setminus \Omega_t.
\end{equation}
and part (b) follows because of  \eqref{S1z}.
 
(c) Letting $z \to \infty$ in \eqref{motherbody} 
and comparing terms of order $z^{-1}$, we find  $\area(\Omega_t) = \pi t$.

For $k \geq 1$ we compute the integral \eqref{harmonicmoment} for
the exterior harmonic moment by applying Green's theorem to the exterior domain
\begin{equation}
    -\frac{1}{\pi}\int_{\mathbb C\setminus\Omega_t}\frac{dA(z)}{z^{k}}
    = \frac{1}{2\pi i}\oint_{\partial \Omega_t} \frac{\overline{z}}{z^{k}} dz
    = \frac{1}{2\pi i}\oint_{\partial \Omega_t} \frac{S_1(z)}{z^{k}} dz.
\end{equation}
The latter integral can be computed with the help of
the residue theorem in the exterior domain, as we did in the proof of 
part (b). The only contributions are from the poles of $S_1$
at $\pm ia$ with residue $c$.
They give us 
\[
	-\frac{1}{\pi}\int_{\mathbb C\setminus\Omega_t}\frac{dA(z)}{z^{k}}
	= - \frac{c}{(ia)^k}  - \frac{c}{(-ia)^k}
		= \begin{cases}
				0,  & \quad \text{if $k$ is odd}, \\
				(-1)^{1+ \frac{k}{2}} \frac{2c}{a^{k}}, & \quad \text{if $k$ is even},
			\end{cases}
\]
which agrees with \eqref{wz2}.

(d) Note that $V$ given by \eqref{potential} has
the Maclaurin expansion
\begin{align*} V(z) & = c \log \left( z^2+a^2 \right) = V(0) 
	+  c \log\left( 1 + \frac{z^2}{a^2} \right) \\
	& = V(0) - c \sum_{k=1}^{\infty} \frac{1}{k} \left(- \frac{z^2}{a^2} \right)^k 
\end{align*}
which agrees with \eqref{taylor} with $t_k$ given by \eqref{wz2},
up to the additive constant $V(0)$ which is however immaterial for the model.
Since the exterior harmonic moments characterize the droplet, see \cite{WZ00},
part (d) follows. 
\end{proof}

\begin{remark} \label{fconformaltoR}
Recall that the rational function $f$ from \eqref{conformalmap} is the conformal map from the exterior
of the unit circle to the exterior of the droplet.	From \eqref{fsymmetry}
we obtain that $f(w) = z$ implies $f(w^{-1}) = S_1(z)$ for $|w| = 1$, and then by analytic continuation also in a neighboorhood of $|w|=1$. It then readily follows that
\begin{equation} \label{Rratpar} w \mapsto (S,z) = (f(w^{-1}), f(w)) \end{equation} 
is a rational parametrization of the spectral curve \eqref{PSz}.
In particular $f$ is a conformal map from the Riemann sphere to the Riemann surface $\mathcal R$.
The existence of a rational parametrization \eqref{Rratpar} 
also explains why \eqref{PSz} has to be symmetric in the two variables $z$ and $S$. 
\end{remark}

We can now also state and prove the equations 
that relate the constants $\rho$, $\kappa$ and $\alpha$ in
the conformal map \eqref{conformalmap} to the parameters $a, c$, and $t$
of the model. 
See \cite{B24} for a similar derivation.

\begin{lemma} \label{lemma56}
The parameters $a, c, t$ of the model can be recovered 
via the equations
\begin{align} \label{Eq1}
	a & = \frac{\rho}{\alpha} - \frac{2 \kappa \alpha}{1-\alpha^4}, \\
	\label{Eq2}
	c & = \frac{\rho \kappa}{\alpha^2} + \frac{2\kappa^2(1+\alpha^4)}{(1-\alpha^4)^2}, \\
	\label{Eq3}
	t+2c & = \rho^2 + \frac{2\rho \kappa }{\alpha^2}. 
\end{align}
\end{lemma}
\begin{proof} The property $f(i \alpha^{-1}) = ia$ with $f$
	given by \eqref{conformalmap} leads to the equation \eqref{Eq1}.
	Recall from \eqref{S1z} that $S_1$ has a simple pole at $z=ia$ with
	residue $c$. Since $z=ia$ on the first sheet of the Riemann surface
	corresponds to $w=i \alpha^{-1}$ we obtain
	\begin{equation} \label{Eq2proof} 
		c = \lim_{z \to ia} (z-ia) S_1(z) = 
		\lim_{w \to i\alpha^{-1}} (f(w) - ia) f(w^{-1}). \end{equation}
	We already proved that $a$ is given by \eqref{Eq1}. Then
	with $f$ given by \eqref{conformalmap} we obtain \eqref{Eq2} 
	from \eqref{Eq2proof} after straightforward calculations.
	
	Similarly, from \eqref{S3z} we know that $S_3$ has a simple pole 
	at $w=0$ with residue $t+2c$. This point corresponds to $w=0$
	in the rational parametrization \eqref{Rratpar}, and therefore
	\begin{equation} \label{Eq3proof} t + 2c = \lim_{z \to 0} z S_3(z)
		= \lim_{w \to 0} f(w) f(w^{-1}).
		\end{equation} 
	We use the formula \eqref{conformalmap} for $f$ and we obtain \eqref{Eq3}.
\end{proof}

\subsection{Proof of Theorem \ref{thm:mu12} part \ref{item6}}

We first prove that $t^*$ is finite.
\begin{lemma} \label{lemma57}
	$t^*$ is finite.
\end{lemma}	
\begin{proof}
	Suppose $t > 0$ is such that there are constants $\rho > 0$,
	$\alpha \in (0,1)$ and $\kappa \in \mathbb R$ that satisfy
	the equations \eqref{Eq1}--\eqref{Eq3}. From \eqref{Eq2} and \eqref{Eq3} it then
	follows that 
	\[ \rho^2 - t = 2 \left( c - \frac{\rho \kappa}{\alpha^2} \right)
	= 4 \frac{\kappa^2 (1+\alpha^4)}{(1-\alpha^4)^2} \geq 0, \]
	and hence $\sqrt{t} \leq \rho$.
	
	We recall $0 < \alpha < 1$. In the case $\kappa \leq 0$ we get from \eqref{Eq1} that $a \geq \frac{\rho}{\alpha}$, and therefore $\rho \leq \alpha a < a$.
	If $\kappa > 0$, then from \eqref{Eq1} we have
	\[ \rho = \alpha a + \frac{2 \kappa \alpha^2}{1-\alpha^4}
	\leq a + \frac{2 \kappa}{1-\alpha^4}. \]
	From \eqref{Eq2} it follows that
	\[ c \geq \frac{2\kappa^2 (1+ \alpha^4)}{(1-\alpha^4)^2}
	\geq  2 \left(\frac{\kappa}{1-\alpha^4} \right)^2. \]
	Combining the latter two inequalities, we conclude that
	$\rho \leq a + \sqrt{2 c}$ holds in case $\kappa > 0$.
	Thus, we have $\rho \leq a + \sqrt{2c}$ in both cases.
	
	Since $\sqrt{t} \leq \rho$, we find $t \leq (a + \sqrt{2c})^2$
	in case the equations \eqref{Eq1}--\eqref{Eq3} have a solution
	with $\rho > 0$, $\alpha \in (0,1)$ and $\kappa \in \mathbb R$.
	
	For $t < t^*$ there is a proper cut-off, which, as we have shown,
	leads to the existence of a droplet with conformal map \eqref{conformalmap}. As a consequence, there is a solution to the equations \eqref{Eq1}--\eqref{Eq3} 
	with $\rho > 0$, $\alpha \in (0,1)$ and $\kappa \in \mathbb R$,
	and thus  
	\begin{equation} \label{tstarbound} 
		t^* \leq (a + \sqrt{2c})^2. \end{equation}
	In particular $t^* < +\infty$. 
\end{proof}

We emphasize that the upper bound \eqref{tstarbound} is not sharp.
Using the equations \eqref{Eq1}--\eqref{Eq3} for the 
constants in the conformal map  \eqref{conformalmap} 
one can show that $t^{*}$ is the unique positive root of the following degree $5$ polynomial  whose coefficients depend only on $a$ and $c$. 
\begin{multline} \label{eqfortstar}
	36 t^5 + 207 c t^4 + (24 a^4 + 468 c^2) t^3 + 
	c (114 a^4 + 522 c^2) t^2 \\
	+ (4 a^8 + 108 a^4 c^2 + 288 c^4) t -c (a^8 + 30 a^4 c^2 - 63 c^4).
\end{multline} 

\begin{proof}[Proof of Theorem \ref{thm:mu12} \ref{item6}]
We already saw that $x_1(t)$ and $x_2(t)$ are increasing functions for $t \in (0,t^*)$
and $t^*$ is finite by Lemma \ref{lemma57}.
Since $\nu_{2,t}$ has the constant density $\frac{a}{\pi}$ on $[-x_2(t),x_2(t)]$
and its total mass is $t +c$, we have the bounds $x_1(t) < x_2(t) < \frac{(t+c) \pi}{2a}$ for $t < t^*$. Hence the limits 
$\lim_{t \to t^*-} x_j(t)$  for $j=1,2$ exist and are finite. Suppose 
$\lim_{t \to t^*-} x_1(t) < \lim_{t \to t^*-} x_2(t)$. After taking $X$ with
$\lim_{t \to t^*-} x_1(t) < X < \lim_{t \to t^*-} x_2(t)$, we would also
have that $\lim_{t \to t^*-} \nu_{j,t}^X$ exist for $j=1,2$ and that 
they are equal to $\nu_{j,t^*}^X$. 
We would get $\supp(\nu_{1,t^*}^X) \subset (-X,X)$ and
$\supp(\sigma_0 - \nu_{2,t^*}^X) \subset \mathbb R \setminus [-X,X]$.
Implying that  $X$ is a proper cutoff for $t^{*}$ with $x_1(t^*) < X < x_2(t^*)$. Then by Lemma \ref{lemma313}
there is an $\varepsilon > 0$ such that  $X$ is proper  for every $t \in (t^*, t^*+\varepsilon)$.
This would lead to a contradiction with the definition of $t^*$ given in \eqref{deft*}.
Hence it must be that  
$\lim_{t \to t^*-} x_1(t) = \lim_{t \to t^*-} x_2(t)$,
and part \ref{item6} of Theorem \ref{thm:mu12}  follows.
\end{proof}

\section{Steepest descent analysis}\label{Steepestdescent}

 The Riemann-Hilbert problem for $Y$ is given in
RH problem \ref{rhproblemY}.
Recall that $\Sigma_Y = \gamma $ is a simple closed contour going around
the origin and around $\pm ia$. It will be chosen such that item \ref{item5}
of Theorem \ref{thm:mu12} is satisfied.

\subsection{\texorpdfstring{First transformation $Y \mapsto X$}{}}

The first transformation $Y \mapsto X$ takes a different
form in the four connected components of the complement of $\Sigma_Y \cup \mathbb R$. The goal is twofold.
This will introduce a jump discontinuity for $X$ on the real line as well reduce the jump matrices to (in all but one case) desired block triangular forms.

\begin{definition}\label{YtoX}
Recall $V_1$ and $V_2$ defined in \eqref{V1} and \eqref{V2} respectively. We define
\begin{align} \label{Xdef1}
	X(z) & = Y(z) \begin{pmatrix} 1 & 0 & 0 \\
		0 & 0 & e^{-n V_2(z)}  e^{-iaNz} \\
		0 & - e^{nV_2(z)} e^{iaNz} & e^{iaNz} \end{pmatrix},
			&& \begin{array}{l} \text{$z$ outside of $\gamma$} \\
				\text{and $\Im z > 0$}, \end{array} \\ \label{Xdef2}
	X(z) & = Y(z) \begin{pmatrix} 1 & e^{n V_1(z)} & 0 \\
	0 & e^{-iaNz} & 0  \\
	0 & -e^{nV_2(z)} e^{iaNz} & e^{iaNz} \end{pmatrix}, 
	&& \begin{array}{l} \text{$z$ inside of $\gamma$} \\ 
		\text{and $\Im z > 0$}, \end{array} \\ \label{Xdef3}
	X(z) & = Y(z) \begin{pmatrix} 1 & 0 & 0 \\
	0 & e^{-iaNz} & 0  \\
	0 & 0 & e^{iaNz} \end{pmatrix}, &&
	\begin{array}{l} \text{$z$ inside of $\gamma$} \\  
		\text{and $\Im z < 0$}, \end{array} \\ \label{Xdef4}
	X(z) & = Y(z) \begin{pmatrix} 1 & 0 & 0 \\
	0 & e^{-iaNz} & e^{-nV_2(z)} e^{-iaNz}   \\
	0 & 0 & e^{iaNz} \end{pmatrix}, && 
	\begin{array}{l} \text{$z$ outside of $\gamma$} \\
		\text{and $\Im z < 0$}. \end{array} 
	\end{align}
\end{definition}

In what follows, we will use $x_{\gamma}$ to denote the point of intersection
of $\gamma$ with the positive real axis. Recall that
$x_3 < x_{\gamma} < x_2$.

\begin{rhproblem} \label{rhproblemX} 
	$X$ is the solution to the following RH problem.
	\begin{itemize}
	\item $X : \mathbb C \setminus \Sigma_X \to \mathbb C^{3 \times 3}$ is analytic,
	where $\Sigma_X = \Sigma_Y \cup \mathbb R$.
	\item $X_+ = X_-J_X$ on $\Sigma_X$, where
	\begin{align} \label{Xjump1} 
		J_X & = \begin{pmatrix} 1 & 0 & e^{n(V_1-V_2)} \\
				0 & 1 & 0 \\ 0 & 0 & 1 \end{pmatrix} \quad \text{ on } \gamma, \\
				\label{Xjump2}
			J_X & = \begin{pmatrix} 1 & e^{n V_1} & 0 \\ 0 & 1 & 0 \\ 0 & -e^{n V_2} & 1 \end{pmatrix}	
			\quad \text{ on } (-x_{\gamma}, x_{\gamma}), \\
			\label{Xjump3}
			J_X & = \begin{pmatrix} 1 & 0 & 0 \\ 0 & 1 & 0 \\ 0 & -e^{nV_2} & 1 \end{pmatrix}	
			\quad \text{ on } \mathbb R \setminus [-x_{\gamma},x_{\gamma}]. \end{align}
	\item $X$ remains bounded near $\pm x_{\gamma}$.
	\item As $z \to \infty$
	\begin{multline} \label{Xasymp} X(z) = \left(I_3 + \mathcal O(z^{-1})\right) X_\infty(z)
		\begin{pmatrix} z^n & 0 & 0 \\
			0 & z^{cN} e^{\pm iaNz} &  0  \\
			0 & 0 & z^{-(n+cN)} e^{\mp ia Nz} \end{pmatrix}, \\ 
		\text{with } \pm \Im z > 0, \end{multline}
	where
	\begin{equation} \label{Xasymp2}
		X_\infty(z) = \begin{cases}  	\begin{pmatrix} 1 & 0 & 0 \\ 0 & 0 & 1 \\ 0 & -1 & e^{2iaNz} \end{pmatrix}, & \Im z > 0, \\
	\begin{pmatrix} 1 & 0 & 0 \\ 0 & 1 & e^{-2iaNz} \\ 0 & 0 & 1 \end{pmatrix}, 
	& \Im z < 0. \end{cases} \end{equation}			 
\end{itemize}
\end{rhproblem}
\begin{proof}
	 We emphasise that while $V_1$ and $V_2$
	are multi-valued, both $e^{nV_1}$ and $e^{nV_2}$ are single-valued (in fact, polynomial) since $e^{n V_1(z)} =  (z^2+a^2)^{cN}$ and $e^{nV_2(z)} = z^{n + 2cN}$ and $n$ and $cN$ are positive integers. The factor $e^{-nV_2(z)}$ that appears in \eqref{Xdef1} and \eqref{Xdef4}
has a pole at $z=0$, which is inside of $\gamma$, but is otherwise analytic. Hence $X$ defined by \eqref{Xdef1}-\eqref{Xdef4} is indeed
	analytic in $\mathbb C \setminus (\gamma \cup \mathbb R)$.
	
	To verify the jump \eqref{Xjump1}
	it is convenient to write the jump matrix \eqref{Yjump} for $Y$ as 
\begin{equation} \label{Yjump2} 
	J_Y(z) = \begin{pmatrix} 1 & 0 & e^{n(V_1(z) - V_2(z))} e^{-iaNz} \\
		0 & 1 & e^{-n V_2(z)} e^{-2iaNz} \\
		0 & 0 & 1 \end{pmatrix}, \quad z \in \gamma. \end{equation}
	Then, for $z \in \gamma \cap \mathbb C^+$, we find from \eqref{Yjump2}, \eqref{Xdef1}, and \eqref{Xdef2}, that
	\begin{align*}
		J_X(z) & = \begin{pmatrix} 1 & 0 & 0 \\
			0 & 0 & e^{-n V_2(z)}  e^{-iaNz} \\
			0 & - e^{nV_2(z)} e^{iaNz} & e^{iaNz} \end{pmatrix}^{-1} \\
			& \qquad \times
			\begin{pmatrix} 1 & 0 & e^{n(V_1(z) - V_2(z))} e^{-iaNz} \\
				0 & 1 & e^{-n V_2(z)} e^{-2iaNz} \\
				0 & 0 & 1 \end{pmatrix}
				 \begin{pmatrix} 1 & e^{n V_1(z)} & 0 \\
				 	0 & e^{-iaNz} & 0  \\
				 	0 & -e^{nV_2(z)} e^{iaNz} & e^{iaNz} \end{pmatrix}   
	\end{align*}
	which indeed gives \eqref{Xjump1} by straightforward calculation.
	A similar calculation leads to the same jump matrix on $\gamma \cap \mathbb C^-$
	and \eqref{Xjump1} follows. The jump conditions \eqref{Xjump2} and \eqref{Xjump3}
	follow similarly. 
	
	For the asymptotic condition, we recall that $e^{n V_2(z)} = z^{n + 2cN}$. Then
	we find from 	\eqref{Yasymp} and \eqref{Xdef1} that as $z \to \infty$ with $\Im z > 0$, 
	\begin{align*} X(z) & = \left(I_3 + \mathcal O(z^{-1}) \right)
		\begin{pmatrix} z^n & 0 & 0 \\ 0 & z^{cN} & 0 \\ 0 & 0 & z^{-n-cN} \end{pmatrix}
		\begin{pmatrix} 1 & 0 & 0 \\ 0 & 0 & z^{-n-2cN} e^{-iaNz} \\
			0 & - z^{n+2cN} e^{iaNz} & e^{iaNz} \end{pmatrix} \\
   & =  \left(I_3 + \mathcal O(z^{-1}) \right)
		\begin{pmatrix}
			1 & 0 & 0 \\
			0 & 0 & 1  \\ 
			0 & -1  & e^{2iaNz} 
		\end{pmatrix}
		\begin{pmatrix} z^n & 0 & 0 \\
			0 & z^{cN} e^{iaNz} & 0  \\
			0 & 0 &  z^{-n-cN} e^{-iaNz}  \end{pmatrix}.
			\end{align*}
	Similarly, from \eqref{Yasymp} and \eqref{Xdef4}, we obtain as $z \to \infty$ with $\Im z < 0$,
	\begin{align*} X(z) & = \left(I_3 + \mathcal O(z^{-1}) \right)
		\begin{pmatrix} z^n & 0 & 0 \\ 0 & z^{cN} & 0 \\ 0 & 0 & z^{-n-cN} \end{pmatrix}
		\begin{pmatrix} 1 & 0 & 0 \\ 0 & e^{-iaNz} & z^{-n-2cN} e^{-iaNz} \\
			0 &  0  & e^{iaNz} \end{pmatrix} \\
		& =  \left(I_3 + \mathcal O(z^{-1}) \right)
		\begin{pmatrix}
			1 & 0 & 0 \\
			0 & 1 & e^{-2iaNz}  \\ 
			0 & 0  & 1 
		\end{pmatrix}
		\begin{pmatrix} z^n & 0 & 0 \\
			0 & z^{cN} e^{-iaNz} & 0  \\
			0 & 0 &  z^{-n-cN} e^{iaNz}  \end{pmatrix}.
	\end{align*}
	This establishes \eqref{Xasymp} with $X_\infty$ given by \eqref{Xasymp2}.
 Since $Y(z)$ is bounded near $\pm x_{\gamma}$ clearly $X(z)$ is also bounded near 
 $\pm x_\gamma$, and we verified all the conditions in the RH problem~\ref{rhproblemX}. 
\end{proof}

\subsection{\texorpdfstring{Second transformation $X \mapsto T$}{}}

The second transformation $X \mapsto T$ makes use
of the measures $\mu_1$, $\mu_2$ that satisfy the
properties listed in Theorem \ref{thm:mu12} corresponding to the
value
\begin{equation} \label{tnN} 
	t = t_{n,N} = \frac{n}{N}. \end{equation}
We assume that $0 < t_{n,N} < t^*$.
The measures $\mu_1$ and $\mu_2$ will enter the $g$-functions defined by
\begin{align} \label{g1g2}
		g_j(z) = \int \log(z-s) d\mu_j(s), \quad j=1,2, \quad z \in \mathbb C \setminus \mathbb R
	\end{align}
associated with the value \eqref{tnN}.
For each $s \in \mathbb R$, $z \mapsto \log(z-s)$ is defined as the principal branch of
the logarithm, i.e., with a branch cut along $(-\infty,s]$. Then, $g_1$ and $g_2$ are analytic in $\mathbb C \setminus \mathbb R$,
with boundary values 
\[ g_{j,\pm}(x) = - U^{\mu_j}(x) \pm \pi i \mu_j([x,\infty)) \]
on the real line. This implies 
\begin{align} \label{g1min} 
	g_{1,+}(x) - g_{1,-}(x) & = \begin{cases}
	0, 	&   \text{ for } x > x_1, \\
	2\pi i \mu_1([x,x_1]), & \text{ for } -x_1 < x < x_1, \\
	2\pi i, & \text{ for } x < -x_1. \end{cases} \\ \label{g2min}
	g_{2,+}(x) - g_{2,-}(x) & = 2 \pi i \mu_2([x,\infty)),  \qquad x \in \mathbb R. 
\end{align}
Recall that $\mu_2 = \sigma$ on $[-x_2,x_2]$, see \eqref{mu12cond3b}. Using also the total mass condition
for $\mu_2$ and the fact that $\mu_2$ is even, we have for $x \in [-x_2,x_2]$,
\[ \mu_2([x,\infty)) = \mu_2([0,\infty)) - \int_0^x d\sigma 
= \frac{1}{2} (1+ \tfrac{c}{t}) - \tfrac{ax}{\pi t}.  \]
Combining this with \eqref{g2min} we get 
\begin{align} \label{g2min2}
	g_{2,+}(x) - g_{2,-}(x) & = \pi i(1+\frac{c}{t}) - \frac{2ai x}{t}, \quad x \in [-x_2,x_2].
\end{align}

For the sum we have $g_{j,+} + g_{j,-} = -2 U^{\mu_j}$, and in view of 
the equilibrium conditions \eqref{mu12cond4}-\eqref{mu12cond5}, we get
\begin{align} \label{g1plus} 
	 g_{1,+} + g_{1,-} - \tfrac{1}{2}(g_{2,+} + g_{2,-}) + \frac{c}{t} \log(x^2+a^2) + \ell
 & \begin{cases} = 0, & \text{ on } [-x_1,x_1], \\
	< 0, & \text{ on } (-x_2,-x_1) \cup (x_1,x_2). \end{cases} \\ \label{g2plus}
 g_{2,+} + g_{2,-} - \tfrac{1}{2}(g_{1,+} + g_{1,-}) - \frac{t+2c}{t} \log|x| &
\begin{cases} > 0, & \text{ on } (-x_2,x_2), \\
	= 0, & \text{ on } (-\infty, -x_2] \cup [x_2, \infty). \end{cases} 
\end{align}

The following transformation uses the \enquote{$g$ functions} defined above in \eqref{g1g2}. As a result the jump matrices simplify to a great extent and the asymptotic behaviour of the RHP as $z\rightarrow\infty$ becomes tractable. 

\begin{definition} \label{XtoT}
We define for $z \in \mathbb C \setminus \Sigma_X$,
\begin{multline} \label{Tdef}
		T(z) = \begin{pmatrix} e^{n (\ell - \frac{\pi i}{2}(1+\frac{c}{t}))} & 0 & 0 \\ 0 & 1 & 0 \\ 0 & 0 & -1 \end{pmatrix}
		X(z) \begin{pmatrix} e^{-n (g_1(z) + \ell - \frac{\pi i}{2}(1+ \frac{c}{t}))} & 0 & 0 \\ 0 & e^{n(g_1(z) - g_2(z))} & 0 \\
			0 & 0 & -e^{n g_2(z)} \end{pmatrix} 
			\\ \times \begin{cases} \begin{pmatrix} 1 & 0 & 0 \\
				0 & e^{- iaNz} & 0 \\ 0 & 0 & e^{iaNz} \end{pmatrix}, &
				\quad \Im z > 0, \\ 
				\begin{pmatrix} 1 & 0 & 0 \\
					0 & (-1)^{n+cN} e^{iaNz} & 0 \\ 0 & 0 & (-1)^{n+cN} e^{- iaNz} \end{pmatrix},
				& \quad \Im z < 0. \end{cases}
		\end{multline}
\end{definition}

\begin{rhproblem} \label{rhproblemT} $T$ is the solution of the following RH problem.
	\begin{itemize}
	\item $T : \mathbb C \setminus \Sigma_T \to \mathbb C^{3\times 3}$
		is analytic, where $\Sigma_T = \Sigma_X = \gamma \cup \mathbb R$.
	\item $T_+ = T_- J_T$ on $\Sigma_T$ where
	\begin{align} \label{Tjump1}  \displaybreak[0]
		J_T & = \begin{pmatrix} 1 & 0 & -e^{n(V_1-V_2 + g_1 + g_2 + \ell \pm \frac{i az}{t} 
				\mp \frac{\pi i}{2}(1+\frac{c}{t}))} \\
			0 & 1 & 0 \\
			0 & 0 & 1 \end{pmatrix} \quad \text{ on } \gamma \cap \mathbb C^{\pm}, \\		\label{Tjump2}	
		J_T & = \begin{pmatrix} e^{-n(g_{1,+} - g_{1,-})} & 1  & 0 \\
			0 & e^{n(g_{1,+}-g_{1,-})} & 0 \\
			0 & (-1)^{n+cN} e^{n(V_2 + g_{1,+} - g_{2,+}-g_{2,-})}   & 1  \end{pmatrix} \quad 
			\text{on } [-x_1,x_1], \\  \label{Tjump3}
		J_T & = \begin{pmatrix} 1 & e^{n(V_1 + g_{1,+} + g_{1,-} -
				\frac{1}{2} g_{2,+}- \frac{1}{2} g_{2,-} +\ell)} & 0 \\
				0 & 1& 0 \\
				0 & (-1)^{n+cN} e^{n(V_2 +g_{1,+} - g_{2,+}-g_{2,-})}   & 1  \end{pmatrix} \quad 
		 \text{on }
			(-x_{\gamma},x_1] \cup [x_1,x_{\gamma}),  		\\  
			\label{Tjump4}
		J_T & = \begin{pmatrix} 1 & 0  & 0 \\
			0 & 1 & 0 \\
			0 & (-1)^{n+cN} e^{n(V_2 +g_{1,+} - g_{2,+}-g_{2,-})}   & 1  \end{pmatrix} \quad 
		\text{on } [-x_2,-x_{\gamma}) \cup (x_{\gamma},x_2]  \\  
		 \nonumber 
		J_T & = \begin{pmatrix} 1 & 0 & 0 \\
			0 & (-1)^{n+cN} e^{-n(g_{2,+}-g_{2,-})} e^{-2iaNx} & 0 \\
			0 & (-1)^{n+cN} & (-1)^{n+cN} e^{n(g_{2,+}-g_{2,-})} e^{2iaNx} \end{pmatrix} \\
			& \hspace*{6cm}	\label{Tjump5}
			\hfill{ \text{ on } (-\infty,-x_2] \cup [x_2,\infty)}. 
	\end{align} 
	\item $T$ remains bounded near $\pm x_{\gamma}$.
	\item As $z \to \infty$,
	 \begin{equation} \label{Tasymp}
	 	T(z)  = \left(I_3 + \mathcal O(z^{-1})\right) 
	 	T_\infty(z) 
	\end{equation}
	with 
	\begin{equation} \label{Tasymp2}
		T_\infty(z) = \begin{cases} \begin{pmatrix} 1 & 0 & 0 \\ 0 & 0 & -1 \\ 0 & 1 & e^{2iaNz} \end{pmatrix},
		 & \Im z > 0, \\
		  \begin{pmatrix} 1 & 0 & 0 \\ 0 & (-1)^{n+cN} & (-1)^{n+cN+1} e^{-2iaNz} \\ 0 & 0 & (-1)^{n+cN} \end{pmatrix},
		 & \Im z < 0. 
		 \end{cases} \end{equation}
\end{itemize}
\end{rhproblem}
\begin{proof}
	Analyticity of $T$ is clear. 
	The jump \eqref{Tjump1} follows from immediate calculations from \eqref{Xjump1}
and \eqref{Tdef} and the fact that $g_1$ and $g_2$ are analytic away from the real axis.
	
	On the part of the real line inside $\gamma$ we obtain from \eqref{Xjump2} and \eqref{Tdef}
	that 
	\begin{equation} \label{Tjumpin}	
		J_T  = \begin{pmatrix} e^{-n(g_{1,+} - g_{1,-})} & e^{n(V_1 + g_{1,+} + g_{1,-} - g_{2,+} +\ell
				- \frac{\pi i}{2}(1+ \frac{c}{t}))} e^{-iaNx}   & 0 \\
		0 & (-1)^{n+cN} e^{n(g_{1,+}-g_{1,-} - g_{2,+}+g_{2,-})} e^{-2iaNx} & 0 \\
		0 & (-1)^{n+cN} e^{n(V_2 + g_{1,+} - g_{2,+}-g_{2,-})}   & (-1)^{n+cN} e^{n(g_{2,+}-g_{2,-})} e^{2iaNx} \end{pmatrix}. \end{equation}
	In view of \eqref{g2min2} we have 
	\begin{equation} \label{g2min3} n(g_{2,+}-g_{2,-}) + 2i aNx = n\pi i \left(1+ \tfrac{c}{t}\right) = 
			(n+cN)\pi i
			\quad \text{ on } [-x_2,x_2], \end{equation}
	which we use to simplify the $(2,2)$ and $(3,3)$ entries (recall that $\mathbb R \setminus [-x_2,x_2]$ is outside of $\gamma$).
	We use \eqref{g1plus} and \eqref{g2min3} for the $(1,2)$ entry, and
	\eqref{g1min} for the $(1,1)$ and $(2,2)$ entries, and we obtain \eqref{Tjump2}
	and \eqref{Tjump3}.
	
	On the real line outside $\gamma$ we obtain from \eqref{Xjump3} and \eqref{Tdef}
	that 
	\begin{equation} \label{Tjumpout}	
		J_T  = \begin{pmatrix} e^{-n(g_{1,+} - g_{1,-})} & 0   & 0 \\
		0 & (-1)^{n+cN} e^{n(g_{1,+}-g_{1,-} - g_{2,+}+g_{2,-})} e^{-2iaNx} & 0 \\
		0 & (-1)^{n+cN} e^{n(V_2 + g_{1,+} - g_{2,+}-g_{2,-})}   & (-1)^{n+cN} e^{n(g_{2,+}-g_{2,-})} e^{2iaNx} 
		\end{pmatrix}
	\end{equation} which is the same as \eqref{Tjumpin}, except that the $(1,2)$ entry is zero.
	Then \eqref{Tjump4} follows from \eqref{Tjumpout} in the same way as \eqref{Tjump3} follows from \eqref{Tjumpin}.
	Outside $(-x_2,x_2)$ we cannot use \eqref{g2min3} to simplify the $(2,2)$ and
	$(3,3)$ entries. However, we do have \eqref{g1min} to simplify the $(1,1)$ and $(2,2)$ entries.
	Finally, we apply \eqref{g2plus} to the $(3,2)$ entry. 
	For $x \geq x_2$, we have $V_2(x) = \frac{t+c}{t} \log |x|$ and $g_{1,-}(x) = g_{1,+}(x)$ 
	and we get from \eqref{g2plus}  
	\[ V_2(x) - g_{2,+}(x) - g_{2,-}(x) + g_{1,+}(x) = 0, \qquad x \geq x_2. \]
	For $x \leq -x_2$, we have $V_2(x) = \frac{t+2c}{t} (\log |x| \pm \pi i)$ and
	$g_{1,+}(x) = g_{1,-}(x) + 2\pi i$, so that by \eqref{g2plus}
	\[ V_2(x) - g_{2,+}(x) - g_{2,-}(x) + g_{1,+}(x) = 
		\pm \tfrac{t+2c}{t} \pi i + \pi i, \qquad x \leq x_2. \]
The asymptotic condition \eqref{Tasymp} follows from \eqref{Xasymp}
	and \eqref{Tdef} if we realize that $e^{ng_1(z)} = z^n (1+O(z^{-1}))$,
	$e^{n g_2(z)} = z^{n+cN}(1+ \mathcal O(z^{-1}))$ as $z \to \infty$.
	The latter property (in particular the $\mathcal O$-term) follows from the
	fact that the density of $\mu_2$ behaves like $\mathcal O(x^{-2})$ as $|x| \to \infty$,
	see \eqref{mu2asymp} below.  
	We also use $ X_\infty(z) \left(I_3 + \mathcal O(z^{-1}) \right) = \left(I_3 + \mathcal O(z^{-1})\right) X_\infty$ 
	as $ z \to \infty$,
	which is indeed valid since $X_\infty(z)$ is uniformly bounded for $z \in \mathbb C \setminus \mathbb R$. As before $T(z)$ is obviously bounded near $\pm x_{\gamma}.$
	\end{proof}

\subsection{Rewriting of the jumps}
We recall the definitions \eqref{S1z}--\eqref{S3z} of $S_{1},S_{2},S_{3}$
with $\nu_j= t \mu_j$ for $j=1,2$.
The point $x_3$ was introduced in the proof of part \ref{item5} of Theorem \ref{thm:mu12} as the unique value in $(x_1,x_2)$ where
$U^{\mu_1} + U^{\mu_2} - \Re V_1 + \Re V_2 = \ell$.

We now introduce the $\phi$ functions, which drastically simplify the entries of the jump matrices of the RH problem \ref{rhproblemT}. This is summarised in Corollary \ref{phicor}.
\begin{definition} The $\phi$-functions are defined as follows.
\begin{enumerate}
    \item [\rm (a)] $\phi_1$ is defined and analytic in $\mathbb C \setminus
	((-\infty, x_1] \cup [x_2, \infty) \cup (-i\infty, -ia] \cup [ia, i \infty))$ with \begin{equation}\label{phi1z}
   \phi_1(z)  = \frac{1}{2t} \int_{x_1}^z (S_2(s) - S_1(s)) ds .
\end{equation}
\item [\rm (b)] $\phi_2$ is defined and analytic in $\mathbb C \setminus
	((-\infty, x_1] \cup [x_2, \infty))$ with  
 \begin{equation}\label{phi2z}
     \phi_2(z)  = \frac{1}{2t} \int_{x_2}^z (S_2(s) - S_3(s)) ds .
 \end{equation}
 \item [\rm (c)] $\phi_3$ is defined and analytic in
	$\mathbb C \setminus ((-\infty,x_1] \cup [x_2, \infty)\cup[-ia,ia])$
	with 
 \begin{equation}\label{phi3z} 
     \phi_3(z)  = \frac{1}{2t} \int_{x_3}^z
	(S_3(s) - S_1(s)) ds. 
 \end{equation}
\end{enumerate}

\end{definition}
\begin{lemma}\leavevmode\newline\vspace{-0.6cm}
\begin{enumerate}
	\item[\rm (a)] The following holds true for $\phi_1$ 
	\begin{align} \label{phi11}
	-2\phi_{1,+} & = 2 \phi_{1,-}  = g_{1,+} - g_{1,-},  \quad \text{ on } [-x_1,x_1], \\
	 \label{phi12}
		- 2\phi_1 & = V_1 + 2 g_{1} - g_{2,+} + \ell - \tfrac{iax}{t}, \quad \text{ on } [x_1,x_2], \\
		\label{phi13}
		-2 \phi_{1,\pm} & = V_1 + 2g_{1,\pm} -g_{2,+} + \ell - \tfrac{iax}{t} \pm 2 \pi i, \quad \text{ on } [-x_2,-x_1]. 
	\end{align}
	\item[\rm (b)] The following holds true for $\phi_2$
	\begin{align} \label{phi21}
		2\phi_{2,+} & = -2 \phi_{2,-}  = g_{2,+} - g_{2,-} + \tfrac{2iax}{t} - \pi i\left(1+\tfrac{c}{t}\right),  \text{ on } [x_2,\infty), \\
		\label{phi22}
		-2\phi_2 & = V_2 + g_{1} - g_{2,+} - g_{2,-}, \quad  \text{ on } [x_1,x_2], \\
		\label{phi23}
		- 2 \phi_{2,\pm} & = V_{2,\pm} + g_{1,\pm} - g_{2,+} - g_{2,-}, \quad \text{ on } [-x_1,x_1], \\
		\label{phi24}
		-2\phi_2 & = V_2 + g_{1} - g_{2,+} - g_{2,-}, \quad  \text{ on } [-x_2,-x_1], \\
		\label{phi25}
		2\phi_{2,+} & = -2 \phi_{2,-}  = g_{2,+} - g_{2,-} + \tfrac{2iax}{t} - \pi i\left(1+ \tfrac{c}{t}\right),  \text{ on }
		(-\infty,-x_2]. 
	\end{align}
	\item[\rm (c)]  $\phi_3$ can be written as follows
	\begin{align} \label{phi31} 
		-2 \phi_3(z) =  V_1(z)-V_2(z) + g_1(z) + g_2(z) + \ell \pm \tfrac{i az}{t} \mp \frac{\pi i}{2}\left(1+ \tfrac{c}{t}\right). 
	\end{align}
	\end{enumerate}
\end{lemma}
\begin{proof}
	(a)
From \eqref{g1min} and \eqref{nu1x} we have
\[ g_{1,+}(x) - g_{1,-}(x)
	= 2\pi i \mu_1([x,x_1])
	= \frac{1}{t} \int_{x}^{x_1} \left(S_{1,-}(s) - S_{1,+}(s) \right) ds.  \]
Then, we use $S_{1,\pm} = S_{2,\mp}$ on $[-x_1,x_1]$ and compare with \eqref{phi1z} to conclude
that \eqref{phi11} holds. For $x \in [x_1,x_2]$ we have by \eqref{phi11}, \eqref{S1z}, and \eqref{S2z},
\begin{align*} 
	- 2 \phi_1'(x)  & = \frac{1}{t} (S_1(x) - S_2(x)) \\
	& = \frac{2cx}{t(x^2+a^2)} + 2 \int \frac{d\mu_1(s)}{x-s} - \left(\int \frac{d\mu_2(s)}{x-s} \right)_+  - \frac{ia}{t}  \end{align*}
which is equal to the derivative of the right-hand side of \eqref{phi12}, see also \eqref{V1}, \eqref{V2}.
Since both sides of \eqref{phi12} are zero at $x_1$, and their derivatives agree on $[x_1,x_2]$,
the identity \eqref{phi12} follows.

\medskip

(b) We compute for $x > x_2$,
\begin{align*} 
	2 \phi_{2,+}(x) & =  \frac{1}{t} \int_{x_2}^x \left( S_{2,+}(s) - S_{3,+}(s)\right) ds \\
	& = \frac{1}{t} \int_{x_2}^x \left( S_{3,-}(s) - S_{3,+}(s)\right) ds \\
	& = \int_{x_2}^x \left( (-g_{2,-}'(s) + \frac{t+2c}{ts} + \frac{ia}{t}) - (-g_{2,+}'(s) + \frac{t+2c}{ts} - \frac{ia}{t}) \right) ds  \\
	& = \left[g_{2,+}(s) - g_{2,-}(s) + \frac{2ias}{t} \right]_{s=x_2}^{s=x}.
\end{align*}
Due to \eqref{g2min2}, the lower limit $s=x_2$ gives the contribution $-\pi i(1+\frac{c}{t})$.
Then \eqref{phi21} follows, from $2\phi_{2,-} = - 2\phi_{2,+}$ on $[x_2,\infty)$.

On $[x_1,x_2]$ we have
\begin{align*} S_3(x) - S_2(x) & =   (- tg_{2,+}'(x) +  tV_2'(x) - ia)  - 
	(-t g_1'(x) + t g_{2,-}'(x) - ia) \\ 
	& = 
	t \left[ V_2'(x) + g_{1}'(x) - g_{2,+}'(x) - g_{2,-}'(x) \right].
		\end{align*}
The right-hand side of \eqref{phi22} vanishes at $x_2$ and,
in view of the above, its derivative for $x \in [x_1,x_2]$
is equal to $\frac{1}{t} (S_3 - S_2)$. Then \eqref{phi22} follows.
The identity \eqref{phi23} follows in a similar way.

\medskip
(c) We compute
\begin{align*}
	-2 t \phi_3'(z) & = S_1(z) - S_3(z) \\
	& = t \int \frac{d\mu_1(s)}{z-s} + t \int \frac{d\mu_2(s)}{z-s} + \frac{2cz}{z^2+a^2} - \frac{t+2c}{z} \pm ia, \\
	& = t g_1'(z) + t g_2'(z) + t V_1'(z) - t V_2'(z) \pm ia.
\end{align*}	
Thus, the two sides of \eqref{phi31} have the same derivative.
The left-hand side vanishes at $z=x_3$.
Because of \eqref{upr2} we have $u(x_3) = \ell$ with
$u$ given by \eqref{upr0} leading to

\[ \ell =  - g_1(x_3) - \frac{1}{2} (g_{2,+}(x_3) + g_{2,-}(x_3))
	- V_1(x_3) + V_2(x_3).  \]
From \eqref{g2min2} and the fact that $x_1 < x_3 < x_2$,
we get
\[ g_{2,+}(x_3) - g_{2,-}(x_3) = \pi i \left(1+\frac{c}{t}\right)
	- \frac{2aix_3}{t}. \] 
Combining all this, we get that the right-hand side 
of \eqref{phi31} 
also vanishes at $z=x_3$, and the identity \eqref{phi31} follows. 
\end{proof}

We use \eqref{phi31} in \eqref{Tjump1},
\eqref{phi11} and \eqref{phi23} in \eqref{Tjump2},
\eqref{phi12}, \eqref{phi13}, \eqref{phi22}, and \eqref{phi24} in
\eqref{Tjump3} and \eqref{Tjump4} (we also use \eqref{g2min2}),
and finally \eqref{phi25} in \eqref{Tjump5} to
find simplified expressions for all the jump matrices for $T$ which is summarised in the following corollary.

\begin{corollary}\label{phicor}
With $\phi_j$ for $j=1,2,3$ defined in \eqref{phi1z}--\eqref{phi3z} the jump matrix $J_T$ in \eqref{Tjump1} can be rewritten as follows.
\begin{align} \label{Tjump1b} 
	J_T & = \begin{pmatrix} 1 & 0 & -  e^{-2n \phi_3} \\
		0 & 1 & 0 \\
		0 & 0 & 1 \end{pmatrix} \quad \text{ on } \gamma, \\		\label{Tjump21b}
	J_T & = \begin{pmatrix} e^{2n \phi_{1,+}} & 1  & 0 \\
		0 & e^{2n \phi_{1,-}} & 0 \\
		0 & (-1)^{n+cN} e^{-2n \phi_{2,+}}  & 1 \end{pmatrix} \quad 
	\text{on } [-x_1,x_1],  \\  \label{Tjump31b}
	J_T & = \begin{pmatrix} 1 & e^{-2n \phi_1} 	& 0 \\
		0 & 1 & 0 \\
		0 & (-1)^{n+cN} e^{-2n \phi_2}   & 1  \end{pmatrix} \quad 
	\text{on } (-x_\gamma,-x_1] \cup [x_1,x_\gamma),  \\  \label{Tjump41b}
	J_T & = \begin{pmatrix} 1 & 0  & 0 \\
		0 & 1 & 0 \\
		0 &  (-1)^{n+cN} e^{-2n \phi_2}   & 1  \end{pmatrix} \quad 
	\text{on } [-x_2,-x_\gamma) \cup (x_{\gamma},x_2], \\ 
	\label{Tjump5b}
	J_T & = \begin{pmatrix} 1 & 0 & 0 \\
		0 & e^{2n \phi_{2,-}} & 0 \\
		0 & (-1)^{n+cN} & e^{2n \phi_{2,+}}  \end{pmatrix} 
	\quad \text{ on } (-\infty, -x_2] \cup  [x_2,\infty).
\end{align} 
\end{corollary}

\subsection{\texorpdfstring{Third transformation $T \mapsto S$}{}}

In the third transformation, we open up lenses around $\Delta_1 = [-x_1,x_1]$ 
and $\Delta_2 = (-\infty, -x_2] \cup [x_2, \infty)$.
We make sure that the unbounded lenses are big enough towards infinity 
so that
they include a sector $|\arg (z-x_2)| \leq \theta$ for some $\theta > 0$
in the right-half plane and a similar sector in the left-half plane.  
There is some freedom in the choice of the lenses, and we 
impose certain restrictions later on.

We use $L_j$ to denote the lens around $\Delta_j$, for $j=1,2$,
and $L_j^{\pm} = L_j \cap \mathbb C^{\pm}$ denotes the part of the lens in
the upper or lower half-plane.

\begin{figure}[t]
	\centering
	\begin{tikzpicture}[scale=4]
		\begin{scope}[very thick,decoration={
				markings,
				mark=at position 0.5 with {\arrow{>}}}
			] 
			\draw [postaction={decorate}](-0.7,-0) ..controls(-0.2,0.3) and (0.2,0.3) .. (0.7,0);
			
			\draw [postaction={decorate}]  (-0.7,-0) ..controls(-0.2,-0.3) and (0.2,-0.3) .. (0.7,0);
			
			\draw [postaction={decorate}]  (-0.7,-0) ..controls(-0.2,-0.0) and (0.2,-0.0) .. (0.7,0);
			
			\draw [postaction={decorate}] (-1.5,0)--(-1,0);
			\draw[postaction={decorate}] (1,0)--(1.5,0);  
			\draw [postaction={decorate}] (-1.5,-0.2)  ..controls (-1.3,-0.15)and (-1.2,-0.1)  ..(-01,-0) ;
			\draw [postaction={decorate}] (-1.5,0.2)  ..controls (-1.3,0.15)and (-1.2,0.1)  ..(-01,0) ;
			\draw [postaction={decorate}] (01,-0) ..controls(1.2,0.1) and (1.3,0.15) .. (1.5,0.2);
			\draw [postaction={decorate}] (01,-0) ..controls(1.2,-0.1) and (1.3,-0.15) .. (1.5,-0.2);
			\draw (-1,0)--(-0.7,0);
			\draw[postaction={decorate}] (0.7,0)--(1,0);
			\draw (0,0.25)  node[above]{$\partial L_{1}$};
			\draw (0,-0.25)  node[below]{$\partial L_{1}$};
			\draw (-1.3,-0.2)  node[below]{$\partial L_{2}$};
			\draw (-1.3,0.2)  node[above]{$\partial L_{2}$};
			\draw (1.3,-0.2)  node[below]{$\partial L_{2}$};
			\draw (1.3,0.2)  node[above]{$\partial L_{2}$};
			\draw (0,0) node[below]{$\Delta_1$};
			\draw (-1.48,0.02) node[below]{$\Delta_2$};
			\draw (1.48,0.02) node[below]{$\Delta_2$};
			
			\filldraw[black] (-0.68,0) circle (0.5pt);
			\draw (-0.7,0) node[left,above]{$-x_1$};
			\filldraw[black] (0.68,0) circle (0.5pt);
			\draw (0.7,0) 	node[left,above]{$x_1$};
			\filldraw[black] (1.02,0) circle (0.5pt);
		    \draw (1,0) node[left,above]{$x_2$};
			\filldraw[black] (-1.02,0) circle (0.5pt);
			\draw (-1.02,0) node[left,above]{$-x_2$};
			
			\filldraw[black] (-0.9,0) circle (0.5pt);
			\draw (-0.94,0) node[left,below]{$-x_\gamma$};
			
			\filldraw[black] (0.9,0) circle (0.5pt);
			\draw (0.9,0) node[left,below]{$x_\gamma$};

			\draw (0,0.8)  node[above]{$\gamma$};
			\coordinate (a) at(0.9,0);
			\coordinate (b) at (0,0.8);
			\coordinate (c) at (-0.9,0);
			\coordinate (d) at (0,-0.8);
			\coordinate (e) at (0.9,0);               
			\path[draw,use Hobby shortcut,closed=true] 
			(a)..(b)..(c)..(d)..(e); 
			\draw [postaction={decorate}] (0,-0.8) node{};
		\end{scope}
	\end{tikzpicture}
	\caption{$\Sigma_S$ consists of the real line, the contour $\gamma$,
	and the lips of the lenses $\partial L_1$, $\partial L_2$ around $\Delta_1$ and $\Delta_2$. \label{fig:SigmaS}}
\end{figure}
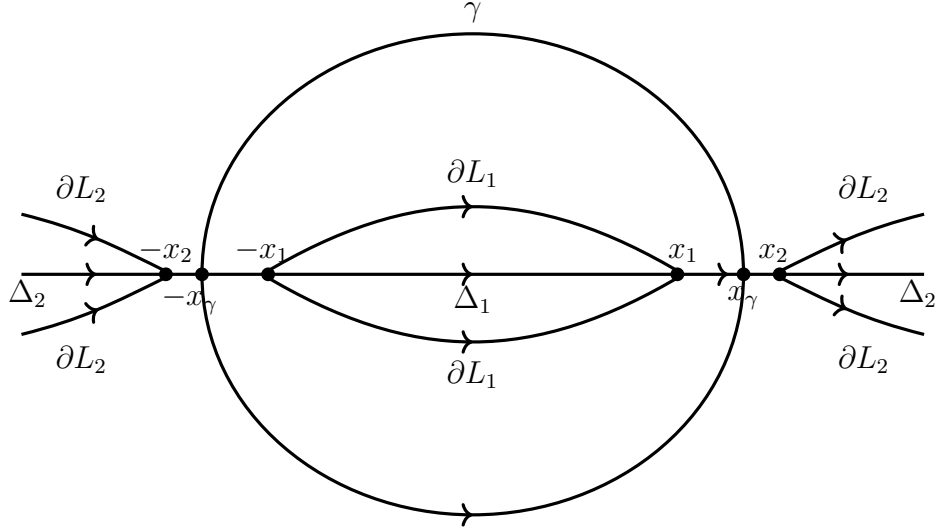

\begin{definition} \label{TtoS}
We define
\begin{equation} \label{Sdef} S = T \times \begin{cases} 
		\begin{pmatrix} 1 & 0 & 0 \\
			\mp e^{2n \phi_1} & 1 & 0 \\
			0 & 0 & 1 \end{pmatrix}, &  \text{ in } L_1^{\pm}, \\
		\begin{pmatrix} 1 & 0 & 0 \\
			0 & 1 & \mp (-1)^{n+cN}  e^{2n \phi_2} \\
			0 & 0 & 1 \end{pmatrix}, &  \text{ in } L_2^{\pm},
		\end{cases} \end{equation}
with $S = T$ outside of the lenses.
 \end{definition}

 \begin{rhproblem}  \label{rhproblemS} $S$ satisfies the following RH problem.
 	\begin{itemize}
 		\item $S : \mathbb C \setminus \Sigma_S 
 		\to \mathbb C^{3\times 3}$
 		is analytic, where $\Sigma_S =\gamma \cup \mathbb R\cup \partial L_1 \cup \partial L_2$
 		with orientation as in Figure \ref{fig:SigmaS}.
 		\item $S_+ = S_- J_S$ on $\Sigma_S$ with  		
 		\begin{align} \label{Sjump1}  %\displaybreak[0]
 			J_S & = \begin{pmatrix} 1 & 0 & - e^{-2n \phi_3} \\
 				0 & 1 & 0 \\
 				0 & 0 & 1 \end{pmatrix} \quad \text{ on } \gamma, \\ \label{Sjump2}
 			J_S & = \begin{pmatrix} 0 & 1 & 0  \\ 
 				-1 & 0  & 0 \\
 				(-1)^{n+cN+1} e^{2n(\phi_{1,+}-\phi_{2,+})} & (-1)^{n+cN} e^{-2n \phi_{2,+}}  & 1  \end{pmatrix} \quad 
 			\text{on } (-x_1,x_1), \\  \label{Sjump3}  \displaybreak[0]
 			J_S & = \begin{pmatrix} 1 & e^{-2n \phi_1} 	& 0 \\
 				0 & 1 & 0 \\
 				0 &  (-1)^{n+cN} e^{-2n \phi_2}   & 1  \end{pmatrix} \quad 
 			\text{on } (-x_{\gamma}, -x_1) \cup (x_1,x_\gamma) \\	  \label{Sjump4}
 			J_S & = \begin{pmatrix} 1 & 0  & 0 \\
 				0 & 1 & 0 \\
 				0 &  (-1)^{n+cN} e^{-2n \phi_2}   & 1  \end{pmatrix} \quad
 			\text{on } (-x_2,-x_\gamma)  \cup (x_\gamma, x_2), \\ \label{Sjump5}
 			J_S & = \begin{pmatrix} 1 & 0 & 0 \\
 				0 &  0 & (-1)^{n+cN+1} \\
 				0 & (-1)^{n+cN} & 0   \end{pmatrix} 
 			 \text{ on } (-\infty, -x_2) \cup (x_2,\infty), \\
  \label{Sjump6}
 		J_S & = \begin{pmatrix} 1 & 0 & 0 \\
 				e^{2n \phi_1} & 1 & 0 \\
 				0 & 0 & 1 \end{pmatrix} \quad \text{ on } \partial L_1, \\ \label{Sjump7}
 		J_S & = \begin{pmatrix} 1 & 0 & 0 \\
 				0 & 1 & (-1)^{n+cN} e^{2n \phi_2}  \\
 				0 & 0 & 1 \end{pmatrix} \quad \text{ on } \partial L_2.
 		\end{align}
 		\item $S$ remains bounded near $\pm x_1$, $\pm x_2$,
 		and $\pm x_{\gamma}$.
 		\item As $z \to \infty$,
 		\begin{equation} \label{Sasymp}
 			S(z)  = \left(I_3 + \mathcal O(z^{-1})\right) S_\infty(z),
 		\end{equation}
 		with 
 		\begin{equation} \label{Sasymp2} S_\infty(z) = 
 			\begin{cases}
 			\begin{pmatrix} 1 & 0 & 0 \\ 0 & 0 & -1 \\ 0 & 1 & 0 \end{pmatrix}, & \Im z > 0, \\
 			\begin{pmatrix} 1 & 0 & 0 \\ 0 & (-1)^{n+cN} & 0 \\ 0 & 0 & (-1)^{n+cN} \end{pmatrix}, & \Im z < 0. 
 			\end{cases}  \end{equation}
 	\end{itemize}
 \end{rhproblem}
 \begin{proof}
 	All jumps follow from direct calculations from the definition \eqref{Sdef} and
 	the jumps for $T$ as written in \eqref{Tjump21b}--\eqref{Tjump5b}.
 	The asymptotic condition \eqref{Sasymp}--\eqref{Sasymp2} outside of the lens $L_2$ is immediate from
 	\eqref{Tasymp}--\eqref{Tasymp2} as the non-constant entries in \eqref{Tasymp2} are
 	exponentially small as $z \to \infty$ outside $L_2$.
 	
 	Inside $L_2$ in the upper-half plane, the definition
 	\eqref{Sdef} together with \eqref{Tasymp}-\eqref{Tasymp2} yields
 	\begin{align} S(z) \label{Sasymp3} 
 			& = \left(I_3 + \mathcal{O}(z^{-1}) \right) 
 			\begin{pmatrix} 1 & 0 & 0 \\ 
 				0 & 0 & -1 \\
 				0 & 1 & -  e^{2n \phi_2(z) + (n+cN) \pi i} + e^{2ia Nz} \end{pmatrix}. 
 			\end{align}
 		For $z \in L_2^+$, $\Re z > 0$, we note that by \eqref{phi2z}
 		\begin{equation} \label{phi2zinL} 
 			2 \phi_2(z) = 2 \phi_{2,+}(|z|) + \frac{1}{t} \int_{|z|}^z (S_2(s) - S_3(s)) ds. 
 		\end{equation}
 		Because of \eqref{g2min} and \eqref{phi21} the first term on the right of \eqref{phi2zinL}
 		is 
 		\begin{align*} 
 			2 \phi_{2,+}(|z|) & = \frac{2ia|z|}{t} - \left(1 + \frac{c}{t}\right) \pi i + 2 \pi i \mu_2([|z|,\infty)) \\
 			& = \frac{2ia|z|}{t} - \left(1 + \frac{c}{t}\right) \pi i + \mathcal{O}(|z|^{-1}) \end{align*}
 		as $z \to \infty$, since $\mu_2$ has a density $\frac{d\mu_2}{dx}$ that decays like 
 		$\mathcal O(x^{-2})$. For the second term, we have by \eqref{S123asymp}
 		\begin{align*} \frac{1}{t} \int_{|z|}^z (S_2(s) - S_3(s)) ds &
 			= \frac{1}{t} \int_{|z|}^z \left(2ia + \mathcal O(s^{-2})\right) ds 
 			 = \frac{2iaz}{t} - \frac{2ia|z|}{t}  + \mathcal O(z^{-1}) \end{align*}
 		as $z \to \infty$.  We conclude 
 		\begin{equation}\label{phi2inf}
 		    2 \phi_2(z) + \left(1+ \frac{c}{t}\right) \pi i - \frac{2iaz}{t}  = \mathcal O(z^{-1}).\end{equation}
 		Then, the $(3,3)$ entry of \eqref{Sasymp3} is
 		\[ e^{2ia Nz} \left(1 - e^{n(2 \phi_2(z)+   (1+ \frac{c}{t}) \pi i - \frac{2ia}{t} z)} \right)
 			= \mathcal O(z^{-1}) \]
 		as $z \to \infty$ uniformly within $L_2^+$ in the right-half plane,
 		and \eqref{Sasymp3} reduces to \eqref{Sasymp2} for $z \in L_2^+$, $\Im z >0$.
 		
 		Similar calculations give us \eqref{Sasymp2} in the other parts of the unbounded
 		lens $L_2$.
   Since $T(z)$ and the transformation in \eqref{Sdef} are bounded near $\pm x_1$,
   $\pm x_2$, and $\pm x_{\gamma}$ so is $S(z)$.
  \end{proof}
 
 \subsection{Global parametrix}
 
 The following lemma shows that the non-constant (in $z$) entries in the
 jump matrices \eqref{Sjump1}--\eqref{Sjump7} tend to zero as $n \to \infty$
 at an exponential rate if we stay away from the branch points $\pm x_1$ and $\pm x_2$.
 
 \begin{lemma} \label{lemma612}
 	The following statements hold true
 	\begin{enumerate}
 		\item[\rm (a)] $\Re \phi_3 > 0$ on $\gamma$,
 		\item[\rm (b)] $\Re \phi_{2,+} > 0$ on $(-x_2,x_2)$,
 		\item[\rm (c)] $\Re \phi_{1,+} = 0$ on $[-x_1,x_1]$, and $\Re \phi_1 > 0$ on
 		$[-x_2,x_1) \cup (x_1,x_2]$,
 		\item[\rm (d)] The lens $L_1$ around $\Delta_1$ can be opened in such a way
 		that $\Re \phi_1 < 0$ on $\partial L_1 \setminus \{\pm x_1\}$,
 		\item[\rm (e)] The lens $L_2$ around $\Delta_2$ can be opened in such a way that it contains a sector $|\arg (z-x_{2})|\leq \theta$ for some $\theta>0$ along with
 		 $\Re \phi_2 < 0$ on $\partial L_2 \setminus \{\pm x_2 \}$.
 	\end{enumerate} 
 \end{lemma}
 \begin{proof}
 	(a) follows from \eqref{phi31} and the inequality \eqref{mu12cond6} from item \ref{item5}  in Theorem \ref{thm:mu12}.
 	
 	(b) is a direct consequence of \eqref{phi23} and \eqref{mu12cond5} in item \ref{item4} 
 	of Theorem \ref{thm:mu12}. 
 	
 	(c) By \eqref{phi11} and \eqref{g1min} we have 
 	\begin{equation} \label{phi1onDelta1} 
 		- \phi_{1,+}(x) = \pi i \mu_1([x,x_1]) \in i \mathbb R^+,
 		\quad \text{for } x \in [-x_1,x_1],
 		\end{equation}   
 	and \eqref{phi12} and \eqref{phi13} imply
 	\[ 2 \Re \phi_1 = - \Re \left( V_1 - 2 U^{\mu_1} + U^{\mu_2} + \ell \right) 
 		 \quad \text{on } [-x_2, -x_1] \cup [x_1,x_2] \]  
 	which is positive on $[-x_2,-x_1) \cup (x_1,x_2]$ due to \eqref{mu12cond4}
 	in item \ref{item3}  of Theorem~\ref{thm:mu12}.

 	(d) The identity \eqref{phi1onDelta1} also shows that
 	$\frac{\partial }{\partial x} \Im \phi_{1,+}(x) > 0$ for $x \in (-x_1,x_1)$.
 	The Cauchy-Riemann equations, then, yield
 	\[ \lim_{y \to 0+} \frac{\partial}{\partial y} \Re \phi_1(x+iy) < 0,
 		\quad x \in (-x_1,x_1). \]
 	Since $\Re \phi_{1,+}(x) = 0$, it follows that $\Re \phi_1 < 0$ in some
 	region above
 	the interval $(-x_1,x_1)$ in the complex plane. In a similar way, $\Re \phi_1 $ is strictly negative below the interval. We open up the
 	lens $L_1$ within these regions and (d) holds. 
 
 	(e) 
 	By \eqref{phi21}--\eqref{phi25} and \eqref{g2min} we have
 	\[ 2 \phi_{2,+}(x) = 2 \pi i \mu_2([x,\infty)) + \frac{2iax}{t} - \pi i \left(1+\frac{c}{t}\right), \]
 	for $x \in \Delta_2$. Thus for $x \in (-\infty,-x_2) \cup (x_2, \infty)$
 	\[  \frac{\partial}{\partial x} \Im  \phi_{2,+}(x)
 		=  -\pi \frac{d\mu_2}{dx} + \frac{a}{t} > 0.  \]
 	The strict inequality holds since the density of $\mu_2$ is strictly smaller than $ \frac{a}{\pi t}$ on  
 	  $(-\infty,-x_2) \cup (x_2, \infty)$ by item \ref{item2}  in Theorem \ref{thm:mu12}.
 	Anew by the Cauchy-Riemann equations, similar to the reasoning in the proof of (d), will imply that $\Re \phi_2 < 0$ on $\partial L_2 \setminus \{ \pm x_2 \}$
 	provided the lens is opened close enough to $\Delta_2$. 
  
  Now we ensure the first condition that the lens $\partial L_{2}$ contains a sector $|\arg (z-x_{2})|\leq \theta$ for some $\theta>0$ along with the condition that $\Re \phi_{2}<0$ on $\partial L_{2}$.
  On the upper half plane  we have  $\Re \phi_{2}<0$ as $|z|\to\infty$, due to \eqref{phi2inf}. Hence, for large enough $R_1$ we have,
  \begin{equation}\label{phi2large}
      \Re \phi_{2}<0 \qquad \mathrm{when} \hspace{0.2cm}|z|>R_1.
  \end{equation}
   Our method of choosing the lens $\partial L_{2}$ is as follows. If $|z|<R_1$, we open our lens $\partial L_{2}$ close enough to $\Delta_{2}$ such that $\Re \phi_{2}<0$. For $|z|>R_1$ we deform our lens such that   $ L_{2}$ contains a sector $|\arg (z-x_{2})|\leq \theta$, for some $\theta>0$. $\Re \phi_{2}<0$ is still satisfied on $\partial L_{2}$ due to \eqref{phi2large}.
  The same construction can be carried out in the lower half-plane.
 	 \end{proof}
 
 For the global parametrix, we forget about the non-constant entries in $J_S$.
Therefore we ask for a matrix valued function $M$ satisfying the following.
 
\begin{rhproblem}  \label{rhproblemM}
 	We are looking for $M$ satisfying
 	\begin{itemize}
 		\item $M : \mathbb C  \setminus \Sigma_M \to  \mathbb C^{3\times 3}$
 		is analytic,  		where $\Sigma_M = \Delta_1 \cup \Delta_2$.
 		\item $M_+ = M_- J_M$ on $\Sigma_M$ with 
 		\begin{align} \label{Mjump1} \displaybreak[0]
 			J_M & = \begin{pmatrix} 0 & 1 & 0  \\ 
 				-1 & 0  & 0 \\
 				0 & 0  & 1  \end{pmatrix} \quad 
 			\text{on } \Delta_1 = (-x_1,x_1), \\  \label{Mjump2}
 			J_M & = \begin{pmatrix} 1 & 0 & 0 \\
 				0 &  0 & (-1)^{n+cN+1} \\
 				0 & (-1)^{n+cN} & 0   \end{pmatrix} 
 			\quad \text{ on } \Delta_2 = (-\infty, -x_2) \cup (x_2,\infty).
 		\end{align}
 		\item As $z \to x^*$ with $x^* \in \{ \pm x_1, \pm x_2\}$,
 		\begin{equation} \label{Mnearbranch} 
 			M(z) = \mathcal{O}\left((z-x^*)^{-1/4} \right).
 		\end{equation} 
 		\item As $z \to \infty$,
 		\begin{equation} \label{Masymp}
 			M(z)  = \left(I_3 + \mathcal O(z^{-1})\right) M_\infty(z) 
 		\end{equation}
 		with $M_\infty(z) = S_{\infty}(z)$, see \eqref{Sasymp2}.
 	\end{itemize}
 \end{rhproblem}
 
It turns out that $M$ cannot remain bounded near $\pm x_1$ and $\pm x_2$, as is the
 case for $S$, and the behavior \eqref{Mnearbranch} with exponent $-1/4$ is
 the weakest kind of singularity that will work out as we will see. The notation \eqref{Mnearbranch} should be understood entrywise,
 but actually the entries in the third column of $M$ remain bounded near $\pm x_1$,
 and the entries in the first column remain bounded near $\pm x_2$.
 \begin{figure}[t]
	\centering
	\includegraphics[trim={1cm 0.5cm 1.5 0.7cm}, clip, scale=0.5]{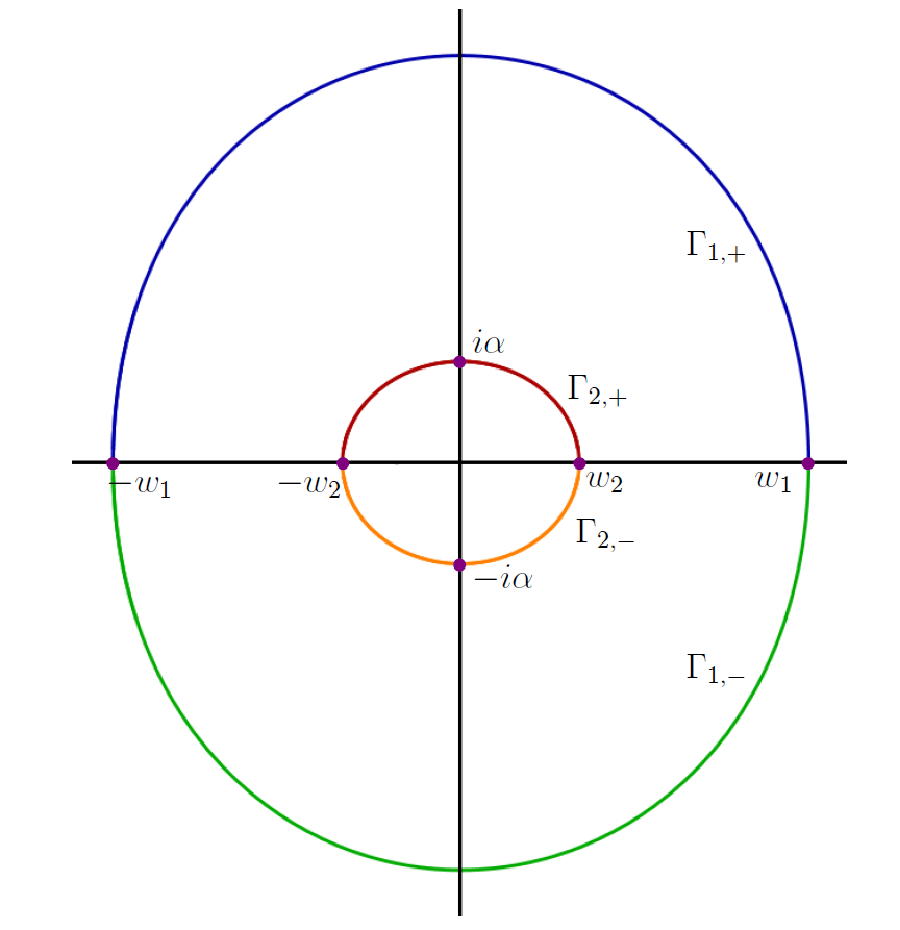}
	\caption{The contours $\Gamma_{1,\pm}$ and  $\Gamma_{2,\pm}$
	\label{cdcontour}}
\end{figure}
 The asymptotic condition \eqref{Masymp} is compatible with the jump \eqref{Mjump2} since
 \[ \begin{pmatrix} 1 & 0 & 0 \\ 0 & 0 & -1 \\ 0 & 1 & 0 \end{pmatrix}
 	= \begin{pmatrix} 1 & 0 & 0 \\ 0 & (-1)^{n+cN} & 0 \\ 0 & 0 & (-1)^{n+cN} \end{pmatrix}
 		\begin{pmatrix} 1 & 0 & 0 \\ 0 & 0 & (-1)^{n+cN+1} \\ 0 & (-1)^{n+cN} & 0 \end{pmatrix}. \]
 	
We find $M$ by means of a conformal map $F : \mathcal R \to \mathbb C \cup \{\infty\}$ to the Riemann sphere, that we take as the inverse of the
rational function $f$ from \eqref{conformalmap},  which by Remark \ref{fconformaltoR} 
is the conformal map from the Riemann sphere to $\mathcal R$.
Let $F_k$ denote the restriction of $F$ to sheet $\mathcal R_k$, for $k=1,2,3$, and
\begin{align*} 
	\Gamma_{1}^{\pm} & = F_{1,\pm}(\Delta_1), \qquad \Gamma_1 = \Gamma_1^+ \cup \Gamma_1^-, \\
	\Gamma_{2}^{\pm} & = F_{3,\pm}(\Delta_2), \qquad \Gamma_2 = \Gamma_2^+ \cup \Gamma_2^-.
	\end{align*}
For $j=1,2$, it holds that $\Gamma_j^+$ is the part of $\Gamma_j$ in the upper half-plane, while $\Gamma_j^-$ is the part in the lower half plane, see Figure \ref{cdcontour}.
 $\Gamma_1$ and $\Gamma_2$ are two simple closed contours in the complex plane,
with $\Gamma_2$ lying inside $\Gamma_1$. The sheet $\mathcal{ R}_1$ is mapped
to the  exterior of $\Gamma_1$, the sheet $\mathcal{R}_2$ is mapped
to the annular region bounded by $\Gamma_1$ and $\Gamma_2$, and $\mathcal{R}_3$
is mapped to the bounded region enclosed by $\Gamma_2$. Let us denote the
endpoints of $\Gamma_j^{\pm}$ by $\pm w_j$ for $j=1,2$, meaning  $w_1 = F_1(x_1) = F_2(x_1)$, 
$w_2 = F_2(x_2) = F_3(x_2)$, $0 < w_2 < w_1$, and
$\Gamma_2$ lies inside $\Gamma_1$. 
Also note that $\pm i \alpha \in \Gamma_2$, since these are
the poles of $f$, see \eqref{conformalmap}. 

\begin{lemma} \label{Mlemma}
	Let  
	\begin{equation} \label{Rwdef} 
		h(w) = \left((w^2-w_1^2)(w^2-w_2^2) \right)^{1/2} \end{equation}
	be defined and analytic on $\mathbb C \setminus (\Gamma_1^+ \cup \Gamma_2^+)$
	in case $n+cN$ is even, and on $\mathbb C \setminus (\Gamma_1^+ \cup \Gamma_2^-)$
	in case $n+cN$ is odd, and in both cases $h(w) \sim w^2$ as $w \to \infty$.  
	Then, the RH problem \ref{rhproblemM} has the solution 
	\begin{equation} \label{Mzdef} 
		M(z) = \begin{pmatrix} M_j(F_k(z)) \end{pmatrix}_{j,k=1,2,3}
	\end{equation}
	with scalar functions $M_j$ for $j=1,2,3$, defined by 
	\begin{align} \label{M1wdef}
		M_1(w) & = \frac{w^2 + \alpha^2}{h(w)}, \\ \label{M2wdef}
		M_2(w) & = -\frac{h_-(i \alpha)}{2i \alpha} \frac{w+i\alpha}{h(w)}, \\ \label{M3wdef}
		M_3(w) & = -\frac{h_+(-i\alpha)}{2i\alpha} \frac{w-i\alpha}{h(w)}. 
	\end{align}
	
	In particular, the $(1,1)$ entry satisfies
	\begin{equation} \label{M11z}
			M_{11}(z) = M_1(F_1(z)) = \left(\rho F_1'(z) \right)^{1/2}, 
			\quad z \in \mathbb C \setminus [-x_1, x_1],
		\end{equation}
	where $\rho > 0$ is such that $\rho F_1'(z) \to 1$ as $z \to \infty$, see also 
	\eqref{conformalmap},
	and the branch of the square root is taken in \eqref{M11z} 
	that is real and positive on $(x_1, \infty)$.	 
\end{lemma}  
\begin{proof}
	Consider the form \eqref{Mzdef} of $M$ with scalar valued functions $M_j$
	that are analytic in $\mathbb C \setminus (\Gamma_1 \cup \Gamma_2)$.
	Given the mapping properties of 
	the conformal map $F$ and its branches $F_k$ for $k=1,2,3$, one sees that
	the jump condition \eqref{Mjump1}  is satisfied if and only if
	(we use counterclockwise orientation on both $\Gamma_1$ and $\Gamma_2$),
	\begin{equation} \label{Mjjump1} 
		M_{j,+} = - M_{j,-} \ \text{ on } \Gamma_1^+,
	\quad M_{j,+} = M_{j,-} \ \text{ on } \Gamma_1^-, \quad j=1,2,3, \end{equation}
	while \eqref{Mjump2} is satisfied if and only if
		\begin{equation} \label{Mjjump2} M_{j,+} = (-1)^{n+cN+1} M_{j,-} \ \text{ on } \Gamma_2^+, 
			\quad
		M_{j,+} = (-1)^{n+cN} M_{j,-} \ \text{ on } \Gamma_2^-, \quad j=1,2,3. \end{equation}
	The asymptotic condition \eqref{Masymp} is satisfied
	if and only if (recall that $\pm i\alpha$ are the images of the points at infinity
	on the second and third sheets of $\mathcal R$) 	
	\begin{align} \label{M1asymp}
		 M_1(\infty) & = 1, & M_{1,\pm}(i\alpha) & = 0,
		 &  M_{1,\pm}(-i\alpha) & = 0, \\ \label{M2asymp}
		 M_2(\infty) & = 0, & M_{2,+}(i\alpha) & = (-1)^{n+cN},
		 &  M_{2,\pm}(-i\alpha) & = 0, \\ \nonumber
		&& M_{2,-}(i\alpha) & = -1, && \\ \label{M3asymp}
		 M_3(\infty) & = 0, & M_{3,\pm}(i\alpha) & = 0,
		 &  M_{3,+}(-i\alpha) & = 1, \\
		 &&&& M_{3,-}(-i\alpha) & = (-1)^{n+cN}. \nonumber
		\end{align}
	Here $M_{j,+}(\pm i\alpha)$ in \eqref{M1asymp}--\eqref{M3asymp} denotes the limiting value of $M_j(w)$ as $w \to \pm  i\alpha$
	from inside of $\Gamma_2$, and $M_{j,-}(\pm i\alpha)$ the limiting value from outside
	of $\Gamma_2$ (which is consistent with the counterclockwise orientation of $\Gamma_2$).
	 	 
	The choice of the branch cut for $h$ in \eqref{Rwdef} is such that
	any function of the form $M_j = \frac{Q_j}{h}$ with a polynomial $Q_j$ 
	will satisfy the required jumps \eqref{Mjjump1} and \eqref{Mjjump2}.
	The polynomials $Q_j$ are then determined so that the asymptotic conditions
	\eqref{M1asymp}--\eqref{M3asymp} hold. From \eqref{M1asymp} we get that
	$Q_1$ is monic of degree $2$
	with zeros in $\pm i\alpha$ and we obtain \eqref{M1wdef}. 
	From \eqref{M2asymp} we get that $Q_2$ has degree one, with a zero in $-i\alpha$.
	Hence $Q_2(w) =  C(w + i\alpha)$ with a constant $C$ that should be such that
	$C \frac{w+i\alpha}{h(w)} \to -1$ as $ w \to i\alpha$  from outside $\Gamma_2$.
	Thus, $C =  - \frac{h_-(i\alpha)}{2i\alpha}$ and \eqref{M2wdef} follows.
	The formula \eqref{M3wdef} follows in a similar fashion.
	
	The construction of $M$ also shows that near the branch points $\pm x_1$,
	$\pm x_2$ one has 
	\begin{equation} \label{Mbranch}
		M(z)  = \begin{cases} \mathcal{O} \begin{pmatrix} 
				(z \mp x_1)^{-1/4} & (z \mp x_1)^{-1/4} & 1 \\
				(z \mp x_1)^{-1/4} & (z \mp x_1)^{-1/4} & 1 \\
				(z \mp x_1)^{-1/4} & (z \mp x_1)^{-1/4} & 1 
			\end{pmatrix}  & \text{as } z \to \pm x_1, \\
			\mathcal{O} \begin{pmatrix} 
				1 &	(z \mp x_2)^{-1/4} & (z \mp x_2)^{-1/4}  \\
				1 & (z \mp x_2)^{-1/4} & (z \mp x_2)^{-1/4}  \\
				1 & (z \mp x_2)^{-1/4} & (z \mp x_2)^{-1/4}  
			\end{pmatrix} & \text{as } z \to \pm x_2, \end{cases} \end{equation}
	with the $\mathcal{O}$ being taken entrywise.
	
	\medskip
	Finally, we prove \eqref{M11z}. The endpoints  $\pm w_1, \pm w_2$ of 
	$\Gamma_1^{\pm}$ and $\Gamma_2^{\pm}$ 
	are the critical points of the inverse map \eqref{conformalmap}, i.e., the
	zeros of its derivative. Because of the form \eqref{conformalmap} we have
	\[ \left(F^{-1}\right)'(w) = f'(w)  = \rho \frac{(w^2-w_1^2)(w^2-w_2^2)}{(w^2 + \alpha^2)^2}. \]
	Then, if $w = F_1(z)$, we have by the rule for the derivative of the inverse function
	\[ \rho F_1'(z) = \frac{(w^2 + \alpha^2)^2}{(w^2-w_1^2)(w^2-w_2^2)} = M_1(w)^2
		= M_1(F_1(z))^2, \]
	see \eqref{M1wdef} and \eqref{Rwdef}.  
	The identity \eqref{M11z} follows.
\end{proof}

\subsection{Local parametrices}

The local parametrices are defined in small neighborhoods of the 
branch points $\pm x_1$ and $\pm x_2$. 
For definiteness we take  
\begin{equation} \label{defD} D =  D(x_1,\delta) \cup D(-x_1,\delta) \cup D(x_2,\delta) \cup D(-x_2,\delta) \end{equation}
as the union of four disks of radius $\delta > 0$, but
we could choose other small neighborhoods.
For $\delta > 0$ sufficiently small, the local parametrix $P$
is defined in $D$ with jump matrices that agree with $J_S$ up to 
a number of exponentially small entries. In addition, $P$ should agree with $M$ 
on the boundary of the disks up to an error $\mathcal O(n^{-1})$
as $n \to \infty$. It leads to the following RH problem for $P$.

\begin{rhproblem} \label{rhproblemP}
	$P$ satisfies the following.
	
	\begin{itemize}
		\item $P : D \setminus \Sigma_P \to \mathbb C^{3 \times 3}$ is analytic,
		where $\Sigma_P = \Sigma_S \cap D$ (recall $\Sigma_S$ from \ref{rhproblemS}).
		\item $P_+ = P_- J_P$ on $\Sigma_P$ with
		\begin{align} \label{JP1}
			J_P  = \begin{cases} \begin{pmatrix} 0 & 1 & 0  \\ 
				-1 & 0  & 0 \\
				0 & 0  & 1  \end{pmatrix}  & \text{ on } [-x_1,-x_1+\delta) \cup (x_1-\delta, x_1], \\
			\begin{pmatrix} 1 & e^{-2n \phi_1} 	& 0 \\
					0 & 1 & 0 \\
					0 &  0  & 1  \end{pmatrix} & \text{ on } [-x_1-\delta,-x_1] \cup [x_1,x_1+\delta), \\
			\begin{pmatrix} 1 & 0 & 0 \\
		e^{2n \phi_1} & 1 & 0 \\
		0 & 0 & 1 \end{pmatrix} & \text{ on } \partial L_1 \cap (D(-x_1,\delta) \cup D(x_1,\delta)),
		\end{cases}
		\end{align}
		\begin{align} \label{JP3}
			J_P = \begin{cases} \begin{pmatrix} 1 & 0 & 0  \\ 
				0 & 0  & (-1)^{n+cN+1} \\
				0 & (-1)^{n+cN} & 0  \end{pmatrix} & \text{ on } (-x_2-\delta,-x_2] \cup [x_2, x_2 + \delta), \\
			\begin{pmatrix} 1 & 0 & 0 \\
				0 & 1 & 0 \\
				0 & (-1)^{n+cN} e^{-2n \phi_2} & 1  \end{pmatrix} &  
				\text{ on } [-x_2,-x_2+\delta) \cup (x_2-\delta,x_2], \\
			\begin{pmatrix} 1 & 0 & 0 \\
				0 & 1 & (-1)^{n+cN} 	e^{2n \phi_2} \\ 
				0 & 0 & 1 \end{pmatrix} & \text{ on } \partial L_2 \cap (D(-x_2,\delta) \cup  D(x_2,\delta)),
				\end{cases}
		\end{align}
	\item $P$ matches with the global parametrix $M$ in the sense that
	\begin{equation} \label{Pmatching}
		P(z) = \left(I_3 + \mathcal O(n^{-1}) \right) M(z) \quad \text{ as } n \to \infty, \end{equation}
	uniformly for $z \in \partial D$.
	\end{itemize}
\end{rhproblem}
   The construction of $P$ is done in each disk separately. In each disk, we have nontrivial jumps in a $2 \times 2$ block, only. These jumps can be mapped to the jumps in the usual Airy parametrix, and  the
   RH problem for $P$ can be solved with Airy functions. For the matching condition
   \eqref{Pmatching} it is important that  the functions $\phi_j$ for $j=1,2$, vanish at $x_j$ with an exponent $3/2$, and their $2/3$ powers are conformal maps in the corresponding disk 
   (provided $\delta > 0$ is small   enough.)
   At $- x_j$ it will be convenient to replace $\phi_j$ by
   \begin{align} \label{phitilde1z}
   	\widetilde{\phi}_1(z) & = \frac{1}{2t} \int_{-x_1}^z (S_2(s) - S_1(s)) ds, \\
   	\widetilde{\phi}_2(z) & = \frac{1}{2t} \int_{-x_2}^z (S_2(s) - S_3(s)) ds.
   \end{align} 	 
   Then, it is $\widetilde{\phi}_j = \phi_j \pm 2 \pi i$, and changing $\phi_j$ to $\widetilde{\phi}_j$
   does not change any of the jumps. However, $\widetilde{\phi}_j$ does vanish at $-x_j$ with exponent $3/2$,
   and its $2/3$ power gives a conformal map.
    
   The lenses will be opened such that the boundaries of the lenses
   are mapped by the conformal map to rays with argument either $\pm 2 \pi /3$ (in case of $x_1$ and $-x_{2}$)
   or $\pm \pi/3$ (in case of $x_2$ and $-x_{1}$).

	The standard Airy parametrix is given by ($\Ai$ denotes the Airy function and $\omega = 2 \pi i/3$)
\begin{equation} \label{airyparametrix} 
	A(z)= \sqrt{2\pi} \times \begin{cases} \begin{pmatrix}
		\Ai(z)&-\omega^{2}\Ai(\omega^{2}z)\\
		-i\mathrm{Ai'}(z) &i\omega \Ai'(\omega^{2}z)    \end{pmatrix}, 
		&  0 < \arg z <2\pi/3, \\
		\begin{pmatrix}
		-\omega\Ai(\omega z)&-\omega^{2}\Ai(\omega^{2}z)\\
		i\omega^{2}\Ai'(\omega z) &i\omega \Ai'(\omega^{2}z)    \end{pmatrix},
		&  2\pi/3< \arg z <\pi, \\
		 \begin{pmatrix}
		-\omega^{2}\Ai(\omega^{2} z)&\omega\Ai(\omega z)\\
		i\omega\Ai'(\omega^{2} z) &-i\omega^{2} \Ai'(\omega z) \end{pmatrix},
		& -\pi< \arg z <-2\pi/3, \\
	\begin{pmatrix}
		\Ai(z)&\omega\Ai(\omega z)\\
		-i\Ai'(z) &-i\omega^{2} \Ai'(\omega z) \end{pmatrix},
		&  -2\pi/3< \arg z<0. \end{cases}
\end{equation}
The prefactor $\sqrt{2\pi}$ is not essential, but we include it in the definition since in this way $\det A \equiv 1$.
The real line is oriented from left to right, and the two rays $\arg z = \pm 2\pi/3$
are oriented towards the origin. This gives  $A_+ = A_- J_A$ with piecewise
constant jump matrices
\begin{equation} \label{jumpsJA}
	 J_A = \begin{cases} 
	 	\begin{pmatrix} 1 & 1 \\ 0 & 1 \end{pmatrix}, &  \arg z = 0, \\
	 	\begin{pmatrix} 0 & 1 \\ -1 & 0 \end{pmatrix}, &  \arg z = \pi, \\
	 	\begin{pmatrix} 1 & 0 \\ 1 & 1 \end{pmatrix}, & \arg z = \pm 2\pi/3. \end{cases} 	 
\end{equation}

In the following lemma, we use $	\sigma_1 = \begin{pmatrix} 0 & 1 \\ 1 & 0 \end{pmatrix}$ and
$\sigma_3 = \begin{pmatrix} 1 & 0 \\ 0 & -1 \end{pmatrix}$ the first and the third Pauli matrix.

\begin{lemma}
	For $z \in D$, let
	\[ A_n(z) = \begin{cases}  A\left( \frac{2}{3} n^{2/3} \phi_1(z)^{3/2}\right),
		&  z \in D(x_1,\delta), \\
		\sigma_3 A\left( \frac{2}{3} n^{2/3} \widetilde{\phi}_1(z)^{3/2}\right) \sigma_3,
		& z \in D(-x_1,\delta), \\
		\sigma_3^{n+cN-1} \sigma_1 A\left( \frac{2}{3} n^{2/3} \phi_2(z)^{3/2}\right)
		\sigma_1 \sigma_3^{n+cN-1},
		& z \in D(x_2,\delta), \\
		\sigma_3^{n+cN}  \sigma_1 A\left( \frac{2}{3} n^{2/3} \widetilde{\phi}_2(z)^{3/2}\right) 
		\sigma_1 \sigma_3^{n+cN},
		& z \in D(-x_2,\delta). 		
		\end{cases} \]  
	
	Then, there is an analytic prefactor $E_n : D \to \mathbb C^{3 \times 3}$ such that
	\begin{align} \label{Pdef} P(z) =  E_n(z) \times \begin{cases}
		\begin{pmatrix} A_{n}(z)  &  \begin{matrix} 0  \\ 0 \end{matrix} \\
			\begin{matrix} 0 & 0 \end{matrix} & 1 \end{pmatrix}
		\begin{pmatrix} e^{n \phi_1(z)} & 0 & 0 \\ 0 & e^{-n \phi_1(z)} & 0 \\
			0 & 0 & 1 \end{pmatrix}, & z \in D(x_1,\delta), \\
		\begin{pmatrix} A_{n}(z)  &  \begin{matrix} 0  \\ 0 \end{matrix} \\
			\begin{matrix} 0 & 0 \end{matrix} & 1 \end{pmatrix}
		\begin{pmatrix} e^{n \widetilde{\phi}_1(z)} & 0 & 0 \\ 0 & e^{-n \widetilde{\phi}_1(z)} & 0 \\
			0 & 0 & 1 \end{pmatrix}, & z \in D(-x_1,\delta), \\
		\begin{pmatrix}  1  &  \begin{matrix} 0  & 0 \end{matrix} \\
				\begin{matrix} 0 \\ 0 \end{matrix} & A_{n}(z) \end{pmatrix}
			\begin{pmatrix} 1 & 0 & 0 \\ 0 & e^{-n \phi_2(z)} & 0  \\ 0 & 0 & e^{n \phi_2(z)} 
			\end{pmatrix}, & z \in D(x_2,\delta), \\
		\begin{pmatrix}  1  &  \begin{matrix} 0  & 0 \end{matrix} \\
			\begin{matrix} 0 \\ 0 \end{matrix} & A_{n}(z) \end{pmatrix}
		\begin{pmatrix} 1 & 0 & 0 \\ 0 & e^{-n \widetilde{\phi}_2(z)} & 0  \\ 0 & 0 & e^{n \widetilde{\phi}_2(z)} 
		\end{pmatrix}, & z \in D(-x_2,\delta).
			\end{cases}
	\end{align}
	satisfies the RH problem \ref{rhproblemP}.
	\end{lemma}
\begin{proof} 
	$z \mapsto \frac{2}{3} \phi_1(z)^{3/2}$ is a conformal map on $D(x_1,\delta)$ that
	is real and increasing on $(x_1-\delta, x_1+\delta)$, mapping $x_1$ to $0$, and
	 $\partial L_1 \cap D(x_1, \delta)$ into the rays $ \arg z = \pm 2\pi/3$.
	Inside $D(x_1,\delta)$, $A_n(z)$ has the jumps \eqref{jumpsJA} but 
	on $(-x_1-\delta, x_1+\delta)$ and $\partial L_1 \cap D(x_1,\delta)$. 
	Because of the extra diagonal factor
	in the definition \eqref{Pdef} of $P$ inside $D(x_1,\delta)$, we get the jumps \eqref{JP1}.
	On $(x_1-\delta, x_1]$ we also  use $\phi_{1,+} = -\phi_{1,-}$,  see \eqref{phi11}.
	The prefactor $E_n$ has no influence on the jumps but it should be chosen such
	that the matching condition \eqref{Pmatching}  holds. It will 
	take the form
	\begin{equation} \label{Enx1} E_n(z) = M(z) \begin{pmatrix} \frac{1}{\sqrt{2}} & \frac{i}{\sqrt{2}} & 0 \\
		\frac{i}{\sqrt{2}} & \frac{1}{\sqrt{2}} & 0 \\
		0 & 0 & 1 \end{pmatrix}
		\begin{pmatrix}  \left( \frac{2}{3} n^{2/3} \phi_1(z)^{3/2} \right)^{1/4} &
			0 & 0 \\
			0 & \left( \frac{2}{3} n^{2/3} \phi_1(z)^{3/2} \right)^{-1/4} & 0 \\
			0 & 0 & 1 \end{pmatrix},   
	\end{equation}
	for $z \in D(x_1,\delta)$, which is based on the asymptotics of the Airy parametrix
	at infinity. The fact that $E_n$ is analytic is also based on the
behavior \eqref{Mbranch} of $M$ near $x_1$.
	
	The reasoning for the other disks is similar, except that one has to take care
of a possible reversal of the orientation (in the case of $D(-x_1,\delta)$ and $D(x_2,\delta)$)
	and a possible opposite triangularity (in case of $D(\pm x_2,\delta)$)
	of the triangular jump matrices. 
	
	For example, inside $D(x_2,\delta)$ the mapping $z \mapsto \frac{2}{3} \phi_2(z)^{3/2}$
is conformal, but is real and decreasing on $(x_2-\delta, x_2+\delta)$. It implies
	that $A_n$ has the jump matrices $\sigma_3^{n+cN-1} \sigma_1 J_A^{-1} \sigma_1 \sigma_3^{n+cN-1}$ 
	with matrices $J_A$ as in \eqref{jumpsJA} (the inverse is there because of the reversed orientation and the conjugation by $\sigma_1$ gives the opposite triangularity).
	Yet distributed  over the contours inside $D(x_2,\delta)$ as follows
	\begin{align*} 
		\sigma_3^{n+cN-1} \sigma_1 \begin{pmatrix} 1 & 1 \\ 0 & 1 \end{pmatrix}^{-1} \sigma_1 \sigma_3^{n+cN-1}
	   & = \begin{pmatrix} 1 & 0 \\ (-1)^{n+cN} & 1 \end{pmatrix}  
		  	\text{ on } (x_2-\delta, x_2], \\
		\sigma_3^{n+cN-1} \sigma_1 \begin{pmatrix} 0 & 1 \\ -1 & 0 \end{pmatrix}^{-1} \sigma_1 \sigma_3^{n+cN-1} & =
		\begin{pmatrix} 0 & (-1)^{n+cN+1} \\ (-1)^{n+cN} & 0 \end{pmatrix} 
		\text{ on } [x_2,x+2+\delta), \\
		\sigma_3^{n+cN-1} \sigma_1 \begin{pmatrix} 1 & 0 \\ 1 & 1 \end{pmatrix}^{-1} \sigma_1 \sigma_3^{n+cN-1} & =
		\begin{pmatrix} 1 & (-1)^{n+cN} \\ 0 & 1 \end{pmatrix} 
		\text{ on } \partial L_2 \cap D(x_2,\delta).
	  \end{align*}
	  The diagonal factor in the definition \eqref{Pdef} of $P$ inside $D(x_2,\delta)$ has the functions $e^{\pm n \phi_2}$.
	  In combination with the above and the fact that $\phi_{2,+} = - \phi_{2,-}$ on
	  $[x_2,x_2+\delta)$ we then get the required jump matrices 
	  \eqref{JP3} for $P$ inside the disk $D(x_2,\delta)$.
	  An appropriate prefactor $E_n$, similar to the one \eqref{Enx1} that we had in $D(x_1,\delta)$,
	  will also provide the matching \eqref{Pmatching} on the circle 
	  $\partial D(x_2,\delta)$.
  For the disc $D(x_{2},\delta)$ due to opposite triangularity, i.e., $E_{n}(z)$ takes the form
  \begin{multline} \label{Enx2} 
  	E_n(z) = M(z) \begin{pmatrix}  1  &  \begin{matrix} 0  & 0 \end{matrix} \\
			\begin{matrix} 0 \\ 0 \end{matrix} & \sigma_{3}^{n+cN-1} \end{pmatrix}\begin{pmatrix}1&0&0\\
   0&\frac{1}{\sqrt{2}} & \frac{i}{\sqrt{2}}  \\
		0&\frac{i}{\sqrt{2}} & \frac{1}{\sqrt{2}}  \\
		 \end{pmatrix}\\ \times
		\begin{pmatrix}1&0&0\\
  0&\left( \frac{2}{3} n^{2/3} \phi_{2}(z)^{3/2} \right)^{-1/4} &0 \\
			0&0 & \left( \frac{2}{3} n^{2/3} \phi_2(z)^{3/2} \right)^{1/4} 
			\end{pmatrix}\begin{pmatrix}  1  &  \begin{matrix} 0  & 0 \end{matrix} \\
			\begin{matrix} 0 \\ 0 \end{matrix} & \sigma_{3}^{n+cN-1} \end{pmatrix}.
   \end{multline}
 The proof for the other two disks $D(-x_1,\delta)$ and $D(-x_2,\delta)$ follows by symmetry.  	
\end{proof}

\subsection{Final transformation}

Having $M$ which is close to $S$ outside $D$ and $P$ which is close to $S$ on $D$, we are ready for the final transformation $S \mapsto R$. In the end we will show $R$ is close  to the Identity matrix in the sense of Corollary \ref{approxid}.

\begin{figure}
    \centering
   \begin{tikzpicture}[scale=4]
         \begin{scope}[very thick,decoration={
    markings,
    mark=at position 0.5 with {\arrow{>}}}
    ] 
   
   \draw [postaction={decorate}] (-0.6,0.05)..controls(-0.2,0.3) and (0.2,0.3) .. (0.6,0.05);

    \draw[postaction={decorate}]  (-0.6,-0.05) ..controls(-0.2,-0.3) and (0.2,-0.3) .. (0.6,-0.05);

  \draw [postaction={decorate}]  (-1.5,-0.25) ..controls (-1.45,-0.24)and (-1.2,-0.19)  .. (-01.030,-0.06);
       \draw  [postaction={decorate}]  (-1.5,0.25) ..controls (-1.45,0.24)and (-1.2,0.19)  .. (-01.030,0.06);
      \draw   [postaction={decorate}](01.03,0.06) ..controls(1.2,0.19) and (1.45,0.24) .. (1.5,0.25);
       
      \draw  [postaction={decorate}] (1.03,-0.06) ..controls(1.2,-0.19) and (1.45,-0.24) .. (1.5,-0.25);
    \draw [postaction={decorate}] (-0.93,0)--(0.93,0);
        
         \draw (0,0.25)  node[above]{$\partial L_{1}$};
             \draw (0,-0.25)  node[below]{$\partial L_{1}$};
                    \draw (-1.3,-0.22)  node[below]{$\partial L_{2}$};
                     \draw (-1.3,0.22)  node[above]{$\partial L_{2}$};
                      \draw (1.3,-0.22)  node[below]{$\partial L_{2}$};
                       \draw (1.3,0.22)  node[above]{$\partial L_{2}$};
          \draw(-0.65,0) circle (2pt);
          \draw  (0.65,0) circle (2pt);
          \draw (1,0,0) circle (2pt);
          \draw (-1,0,0) circle (2pt);
          \draw (0,0.8)  node[above]{$\gamma$};
     \coordinate (a) at (0.85,0);
\coordinate (b) at (0,0.8);
\coordinate (c) at (-0.85,0);
\coordinate (d) at (0,-0.8);
\coordinate (e) at (0.85,0);               
\path[draw,use Hobby shortcut,closed=true]  (a)..(b)..(c)..(d)..(e); 
\draw [postaction={decorate}] (0,-0.8) node{};
\draw [postaction={decorate}] (0.67,0.069) node{};
\draw [postaction={decorate}] (-0.63,0.069) node{};
 \draw [postaction={decorate}] (-0.98,0.069) node{};         
 \draw [postaction={decorate}] (1.02,0.069) node{};
      \end{scope}
\end{tikzpicture}

	\caption{Contour $\Sigma_R$ for the RH problem for $R$. All jump matrices 
		are close to the identity matrix as $n \to \infty$.
		\label{fig:SigmaR}} 
\end{figure}
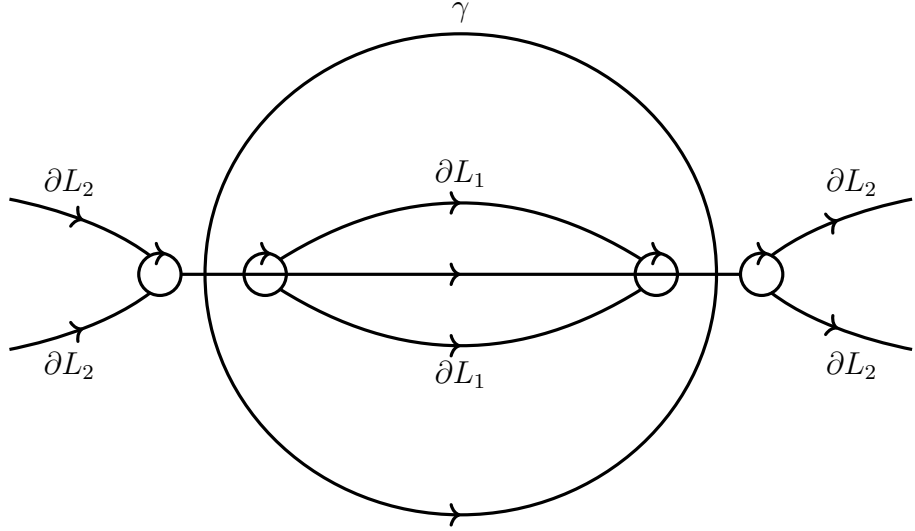

\begin{definition} \label{StoR}
	We define $R : \mathbb C \setminus (\Sigma_S  \cup \partial D) \to \mathbb C^{3\times 3}$
	as follows
\begin{equation} \label{defR}
	R(z)=\begin{cases}
		S(z)M(z)^{-1}, & z \in \left(\mathbb C \setminus \overline{D} \right) 
		\setminus \Sigma_S, \\
		S(z)P(z)^{-1}, & z \in  D \setminus \Sigma_S.
	\end{cases}
\end{equation}
\end{definition}

Since $M$ and $S$ have the same jump on $\Delta_2$, and $P$ and $S$ have
the same jumps inside the disks $D(\pm x_2, \delta)$ and on the 
parts of $\partial L_1$ inside $D(\pm x_1, \delta)$, it is immediate
that $R$ has analytic continuation to $\mathbb C \setminus \Sigma_R$,
where
\[ \Sigma_R = \gamma \cup \left( \partial L_1 \cup \partial L_2 \right) \setminus D
	\cup [-x_2 + \delta, x_2-\delta] \cup \partial D, \]
with orientation as in Figure \ref{fig:SigmaR}. The circles
are oriented clockwise.

\begin{rhproblem} \label{rhproblemR}
	$R$ satisfies the following RH problem
	\begin{itemize}
		\item $R : \mathbb C \setminus \Sigma_R \to \mathbb C^{3 \times 3}$ is analytic.
		\item $R_+ = R_- J_R$ on $\Sigma_R$ with
		\begin{align} \label{Rjump}
			J_R   = \begin{cases}  M J_S M^{-1}  &
				\text{ on } \gamma \cup (\partial L_1 \cup \partial L_2) \setminus D, \\
			M_- J_S M_+^{-1} & \text{ on } (-x_2,x_2) \setminus D, \\ 
			P_- J_S P_+^{-1} & \text{ on } (-x_1-\delta,-x_1+\delta)
				\cup (x_1-\delta,x_1+\delta), \\
			 P M^{-1} & \text{ on } \partial D, \end{cases} 
		\end{align}
		where $J_S$ is given in \eqref{Sjump1}--\eqref{Sjump7}.
		\item $R$ remains bounded near all points of self-intersection of $\Sigma_R$.
		\item $R(z) = I_3 + \mathcal{O}(z^{-1})$ as $z \to \infty$.
	\end{itemize}
\end{rhproblem}

\begin{proof}	
	The jumps of $R$ across the various pieces of $\Sigma_R$ are immediate 
	from the definitions. 
	The asymptotic condition follows from \eqref{Sasymp}, \eqref{Masymp} and the definition \eqref{defR}. \end{proof}

All jump matrices in the RH problem \ref{rhproblemR} for $R$ tend to the identity matrix as $n \to \infty$.
As a result of the steepest descent analysis we then arrive at the following
conclusion, see e.g.\ \cite{DKMVZ99}.
\begin{corollary}\label{approxid} Let $a, c$ be fixed with $a^2\geq 2c$.
	Then, we have
\begin{equation}\label{Rasymp}
    R(z) = I_3 + \mathcal O\left(\frac{1}{n\left(|z|+1\right)}\right) 
    \quad \text{ as } n \to \infty,  \end{equation}
uniformly for $z \in \mathbb C \setminus \Sigma_R$
and uniformly for $t = \frac{n}{N}$ in compact subsets of $(0,t^*)$.
\end{corollary}
\begin{proof}
For each $t \in (0,t^*)$ we have 
\begin{equation} \label{JRest1}
	 J_R = P M^{-1} =  I_3  + \mathcal O(n^{-1}) \text{ as } n \to \infty,
	\text{ uniformly on } \partial D, \end{equation}
because of the matching condition \eqref{Pmatching}.
 
The off-diagonal entry in the jump matrix \eqref{Sjump1} for $S$ on $\gamma$
tends to zero at an exponential rate as $n \to \infty$.
Outside of the disks the off-diagonal entries in the jump matrices 
\eqref{Sjump6} and \eqref{Sjump7} on the lips of the lenses 
also tend to zero at an exponential rate. On the unbounded lens $\partial L_2$
the absolute value of the off-diagonal entry is $e^{2n \Re \phi_2}$.
From the asymptotic behavior \eqref{phi2inf} (valid for $\Im z > 0$) 
and the fact that  $ L_{2}$ contains a sector $|\arg(z-x_{2})|\leq \theta$, 
we conclude $\Re \phi_{2}(z) \leq -c_{2}|z|$ for some $c_{2}>0$ as $z\to \infty$ along $\partial L_{2}^+$. Since $\phi_2(\overline{z}) = \overline{\phi_2(z)}$
we find the same estimate along $\partial L_2^-$. 
Hence, there is a constant $c_0 > 0$ such that
$J_{S}=I_{3}+\mathcal O(e^{-c_0 n|z|})$  as $n \to \infty$ uniformly for $z \in \partial L_2^{\pm} \setminus D$.
Since $M$ and $M^{-1}$ are bounded on $\mathbb C \setminus D$,
we  conclude that for some constant $c_1 > 0$,
\begin{equation} \label{JRest2} J_{R}=I_{3}+\mathcal O(e^{-c_1 n|z|}) \quad \text{ as } n \to \infty,
	\text{ uniformly on } \gamma \cup (\partial L_1 \cup \partial L_2) \setminus D. 
	\end{equation}
	
The remaining part of $\Sigma_R$ is the interval $(-x_2+\delta, x_2-\delta)$
and we claim that for some constant $c_2 > 0$,
\begin{equation} \label{JRest3} 
	J_R = I_3 + \mathcal{O}(e^{-c_2n})	\quad 
	\text{ as } n \to \infty,
	\text{ uniformly on } (-x_2+\delta, x_2-\delta). \end{equation}
Indeed, on $(-x_1, x_1)$ the non-constant entries in \eqref{Sjump2} tend to zero
exponentially fast as $n \to \infty$. Then, since $M$ has the jump
matrix \eqref{Mjump1}
on $(-x_1,x_1)$ and $P$ has the same jump matrix on $(-x_1,-x_1+\delta)$ and
$(x_1-\delta, x_1)$, \eqref{JRest3} uniformly on $(-x_1,x_1)$.
On $(-x_2,-x_1-\delta)$ and $(x_1+\delta,x_2)$ we similarly use
that \eqref{Sjump3} and \eqref{Sjump4} tend to $I_3$ at an exponential rate,
and we obtain \eqref{JRest3} holds true on these intervals as well.
On $(-x_1-\delta,-x_1)$ and $(x_1,x_1+\delta)$ only
the $(3,2)$ entry in \eqref{Sjump3} is exponentially small.
Then we use that the jump matrix $J_P$ in \eqref{JP1} agrees with
\eqref{Sjump3} up to the $(3,2)$ entry to obtain
\[
	J_R  
		= I_3 + (-1)^{n+cN} e^{-2n \phi_2}  
		P_- \begin{pmatrix} 0 & 0 & 0 \\ 0 & 0 & 0 \\ 0 & 1 & 0 \end{pmatrix}
		P_+^{-1} \]
on $(-x_1-\delta,-x_1) \cup (x_1,x_1+\delta)$.
Since $P$ and $P^{-1}$ are uniformly bounded in $D$ and $\Re \phi_2 > 0$
we find \eqref{JRest3} on $(-x_1-\delta,-x_1)$ and $(x_1,x_1+\delta)$ as well
(maybe with a different constant $c_2>0$),
and thus \eqref{JRest3} holds  on the full interval $(-x_2+\delta,x_2-\delta)$ indeed.
 
In conclusion \eqref{JRest1}--\eqref{JRest3} imply that \eqref{Rasymp} 
holds, by standard arguments
on small norm RH problems, see e.g.\ \cite{DKMVZ99}, 

Further inspection shows that the $\mathcal{O}$-terms in 
\eqref{JRest1}--\eqref{JRest3} are also uniform for $t$ in compact
subsets of $(0,t^*)$. Then it also follows by the same arguments 
that \eqref{Rasymp}
is valid uniformly for $t$ in compact subsets of $(0,t^*)$.
\end{proof}

\subsection{Proof of Theorem \ref{thm:PnNasym}}
\begin{proof}
We take $z$ in a compact subset $K$ of $\overline{\mathbb C} \setminus [-x_1,x_1]$
and we  may assume that the lens $L_1$ and the disks
$D(\pm x_1,\delta)$ are such that $K$ is in their exterior.
However, $K$ could intersect the unbounded lens $L_2$ and the disks
around $\pm x_2$.

Following the transformations
\[ Y \mapsto X \mapsto T \mapsto S \mapsto R \]
and recalling $P_{n,N} = Y_{11}$ we find the asymptotic formula as follows.
From transformations $Y \mapsto X \mapsto T$ in Definitions \ref{YtoX} and \ref{XtoT} we obtain
\[  P_{n,N}(z)=Y_{1,1}(z)=X_{1,1}(z) = T_{1,1}(z) e^{ng_1(z)}. \]
 From the transfomation $T \mapsto S$ in Definition  \ref{TtoS} we have 
\[ T_{1,1}(z)=S_{1,1}(z), \]
since $z$ is in the exterior of the lens $L_1$.
Equation \eqref{defR} in Definition \ref{StoR} yields 
\[ S_{1,1} = \begin{cases}
\left(R M \right)_{1,1}, & \text{ on } \mathbb C \setminus \left(\overline{D} \cup \Sigma_S \right) \\
\left(R  P \right)_{1,1}, & \text{ on } \overline{D}\setminus \Sigma_S.
\end{cases} \]
Combining all this with \eqref{Rasymp} we obtain
\begin{equation} \label{PnNasymp1}
    P_{n,N}(z)=\begin{cases}
    M_{1,1}(z) e^{ng_{1}(z)}\left(1+\mathcal{O}\left(n^{-1}\right)\right), &
    z\in \mathbb C \setminus \left(\overline{D}\cup L_1 \cup \Sigma_S \right),\\
    P_{1,1}(z) e^{ng_{1}(z)}
    \left(1+\mathcal{O}\left(n^{-1}\right)\right),
    & z\in \left( \overline{D(x_2,\delta) \cup D(-x_2,\delta)} \right) \setminus \Sigma_S,
    \end{cases}\end{equation}
uniformly on the indicated sets.

From the explicit construction of the local parametrix $P$ in \eqref{Pdef},
we obtain
$P_{1,1}= \left(E_{n} \right)_{1,1}$ on $D(\pm x_{2},\delta)$.
Due to \eqref{Enx2} we find $\left(E_{n}\right)_{1,1}=M_{1,1}$
on $D(x_2,\delta)$, and by symmetry also on $D(-x_{2},\delta)$.
Thus, we may rewrite \eqref{PnNasymp1} as  
\begin{equation} \label{PnNasymp2}
	P_{n,N}(z)  = 
		M_{1,1}(z) e^{ng_{1}(z)}\left(1+\mathcal{O}\left(n^{-1}\right)\right), 
\end{equation}
uniformly for $z \in \mathbb C \setminus \left(\overline{D(x_1,\delta) \cup D(-x_1,\delta)} \cup L_1 \cup \Sigma_S \right)$.
The three functions $P_{n,N}$, $M_{1,1}$ and $e^{ng_1}$ are analytic
on $\mathbb C \setminus [-x_1,x_1]$, and therefore
the formula \eqref{PnNasymp2} extends across 
the part of $\Sigma_S$ that is outside of $\overline{D(x_1,\delta) \cup D(-x_1,\delta)} \cup L_1$,
 with the same error estimate,
and \eqref{PnNasymp2} holds uniformly for
$z \in \mathbb C \setminus \left(\overline{D(x_1,\delta) \cup D(-x_1,\delta)} \cup L_1\right)$, and in particular on $K$.

Since $M_{1,1}(z)$ is given by \eqref{M11z} and $F_1 = F$ is the
conformal map, we obtain \eqref{PnNasym}, uniformly on $K$,
but with $\rho F$ and $g_1$ associated with the
parameter $t_{n,N} = \frac{n}{N}$.
Under the assumptions of Theorem \ref{thm:PnNasym} we have that
$t_{n,N} = t + \mathcal{O}(n^{-1})$ as $n\to \infty$.
Since $\rho F$ and $g_1$ depend on $t$ in a real analytic way,
we also obtain \eqref{PnNasym} for $\rho F$ and $g_1$
associated with $t$, as stated in the theorem. 

This completes the proof of Theorem \ref{thm:PnNasym}.
 \end{proof}
 We find it quite remarkable that although with our transformations 
 we introduced jumps on the unbounded contours $\Delta_2$ and 
 $\partial L_{2}$, these contours as well as the branch points at $\pm x_2$
 do not affect the asymptotic formula \eqref{PnNasym}.
 The measure $\mu_{2}$ plays the role of an auxiliary measure in the 
 asymptotic analysis, and ultimately the polynomial $P_{n,N}$
 itself does not see it.

 \subsection{Proof of Theorem \ref{thm:zeros}}
 \begin{proof}
 The claim is almost immediate from the strong asymptotic
 formula \eqref{PnNasym} Due to \eqref{M1wdef} we have $M_{1,1}(z)\neq 0$ for $z\in \mathbb C\setminus [-x_1,x_1]$. Hence, by \eqref{PnNasym} the
 zeros of $P_{n,N}$ accumulate on $[-x_1,x_1]$ as $n \to \infty$
 as in the theorem.
 
From \eqref{PnNasym} we further conclude  that
\begin{equation}
    \lim_{n\to\infty}\frac{1}{n}\log\left.|P_{n,N}(z)\right|=\Re g_{1}(z) =\int\log|z-s|d\mu_{1}(s)
\end{equation}
uniformly for $z$ in compact subsets of $\mathbb C \setminus [-x_1,x_1]$, 
which is known as weak asymptotic formula for the sequence of polynomials
$(P_{n,N})_n$. It is well known, see e.g. \cite{ST97}, that 
it implies that $\mu_{1}$ is the weak limit of the normalized zero counting 
measures. This completes the proof of 
Theorem \ref{thm:zeros}.
\end{proof}

\section{Conclusion and further work}

So far, our result is only valid in the regime $0<t<t^{*}$,  we plan to investigate the strong asymptotics of the polynomials in other regimes in future work. 
 By item \ref{item6}  of Theorem \ref{thm:mu12} at $t=t^{*}$ the branch point $x_{1}(t^{*})$, $x_{2}(t^{*})$ coincides and a simple application Riemann-Hurwitz formula tells that the density of the measure $\mu_{1}$ decays like a cube root. We expect in this scenario that the local parametrices will be built out of the Pearcey integrals. We refer to \cite{BK07,BH98} for works in this direction.

 For $t^{*}<t<t_{c}$ the branch points move into the complex plane. In this case the support of the limiting zero counting measure is no longer just an interval, but also contains certain contours on the complex plane that enjoy special symmetry properties \cite{MR16, R12}. These contours can also be described in terms of trajectories of certain quadratic differentials as also observed in \cite{BS20, KS15, MS16}.

Regarding the phase transition, we have been able to show that the droplet is simply connected for $t<t_{c}$  and can be determined by the same conformal map $f(w)$ given in \eqref{conformalmap}. Remarkably, the spectral curve degenerates to three linear factors at $t=t_{c}$ and remains degenerate there after.
For $t=t_{c}$, the density of the measure $\mu_{1}$ has a double zero, and the local parametrices are built out of  Hastings-McLeod solution of the Painleve II equation as done in \cite{BBLM15, KLY23}.

Furthermore, we have simplified our model by assuming $a^2\geq 2c$ and symmetric insertion of point charges. We are curious to see what happens when these constraints are removed or multiple point charges are inserted. 

\section*{Acknowledgments}
We extend our  gratitude to Iv\'{a}n Parra for many discussions and insights during the initial phase of the project.

M.K. has been supported
by the Australian Research Council via the Discovery Project grant DP210102887. A.K. is supported by Methusalem grant METH/21/03 – long term structural funding of
the Flemish Government. S.L. acknowledges financial support from the International Research Training
Group (IRTG) between
 KU Leuven and University of Melbourne and  Melbourne Research Scholarship of University of Melbourne.

\end{document}